\documentclass[aop,preprint]{imsart}

\RequirePackage{amsthm,amsmath,amsfonts,amssymb,amstext}
\RequirePackage{mathtools}
\RequirePackage{xcolor, tikz, comment, cite, enumitem, colortbl,rotating}
\RequirePackage[authoryear]{natbib}
\RequirePackage[colorlinks,citecolor=blue,urlcolor=blue]{hyperref}
\RequirePackage{graphicx}
\RequirePackage[utf8]{inputenc}

\startlocaldefs

\theoremstyle{plain} 
\newtheorem{theorem}{Theorem}[section]
\newtheorem{lemma}[theorem]{Lemma}
\newtheorem{proposition}[theorem]{Proposition}
\newtheorem{corollary}[theorem]{Corollary}

\theoremstyle{remark}
\newtheorem{definition}[theorem]{Definition}
\newtheorem{example}[theorem]{Example}
\newtheorem{remark}[theorem]{Remark}

\newcommand{\p}{\mathbb{P}}
\newcommand{\PP}{\mathbb{P}}

\newcommand{\e}{\mathbb{E}}

\newcommand{\D}{\mathrm{d}}

\renewcommand{\a}{{\alpha}}

\newcommand{\eps}{{\varepsilon}}

\renewcommand{\L}{\p}
\newcommand{\norm}[1]{{\|#1\|}}
\newcommand{\br}[2]{\left\langle #1,#2\right\rangle}
\renewcommand{\i}{{\rm i}}
\newcommand{\indepp}{\perp \!\!\! \perp}
\newcommand{\indep}{\perp}

\newcommand{\indepABC}{A \indep B \mid C\;[\Lambda]}
\newcommand{\indepABCwrtD}{A \indep B \mid C\;[\Lambda_D]}
\newcommand{\indepAB}{A \indep B\;[\Lambda]}
\newcommand{\indepABnonpar}{A \indep B\;[\Lambda_{A\cup B}]}
\newcommand{\indepABzero}{A \indep B\;[\Lambda^0_{A\cup B}]}
\newcommand{\cE}{\mathcal{E}}
\newcommand{\cH}{\mathcal{H}}
\newcommand{\cA}{\mathcal{A}}
\newcommand{\cR}{\mathcal{R}}
\newcommand{\cS}{\mathcal{S}}
\newcommand{\cZ}{\mathcal{Z}}
\newcommand{\cC}{\mathcal{C}}
\newcommand{\cL}{\mathcal{L}}
\newcommand{\cB}{\mathcal{B}}

\newcommand{\olambda}{\overline\lambda}
\newcommand{\cG}{\mathcal{G}}

\DeclareMathOperator{\pa}{pa}

\newcommand{\ind}[1]{\mbox{\rm\large 1}_{\{#1\}}}

\newcommand{\R}{\mathbb{R}}
\newcommand{\RR}{\mathbb{R}}

\endlocaldefs

\begin{document}

\begin{frontmatter}

\title{Graphical models for infinite measures\\ with applications to extremes}
\runtitle{Graphical models for infinite measures}

\begin{aug}

\author[A]{\fnms{Sebastian}~\snm{Engelke}\ead[label=e1]{sebastian.engelke@unige.ch}},
\author[B]{\fnms{Jevgenijs}~\snm{Ivanovs}\ead[label=e2]{jevgenijs.ivanovs@math.au.dk}}
\and
\author[C]{\fnms{Kirstin}~\snm{Strokorb}\ead[label=e3]{strokorbk@cardiff.ac.uk}}

\address[A]{Research Center for Statistics, University of Geneva, 1205 Geneva, Switzerland\printead[presep={,\ }]{e1}}
\address[B]{Department of Mathematics, Aarhus University, DK-8000 Aarhus C, Denmark\printead[presep={,\ }]{e2}}
\address[C]{School of Mathematics, Cardiff University, Cardiff CF24 4AG, United Kingdom\printead[presep={,\ }]{e3}}

\end{aug}

\begin{abstract}

 Conditional independence and graphical models are well studied for probability distributions on product spaces. We propose a new notion of conditional independence for any measure $\Lambda$ on the punctured Euclidean space $\mathbb R^d\setminus \{0\}$ that explodes at the origin. The importance of such measures stems from their connection to infinitely divisible and max-infinitely divisible distributions, where they appear as L\'evy measures and exponent measures, respectively. We characterize independence and conditional independence for $\Lambda$ in various ways through kernels and factorization of a modified density, including a Hammersley--Clifford type theorem for undirected graphical models. As opposed to the classical conditional independence, our notion is intimately connected to the support of the measure $\Lambda$. Our general theory unifies and extends recent approaches to graphical modeling in the fields of extreme value analysis and L\'evy processes. Our results for the corresponding undirected and directed graphical models lay the foundation for new statistical methodology in these areas.
 
\end{abstract}


%% Primary
%% 62H22 Probabilistic graphical models
%% 60E07 Infinitely divisible distributions; stable distributions

%% Secondary
%% 60G70 Extreme value theory; extremal stochastic processes
%% 60G51 Processes with independent increments; Lévy processes


\end{frontmatter}

\tableofcontents

\section{Introduction}

Conditional independence is a central concept in probability theory that enables the definition of graphical models. It is a cornerstone of modern statistical methodology for understanding latent structures in the data, construction of parsimonious models in high dimensions and estimation of causal relationships \citep{daw1979, maathuis2018handbook, Lauritzen}.
The classical notion of probabilistic conditional independence is defined for random vectors through factorizations of conditional probabilities. In fact, following an axiomatic approach, conditional independence can be seen as a more general notion of irrelevance with applications beyond random vectors \citep[e.g.,][Chapter 3]{Lauritzen}.

In this paper we define a new notion of conditional independence and graphical models for the class of Borel measures $\Lambda$ on the punctured $d$-dimensional Euclidean space $\cE = \R^d\setminus\{0\}$, where $\Lambda$ is finite on all Borel sets that are bounded away from the origin.
For disjoint subsets $A,B,C$ of the index set $V = \{1,\dots, d\}$, we denote this conditional independence of $A$ and $B$ given $C$ for the measure $\Lambda$ by
\begin{align}\label{CI_intro}
  \indepABC.
\end{align}
The essence of our definition in Section~\ref{sec:def} is that we require classical conditional independence on the normalized restrictions of $\Lambda$ to any charged product form set in $\cE$ that is bounded away from the origin.

%%%% We will now outline (i) in which contexts such measures $\Lambda$ appear naturally, as these areas have motivated this work, (ii) which main results we obtain for our new notion of conditional independence and graphical models, and (iii) which consequences are implied within the respective areas, in particular, how our new results can be interpreted, broaden existing knowledge and open up the route for novel statistical methodology.

Measures $\Lambda$ of the above kind are indeed fundamental in various fields of probability and statistics. Notably, they arise as L\'evy measures of infinitely divisible probability laws with respect to semi-group operations on $\RR^d$ for which the origin is a neutral element \citep{ber1984}. We focus here on the following important examples. A random vector $X$ is (sum-)infinitely divisible or max-infinitely divisible, if for every $n\in\mathbb N$ it has the stochastic representation
\begin{align}\label{id_distr}
  X\stackrel{d}{=}\sum_{i=1}^n X^{(n,i)},\quad \text{or}\quad X\stackrel{d}{=}\bigvee_{i=1}^n X^{(n,i)},\end{align}
respectively, for some independent, identically distributed vectors $X^{(n,1)},\ldots,X^{(n,n)}$, where the sums and maxima are taken in the component-wise sense. In the context of maxima, the term  exponent measure is often preferred over L\'evy measure for $\Lambda$ and we adopt the same convention here \citep{res2008}. Importantly, infinitely divisible distributions appear as limit laws of triangular arrays; for sums this observation dates back to \citet{skorohod57} and \citet[Part XVII.11]{feller2}, for instance, and we refer to \citet{balres77} and \citet{dHres77} for maxima. 
This explains the ubiquity of measures $\Lambda$ and their appearance in the vast literature on related point processes, random sets or large deviation principles \citep[e.g.,][]{dombryribatet2015,hultlind07,dmz2008}. 
In most situations, including typical cases of sum- and max-infinitely divisible distributions, the measure $\Lambda$ explodes at the origin and therefore has infinite mass on $\cE$. Therefore, classical probabilistic conditional independence is not meaningful for such measures. Our new notion in~\eqref{CI_intro} resolves this issue in an arguably natural way.

A first fundamental result for our new conditional independence is that, under a mild explosiveness assumption, the four axioms of a semi-graphoid are satisfied \citep[][Chapter 3]{Lauritzen}. This directly implies various useful properties of the corresponding  graphical models.
Moreover, we provide several alternative characterizations. If $\Lambda$ possesses a density~$\lambda$ with respect to some product measure, then the conditional independence~\eqref{CI_intro} is equivalent to the factorization of (a modified version of) this density. 
Under the mild explosiveness condition,~\eqref{CI_intro} induces certain constraints on the support of the measure $\Lambda$. In particular, the independence $\indepAB$ is completely characterized by the condition that $\Lambda$ can have mass only on certain sub-faces of $\cE$. This is in sharp contrast to usual probabilistic independence, which would suggest a factorization of the measure into the product of marginals. 
Graphical models for measures $\Lambda$ as above can be defined readily by means of suitable Markov properties. We reveal several fundamental links between properties of the measure $\Lambda$ and the graph structure. For instance, for an undirected, decomposable graph we prove a Hammersley--Clifford type theorem that states that the global Markov property for $\Lambda$ is equivalent to a factorization of the modified density of $\Lambda$ on the cliques of the graph. 

The general theory of conditional independence and graphical models for $\Lambda$ measures translates into new insights in the respective areas of probability theory, where such measures arise naturally. 
In particular, many of our findings in this manuscript have been motivated by recent developments and open questions from multivariate extreme value theory, which studies the dependence structure in the distributional tail of a random vector. The most fundamental limiting models that appear in this field are max-stable distributions, which are max-infinitely divisible distributions with homogeneous exponent measures $\Lambda$, that is,
\begin{align}\label{def:homog}
\Lambda(t A)=t^{-\a}\Lambda(A),\qquad t>0,
\end{align}
for some $\alpha>0$ and all Borel sets $A\subset \cE$ bounded away from the origin.
While classical theory in this field intensively studied stochastic processes with spatial domains \citep[e.g.,][]{deh1984, kab2009}, recent research concentrates on finding sparse structures in multivariate extreme value models; see~\citet{eng2021} for an overview. One line of work develops recursive max-linear models on directed acyclic graphs, which correspond to a specific class of spectrally discrete max-stable distributions~\citep{gis2018, klu2019}. However, if a max-stable distribution admits a positive density, it cannot satisfy non-trivial (classical) conditional independence \citep{papastathopoulos2016conditional}. Instead,  \citet{eng2018} define a notion of extremal conditional independence and undirected graphical models on the level of multivariate Pareto distributions \citep{roo2006}, which can also be characterized by the homogeneous exponent measure $\Lambda$; see Section~\ref{sec:CIextremes} for details.
In the case of tree graphs, these models arise as extremal limits of regularly varying Markov trees \citep{seg2019}.

These two active research directions have so far co-existed without a clear connection.
Since max-linear models and multivariate Pareto distributions are both described by
the exponent measure $\Lambda$, it seems natural that this object also encodes the graphical properties. Indeed, it turns out that our conditional independence notion in~\eqref{CI_intro} unifies and extends both approaches. We show that for recursive max-linear models, any classical conditional independence implied by its directed acyclic graph is also present on the level of $\Lambda$ in the sense of~\eqref{CI_intro}.
Concerning the extremal conditional independence in~\citet{eng2018}, we note that it is originally only defined if $\Lambda$ has a Lebesgue density on $\cE$ and no mass on sub-faces. As a consequence, no extremal independence can exist and their extremal graphical models are always connected. Several discussion contributions\footnote{More precisely, we refer to the contributions of J.~Wadsworth, I.~Papastathopoulos, K.~Strokorb, C.~Darne and A.~C.~Davison, L.~Mhalla, P.~Wan, Y.~Zhang and L.~Wang in the discussion part of \citet{eng2018} (here listed in order of appearance therein).} of their paper raised these points as limitations for statistical modeling.
Our conditional independence~\eqref{CI_intro} extends the notion of~\citet{eng2018} to arbitrary exponent measures $\Lambda$ and thereby overcomes several of these restrictions.
\begin{enumerate}[label=(\roman*)]
\item
  It allows for disconnected extremal graphs with mass on sub-faces and, as an illustration, we provide an extension of H\"usler--Reiss tree models \citep{eng2020,ase2021} to forests with richer dependence structures. Since no densities are required, our notion also includes max-linear models. 
\item
  It extends to max-infinitely divisible distributions that can exhibit asymptotic independence \citep{hus2021}, and therefore opens the door to graphical modeling in this large field of research, in which many approaches go back to the conditional approach for extreme value modelling by \citet{HeffernanTawn2004}.
\item
  The support constraints from the conditional independence of $\Lambda$ connect our theory to sparsity notions for tails of multivariate distributions and the field of concomitant extremes \citep[e.g.,][]{chi2017}.  
\end{enumerate}
We discuss these extensions in Section~\ref{sec:extremes} and provide some basic examples that motivate the potential for the construction of new models. A more in-depth study of these new modelling routes and the
development of corresponding statistical methodology is beyond the scope of this paper.

Another promising application of our theory lies in the field of infinitely divisible distributions and L\'evy processes \citep{iva2020}, where $\Lambda$ plays the role of the  L\'evy measure. Much less literature on graphical models and sparsity exists in this area. In the case of $\alpha$-stable distributions, that is, if the L\'evy measure is homogeneous as in~\eqref{def:homog}, \citet{stableModels} introduce recursive stable models that are similar to \citet{gis2018}. Again, the classical conditional independence is reflected in the corresponding (degenerate) L\'evy measure $\Lambda$ through our conditional independence.
For more general L\'evy measures, we can define graphical models with respect to our $\Lambda$-based notion of conditional independence. We provide an outlook on how this translates into probabilistic properties of the corresponding L\'evy processes in the concluding Section~\ref{sec:outlook}; the details of this direction are part of future research.

\section{Preliminaries}

\subsection{Conditional independence}\label{pre_CI}
Consider a probability space $(\Omega,\cA,\p)$ and let $X,Y,Z$ be random elements with values in some Borel spaces.
\emph{Conditional independence} of $X$ and $Y$ given~$Z$ with respect to $\p$, denoted by \[X \indepp Y \mid Z,\] 
is defined as the factorization of conditional probabilities
\begin{align*}
\p(X\in E_X,Y\in E_Y|Z)=\p(X\in E_X|Z)\p(Y\in E_Y|Z),\qquad \text{a.s.},
\end{align*}
for all $E_X,E_Y$ in the respective $\sigma$-algebras; we refer to~\citet[Chapter\ 6]{kallenberg} for the basic theory and further properties.
In terms of the regular conditional probability $\p(\cdot|Z=z)$ it reads
\[\p(X\in E_X,Y\in E_Y|Z=z)=\p(X\in E_X|Z=z)\p(Y\in E_Y|Z=z),\quad \text{ for }\p_Z\text{-almost all }z, \]
where $\p_Z$ denotes the law of~$Z$. By a standard argument \citep[see][Thm.\ 6.3]{kallenberg}, this factorization holds for all such $E_X,E_Y$ and all $z$ outside some fixed $\p_Z$-null set.
In other words, $X$ and $Y$ are independent under $\p(\cdot|Z=z)$ for $\p_Z$-almost all~$z$.

When $(X,Y,Z)$ possesses a joint density $f(x,y,z)$ with respect to some product measure $\mu=\mu_X\otimes\mu_Y\otimes\mu_Z$, where each component is a $\sigma$-finite measure, the conditional independence $X \indepp Y \mid Z$ is equivalent to the following density decomposition:
\[f(x,y,z)f_Z(z)=f_{XZ}(x,z)f_{YZ}(y,z) \qquad \text{ for }\mu\text{-almost all }(x,y,z).\]
Here, we used the naturally induced notation for the marginal densities; see for instance~\citet{daw1979} or \citet[Chapter 3]{Lauritzen}.
In particular, if $(X,Y,Z)$ is discrete, then its probability mass function factorizes for all $(x,y,z)$.
If $\mu$ is Lebesgue and all the involved densities are continuous, then this equality must be true for all $(x,y,z)$.
Note that continuity of $f$ does not imply continuity of the marginal densities, which explains the formulation of the last sentence.

\subsection{Infinitely divisible distributions for sums and maxima}\label{sec:motivation}
For $d\geq 1$ consider the domain $\cE = \R^d\setminus\{0\}$, which is a  $d$-dimensional space of real numbers punctured at the origin. Let $\Lambda$ be a Borel measure on this domain, such that the following basic boundedness condition is satisfied:
\begin{align}\label{eq:Lambda}
\tag{B}
\Lambda(A)<\infty\qquad \text{for any Borel set }A\subset \cE \text{ bounded away from the origin}.
\end{align}
As discussed in the introduction, such measures appear naturally in various areas of probability and statistics. In particular, they arise in the limits of sums and maxima of random variables in the theory of infinitely divisible and max-infinitely divisible distributions as in~\eqref{id_distr}.

In the max-infinitely divisible case we assume that each component $X_j$ is non-negative with lower endpoint $0$, that is, each $X_j$ has some mass arbitrarily close to 0: 
\begin{align}\label{eq:support}
 \quad\inf\{x:\p(X_j\leq x)>0\} = 0.
\end{align}
Then the distribution function of $X$ is given by
\begin{align}\label{eq:exp_measure}
\p(X\leq x)=\exp\big[-\Lambda(\cE_+ \setminus[0,x])\big],\qquad x\geq 0,
\end{align}
where $\Lambda$ is null outside of the domain $\cE_+ =[0,\infty)^d\setminus\{0\}$
and satisfies~\eqref{eq:Lambda}, while $x \geq 0$ is meant componentwise; see \citet[][Prop.\ 5.8]{res2008}. 
Conversely, every such measure $\Lambda$, called an \emph{exponent measure}, leads to a max-infinitely divisible distribution.
The standard assumption $\Lambda(\cE_+)=\infty$ corresponds to no point mass of $X$ at the origin.
One may also define a Poisson point process with intensity measure $\D t\otimes \Lambda(\D x)$ and retrieve $X$ as the maximum over the space-points with $t\leq 1$.

In the infinitely divisible case the characteristic function of $X$ is given by
\[\e e^{\i\br{\theta}{X}}=\exp\bigg[-\frac{1}{2}\br{\theta}{\Sigma\theta}+\i\br{\gamma}{\theta}+\int_{\cE}\big(e^{\i\br{\theta}{x}}-1-\i\br{\theta}{x}\ind{\norm{x}\leq 1} \big)\Lambda(\D x) \bigg],\,\, \theta\in \R^d,\]
where $\gamma\in \R^d$ is a drift parameter, $\Sigma$ is the symmetric, non-negative definite $d\times d$ covariance matrix of the Gaussian component, and $\Lambda$ is the so-called \emph{L\'evy measure}, which satisfies~\eqref{eq:Lambda} together with
$\int_{\norm{x}\in (0,1)}\norm{x}^2\Lambda(\D x)<\infty$, where $\norm{\cdot}$ is an arbitrary norm on $\R^d$; see~\citet[Thm.\ 8.1]{sato}. The L\'evy measure is used to define a Poisson point process of jumps with intensity measure $\D t\otimes \Lambda(\D x)$, the compensated sum of which  yields the non-Gaussian part of~$X$ in the limit sense.
Conversely, every such characteristic triplet $(\gamma,\Sigma,\Lambda)$ leads to an infinitely divisible law.

\begin{example}\label{ex:gauss}
  Let $\Phi_\rho:\mathbb R^2 \to [0,1]$ be the distribution function of a bivariate normal distribution with correlation coefficient $\rho \in [0,1]$ and standard normal margins. When $\rho > 0$, the bivariate normal distribution is max-infinitely divisible \citep[][Section 5.2]{res2008}, a property inherited by the corresponding log-normal distribution $\Phi_\rho \circ \log$. Thus, by~\eqref{eq:exp_measure}
  \[\Lambda_\rho(\cE_+ \setminus[0,x]) = -\log \Phi_\rho( \log x),\quad  x \geq 0,\]
  is a valid measure satisfying~\eqref{eq:Lambda}. In addition, it explodes at the origin. 

\end{example}

Various connections between the infinitely divisible and max-infinitely divisible laws have been explored in~\citet{kabluchko_summax} and \citet{stoev_maxsum}. \citet{dmz2008} and \citet{mo2009} draw further attention to multivariate $\ell_p$-infinitely divisible laws, or more generally, infinitely divisible laws with respect to semi-group operations. All of these works focus on the important special case where the involved {L\'evy measures} (or {exponent measures}) $\Lambda$ are \emph{$-\a$-homogeneous}, that is, the property~\eqref{def:homog} holds for some $\alpha>0$.
% , that is,
% \begin{align}\label{def:homog}
% \Lambda(t A)=t^{-\a}\Lambda(A),\qquad t>0,
% \end{align}
% for some $\alpha>0$ (we say \emph{$-\a$-homogeneous}) and all Borel sets $A$ bounded away from the origin.
A max-infinitely divisible $X$ with (standard) $\a$-Fr\'echet marginals is \emph{max-stable} if and only if $\Lambda$ is $-\a$-homogeneous \citep[Prop.\ 5.11]{res2008}. An infinitely divisible $X$ is \emph{$\a$-stable} with $\a\in (0,2)$ if and only if $\Sigma=0$ and $\Lambda$ is $-\a$-homogeneous \citep[Thm.\ 14.3(ii)]{sato}.

\begin{example}\label{ex:HR}
  The exponent measure corresponding to the bivariate max-stable H\"usler--Reiss distributions \citep{Husler1989} with parameter $\Gamma > 0$ is
  \[\Lambda_\Gamma(\cE_+ \setminus[0,x]) = - \frac{1}{x_1}\Phi\left( \frac{\sqrt{\Gamma}}{2} + \frac{\log(x_2/x_1)}{\sqrt{\Gamma}} \right) - \frac{1}{x_2}\Phi\left( \frac{\sqrt{\Gamma}}{2} + \frac{\log(x_1/x_2)}{\sqrt{\Gamma}} \right) ,\quad  x \geq 0,\]
  where $\Phi$ denotes the univariate standard normal cumulative distribution function. This measure also satisfies~\eqref{eq:Lambda}. Since it is $-1$-homogeneous, it also explodes at the origin. 

\end{example}

\newpage
\section{Definition of conditional independence for $\Lambda$ and first properties}\label{sec:theory}

\subsection{Setup and notation}
As previously, let $V=\{1,\ldots,d\}$ be a finite index set and $\Lambda$ a measure on $\cE = \R^V\setminus\{0_V\}$ satisfying~\eqref{eq:Lambda}. 
For a subset $A \subset \cE$, we say that a property holds for $\Lambda$-almost all $y \in A$ if the subset of $A$ where it does not hold is a $\Lambda$-null set. 
As test sets for conditional independence we consider the class $\cR(\Lambda)$ of charged product-form sets in $\cE$ not containing the origin $0_V$ in their closure
\begin{align}\label{eq:Rdef}
\cR(\Lambda)=\bigg\{R = \bigtimes_{v\in V}R_v \, : \, R_v\in  \cB(\R), \, \Lambda(R)>0, \, 0_V\notin \overline R\bigg\}.
\end{align}
By definition, we have $\Lambda(R)\in (0,\infty)$ for any $R \in \cR(\Lambda)$, which allows us to define a probability measure $\L_{R}$ on $R$ via
  \begin{align}\label{eq:probabR}
    \L_{R}(\D y)  = \frac{\Lambda(\D y)}{\Lambda(R)}.
    \end{align}
  We write $Y\sim \L_R$ for a $d$-dimensional random vector $Y$ distributed according to~$\L_R$.

For a non-empty subset $D$ of $V$, the measure $\Lambda$ induces two natural measures on the lower dimensional domain 
\begin{align} \label{eq:lower-dim-domain}
\cE^D = \R^{D}\setminus\{0_D\},
\end{align}
namely the marginal measure $\Lambda_D$ and the restricted measure $\Lambda_D^0$ as given by
\begin{align}\label{eq:Lambda0}
&\Lambda_D(\cdot)=\Lambda(y_D\in \cdot ), &\Lambda_{D}^0(\cdot)=\Lambda(y_{D}\in \cdot\, , \, y_{V\setminus  D}=0_{V\setminus  D}),
\end{align}
respectively.
Both measures also satisfy the basic assumption~\eqref{eq:Lambda} in their common domain~$\cE^D$. 
%%We note that the full measure $\Lambda$ can be expressed as 
%%\[\Lambda=\Lambda_V=\Lambda^0_V\] and state the following marginal %%compatibilities for convenience.

The class $\cR(\Lambda_D)$ is understood to contain product form sets $R_D=\bigtimes_{d\in D}R_d$ (satisfying the other conditions as in \eqref{eq:Rdef} with respect to $\RR^D$ instead of $\RR^V$) and the respective probability measures are denoted by~$\p_{R_D}$. The product set $R$ will sometimes be written as $R_D\times R_{V\setminus D}$ irrespective of the indices in~$D$. Note that $R_D\in \mathcal R(\Lambda_D)$ implies $R_D \times \R^{V \setminus D} \in \cR(\Lambda)$. % for any rectangular set $R_{V \setminus D}$ (including $$).
Some basic consistency properties are stated in Appendix~\ref{app:basic}.

\subsection{Definition of conditional independence for $\Lambda$}\label{sec:def}
  We are now ready to give our definition of conditional independence for the measure~$\Lambda$ based on the above test sets.

\begin{definition}   \label{def:CILambda}
  For disjoint sets $A, B, C\subset V$ that form a partition of $V$, we say that $\Lambda$ admits \emph{conditional independence} of $A$ and $B$ given $C$, denoted by
  \[\indepABC,\]
  if we have the classical conditional independence
  \begin{align}\label{def_part}
  Y_A \indepp Y_B \mid Y_C \qquad \text{for } \quad Y\sim \L_{R} \quad \text{ for all } \quad R\in \cR(\Lambda).
  \end{align}
  This is trivially true for $A$ or $B$ being empty, and for $C=\emptyset$ we say that $\Lambda$ admits \emph{independence} of $A$ and $B$, and write \[\indepAB.\]
  If the sets $A, B$ and $C$ are not a partition of $V$, then the above definition remains the same with
  the test class in~\eqref{def_part} replaced by $ \cR(\Lambda_{A \cup B\cup C})$.
\end{definition}

\begin{remark} \label{rk:innerTransformations}
If $\Lambda$ is a zero measure, the test classes $\cR(\Lambda_{A \cup B \cup C})$  are empty, hence any such (conditional) independence statement with respect to $\Lambda$ is true.
Moreover, for strictly monotone, continuous transformations $T_v: \R \to \R$ with $T_v(0) = 0$, $v \in V$, the conditional independence of $\Lambda$ in Definition~\ref{def:CILambda} is equivalent to the corresponding statement for the pushforward measure $\Lambda \circ T^{-1}$ for $T(x) = (T_v(x_v))_{v\in V}$.  Definition~\ref{def:CILambda} is similar to the notion of conditional inner independence as suggested by~\citet{linbo} for a random vector supported by a non-product-form space.
\end{remark}

It is useful to observe that for $A \cup B \cup C \subset D \subset V$, the two statements
\begin{align}\label{subset_equivalenece}
  \indepABC \qquad \text{and} \qquad \indepABCwrtD
\end{align}
are equivalent. However, some caution is needed. We would like to stress that the conditional independence $\indepABC$ does not(!) imply
\[Y_A \indepp Y_B \mid Y_C \qquad \text{for } \quad Y\sim \L_R \quad \text{ for all } \quad R\in \cR(\Lambda)\]
unless $(A,B,C)$ is a partition of $V$, because restriction according to the additional dimensions in $V\setminus (A \cup B \cup C)$ amounts to conditioning on the respective components; see Lemma~\ref{lemma:marginalcompatibility} in Appendix~\ref{app:basic}.

A second simple observation, which will accompany us throughout this work, is the following:
\begin{lemma}\label{lem:Lambda0}
$\indepABC$ implies $A \indep B  \, [\Lambda^0_{A\cup B}]$.
\end{lemma}

%%%%\kir{changed the following proof sligthly, because I did not find it understandable at all otherwise, in particular the second sentence requires some reflection, I think. It needs a version of Lemma \ref{lemma:marginalcompatibility} with $\RR^{V\setminus D}$ replaced by $\{0_{V\setminus D}\}$.}

\begin{proof}
If $\Lambda^0_{A\cup B}$ is null, there is nothing to prove. Else
consider $R_{A\cup B} \in \cR(\Lambda^0_{A\cup B})$ and the respective $Y_{A\cup B}$.
Set $Y_C=0_C$ and note that the law of such a vector $Y$ alternatively arises from 
the measure $\Lambda$ and the set $R_{A\cup B}\times \{0_C\}\in \cR(\Lambda)$.
Now $\indepABC$ implies $Y_A \indepp Y_B \mid Y_C$ and thus also $Y_A \indepp Y_B$, because $Y_C=0_C$ a.s.
\end{proof} 

\begin{remark}\label{rk:CIreduced}
Slightly more general, the same reasoning shows that $\indepABC$ 
implies $A \indep B \,|\, C' \, [\Lambda^0_{A\cup B \cup C'}]$ for any subset $C' \subset C$. When $C=\emptyset$, the statement in Lemma~\ref{lem:Lambda0} is a tautology.
\end{remark}

\subsection{Semi-graphoid properties}\label{sec:semi-graphoid}
A semi-graphoid is an abstract independence model that satisfies a set of algebraic properties \citep{pearl1988probabilistic, Lauritzen, maathuis2018handbook}. For our conditional independence defined in Section~\ref{sec:def} for a measure $\Lambda$, if $A,B,C,D$ are disjoint subsets of~$V$, these properties read as follows.
\begin{itemize}
  \itemsep2mm
\item[\textbf{(L1)}]  If $A \indep B \,|\, C \, [\Lambda]$, then $B \indep A \,|\, C \, [\Lambda]$. \hfill (Symmetry)
\item[\textbf{(L2)}]  If $A \indep (B \cup D) \,|\, C \, [\Lambda]$,  then $A \indep B \,|\, C \, [\Lambda]$. \hfill (Decomposition)
\item[\textbf{(L3)}]  If $A \indep (B\cup D) \,|\, C \, [\Lambda]$, then $A \indep B \,|\, C \cup D \, [\Lambda]$. \hfill (Weak union)
\item[\textbf{(L4)}]  If $A \indep B \,|\, C \, [\Lambda]$ and $A \indep D \,|\, B \cup C \, [\Lambda]$, then $A \indep (B \cup D) \,|\, C \, [\Lambda]$. \hfill (Contraction)
\end{itemize}

These conditions are crucial in the study of graphical models as they directly imply certain equivalences or implications among fundamental Markov properties usually considered in directed and undirected graphical models; see~Sections~\ref{sec:undirected} and~\ref{sec:directed}. 

If we restrict our attention to homogeneous measures $\Lambda$ that admit a Lebesgue density on $\mathcal E_+$ as considered in~\citet{eng2018} in the context of extreme value theory, Steffen Lauritzen argues already in his discussion contribution that their extremal conditional independence is a semi-graphoid. We will confirm this finding below (Section~\ref{sec:semi-graphoid}) and even show that homogeneous measures  $\Lambda$ with a positive continuous Lebesgue-density  on $\mathcal E_+$ (and independent concatenations thereof) give rise to a \emph{graphoid} (Appendix~\ref{app:graphoids}). 

However, our general setting here goes far beyond this situation in several ways. Our $\Lambda$ is not necessarily homogeneous, it does not necessarily have a density, its mass is not restricted to the positive upper orthant and it may even have mass on sub-faces  of $\mathcal E$. 
This leads to a situation, where only the properties (L1)--(L3) always hold. The contraction condition (L4) is more involved. 

\begin{proposition} \label{prop:LauritzenEasyPart}
  Conditional independence with respect to $\Lambda$ as defined in Definition~\ref{def:CILambda} satisfies (L1), (L2) and (L3). 
\end{proposition}
\begin{proof}
  Properties (L1) and (L3) follow directly from their respective counterparts for probability laws, whereas (L2) follows similarly if one takes into account the compatibility with marginal measures as stated in Lemma~\ref{lemma:marginalcompatibility}.
  \end{proof}
It will be shown in Theorem~\ref{thm:L4} below that also property (L4) is true under the explosiveness assumption~\eqref{eq:infinite}. This property is not true, however, for general finite measures $\Lambda$, in which case our Definition~\ref{def:CILambda} might  not be natural. 
The problem arises from the exclusion of the origin in the domain $\cE=\RR^V \setminus \{0_V\}$, which is neither necessary nor natural in the finite case. The following simple example illustrates this.

\begin{example}[Violation of (L4) for a finite measure]
\label{ex:violation}
Consider $A=\{1\}, B=\{2\}, C=\emptyset, D=\{3\}$ and $\Lambda$ putting mass 1 at each of the points $(0,0,1)$ and $(1,1,1)$; see Figure~\ref{fig:ViolationL4} for an illustration. Observe that $\Lambda_{A\cup B}$ is a point mass at $(1,1)$ and so we trivially have $A \indep B \; [\Lambda]$. Furthermore, we also have $A \indep D \mid B\; [\Lambda]$, because the two point masses have a different $B$ coordinate and so any admissible $\L_R$ results in the corresponding conditional independence statement.
Finally, the statement $A \indep (B\cup D) \; [\Lambda]$ is not true, because there exists a product set, namely $R=\R\times\R\times \{1\}$, for which we get that $Y\sim\L_R$ can only assume the values $(0,0,1)$ or $(1,1,1)$, each with probability $1/2$, and thus, $Y_1$ and $(Y_2,Y_3)$ are not independent. 
\end{example}

\begin{figure}[t!]
\mbox{\small   
    \begin{tikzpicture}[scale = 1.5, >=stealth]
              %---------------------------------
      \coordinate (O) at (0,0,0);
      \coordinate (I) at (1,1,1);
      \coordinate (e1+) at (0,0,1);
      \coordinate (e2+) at (1,0,0);
      \coordinate (e3+) at (0,1,0);
      \coordinate (z1+) at (0,0,1.5);
      \coordinate (z2+) at (1.5,0,0);
      \coordinate (z3+) at (0,1.5,0);
      \coordinate (e1-) at (0,0,-1);
      \coordinate (e2-) at (-1,0,0);
      \coordinate (e3-) at (0,-1,0);
            %---------------------------------
            \coordinate (h3-) at (0,-0.7,0);
      %---------------------------------
      \coordinate (e1+2+) at (1,0,1);
      \coordinate (e1-2-) at (-1,0,-1);
      \coordinate (e1+2-) at (-1,0,1);
      \coordinate (e1-2+) at (1,0,-1);
      %---------------------------------
      \coordinate (e2+3+) at (1,1,0);
      \coordinate (e2-3-) at (-1,-1,0);
      \coordinate (e2-3+) at (-1,1,0);
      \coordinate (e2+3-) at (1,-1,0);
      %---------------------------------
          \coordinate (e1+3+) at (0,1,1);
      \coordinate (e1-3-) at (0,-1,-1);
      \coordinate (e1-3+) at (0,1,-1);
      \coordinate (e1+3-) at (0,-1,1);
           %---------------------------------
       \coordinate (T) at (1.4,1.4,0);
      %---------------------------------
      \draw[thick, black, ->] (e1-) to (z1+) node[anchor=east]{$y_{1}$};
      \draw[thick, black, ->] (e2-) to (z2+) node[anchor=west]{$y_{2}$};
      \draw[thick, black, ->] (h3-) to (z3+) node[anchor=west]{$y_{3}$} ;
      %---------------------------------
         \draw[dashed, black] (I) to (e1+2+);
         \draw[dashed, black] (I) to (e2+3+);
         \draw[dashed, black] (I) to (e1+3+);
         \draw[dashed, black] (e1+) to (e1+2+);
         \draw[dashed, black] (e2+) to (e1+2+);
         \draw[dashed, black] (e2+) to (e2+3+);
         \draw[dashed, black] (e3+) to (e2+3+);
         \draw[dashed, black] (e1+) to (e1+3+);
         \draw[dashed, black] (e3+) to (e1+3+);
      \filldraw(O) circle (1pt) node[anchor=south east] {$0$};
      \filldraw[blue](e3+) circle (2pt) node[anchor=south east] {$(0,0,1)$};
      \filldraw[red](I) circle (2pt) node[anchor=north west] {$(1,1,1)$};
\node at (T) {$\Lambda$};
      %---------------------------------
            %---------------------------------
      \coordinate (Otwo) at (3.2,0,0);
      \coordinate (twoe1+) at (3.2,0,1);
      \coordinate (twoe2+) at (4.2,0,0);
      \coordinate (twoz1+) at (3.2,0,1.5);
      \coordinate (twoz2+) at (4.7,0,0);
      \coordinate (twoe1-) at (3.2,0,-1);
      \coordinate (twoe2-) at (2.2,0,0);
            %---------------------------------
      \coordinate (twoe1+2+) at (4.2,0,1);
      %---------------------------------
            \coordinate (Ttwo) at (4.6,0.4,0);
         %---------------------------------
      \draw[thick, black, ->] (twoe1-) to (twoz1+) node[anchor=east]{$y_{1}$};
      \draw[thick, black, ->] (twoe2-) to (twoz2+) node[anchor=west]{$y_{2}$};
      %---------------------------------
         \draw[dashed, black] (twoe1+) to (twoe1+2+);
         \draw[dashed, black] (twoe2+) to (twoe1+2+);
      \filldraw(Otwo) circle (1pt) node[anchor=south east] {$0_{12}$};
      \filldraw[red](twoe1+2+) circle (2pt) node[anchor=north west] {$(1,1)$};
      \node at (Ttwo) {\color{red} $\Lambda_{12}$};
      %---------------------------------
            %---------------------------------
      \coordinate (Othree) at (-3.2,1,0);
      \coordinate (threee1+) at (-3.2,1,1);
      \coordinate (threee2+) at (-2.2,1,0);
      \coordinate (threez1+) at (-3.2,1,1.5);
      \coordinate (threez2+) at (-1.7,1,0);
      \coordinate (threee1-) at (-3.2,1,-1);
      \coordinate (threee2-) at (-4.2,1,0);
            %---------------------------------
      \coordinate (threee1+2+) at (-2.2,1,1);
      %---------------------------------
            \coordinate (Tthree) at (-1.8,1.4,0);
         %---------------------------------
      \draw[thick, black, ->] (threee1-) to (threez1+) node[anchor=east]{$y_{1}$};
      \draw[thick, black, ->] (threee2-) to (threez2+) node[anchor=west]{$y_{2}$};
      %---------------------------------
         \draw[dashed, black] (threee1+) to (threee1+2+);
         \draw[dashed, black] (threee2+) to (threee1+2+);
      \filldraw[blue](Othree) circle (2pt) node[anchor=south east] {$(0,0,1)$};
      \filldraw[red](threee1+2+) circle (2pt) node[anchor=north west] {$(1,1,1)$};
      \node at (Tthree) {$\Lambda|_{\RR \times \RR \times \{1\}}$};
        \end{tikzpicture}
        }

         \caption{\small  
Illustration of Example~\ref{ex:violation}. The measure $\Lambda$ consists of two point masses with different $y_2$-coordinate, so that  $\{1\} \indep \{3\} \mid \{2\}~[\Lambda]$ (center). Its marginal measure $\Lambda_{12}$ is a point mass (right), hence $\{1\} \indep \{2\}~[\Lambda]$. However, the restriction of $\Lambda$ to the admissible product set $\RR \times \RR \times \{1\}$ does not factorize (left); the L4 conclusion $\{1\} \indep \{2,3\}~[\Lambda]$ is violated.}
  \label{fig:ViolationL4}
\end{figure}
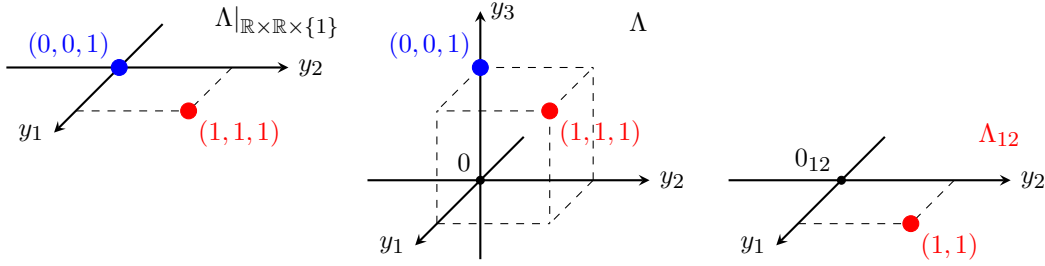

Even stronger properties than those of a semi-graphoid are represented by the notion of a graphoid \citep{Lauritzen}. 
The $\Lambda$-induced conditional independence defines a \emph{graphoid}
\emph{graphoid}, when not only (L1)--(L4) are satisfied, but additionally the axiom  
\begin{itemize}
  \itemsep2mm
\item[\textbf{(L5)}]  If $A \indep B \,|\, C \cup D \, [\Lambda]$ and  $A \indep C \,|\, B \cup D \, [\Lambda]$, then  $A \indep B \cup C \,|\, D \, [\Lambda]$ \hfill (Intersection)
\end{itemize}
In Appendix~\ref{app:graphoids}, we document situations when the graphoid axioms are satisfied, and both trivial and subtle cases under which (L1)--(L4) hold, but (L5) is violated.

% While we demonstrate below (Theorem~\ref{thm:L4}) under which conditions (L4) is valid, Appendix~\ref{app:graphoids} reveals further situations, under which the $\Lambda$-induced conditional independence defines a \emph{graphoid}, and not only a semi-graphoid, i.e., situations, when not only (L1)--(L4) are satisfied, but also: 
% \begin{itemize}
%   \itemsep2mm
% \item[\textbf{(L5)}]  If $A \indep B \,|\, C \cup D \, [\Lambda]$ and  $A \indep C \,|\, B \cup D \, [\Lambda]$, then  $A \indep B \cup C \,|\, D \, [\Lambda]$ \hfill (Intersection)
% \end{itemize}
%  We also document both trivial and subtle situations, under which (L1)--(L4) hold, but (L5) is violated in Appendix~\ref{app:graphoids}.

\section{Alternative characterizations of $\Lambda$-conditional independence} \label{sec:alternative}
We establish three alternative characterizations of the conditional independence  $\indepABC$ in this section, each of which further demonstrates that our notion is both natural and intuitive. The first characterization, Theorem~\ref{thm:test_class}, is a reduction to test sets of a simple form. The second, Theorem~\ref{thm:kernel}, is in terms of the underlying probability (Markov) kernel $\nu(y_C,\cdot)$ from $\cE^C$ to $\R^{A\cup B}$.
It allows to view our notion of conditional independence as the classical independence for a family of probability laws. 
The third, Theorem~\ref{thm:density_factorization}, is in terms of the factorization of the density of $\Lambda$ with respect to some dominating measure, assuming it exists. 
%%%%Additionally, one always needs to verify $\indepABzero$ as stated in Lemma~\ref{lem:Lambda0} and so we refer Scetion~\ref{sec:independence} treating independence relation.

Without loss of generality, we focus on the case, where the sets $A$, $B$ and $C$ form a partition of $V$. The independence case $C=\emptyset$ is excluded, unless mentioned otherwise, but we return to it in Section~\ref{sec:independence}.  The more technical  proofs and auxiliary results are given in Appendix~\ref{sec:CIproofs}.

\subsection{Reduction of the test class}

Our first result reduces the test class $\cR=\cR(\Lambda)$ to sets of the form
\begin{align}\label{eq:Rvep}
R_{v,\eps}=\{y\in \R^d:|y_v|\geq \eps\}= (\RR \setminus (-\eps,\eps))^{\{v\}} \times \RR^{V \setminus \{v\}},\qquad v\in V,\,\eps>0.
\end{align}
When writing $Y\sim \L_{R_{v,\eps}}$, we implicitly assume $\Lambda(R_{v,\eps})>0$ and hence $R_{v,\eps}\in\cR(\Lambda)$. 
The proof is postponed to Appendix~\ref{sec:CIproofs}.

\begin{theorem}[CI via test classes]\label{thm:test_class}
Let $(A,B,C)$ be a partition of $V$.\\ Then   $\indepABC$ is equivalent to any of the following statements:
\begin{enumerate}[label=(\roman*)]
\item $Y_A \indepp Y_B\mid Y_C$ with $Y\sim \L_{R_{v,\eps}}$ for all $v\in V,\,\eps>0$,
\item $Y_A \indepp Y_B\mid Y_C$ with $Y\sim \L_{R_{c,\eps}}$ for all $c\in C,\,\eps>0$, and additionally, $A \indep B  \; [\Lambda^0_{A\cup B}]$.
%%%%%\item[3)] for some $c\in C$ it holds that $Y_A \indepp Y_B\, |\, Y_C$  for all $Y\sim \L_{R_{c,\eps}},\,\eps>0,\Lambda(R_{c,\eps})>0$, and additionally $A \indep B | C\setminus\{c\} \quad [\Lambda^0_{V\setminus \{c\}}]$.
\end{enumerate}
Moreover, $\indepAB$ is equivalent to (i) with $C=\emptyset$.
\end{theorem}

\begin{remark} \label{rk:test_class}
The characterizations (i) and (ii) in Theorem~\ref{thm:test_class} can be modified in various ways. For instance in item (i) a single $v\in V$ is sufficient when $\Lambda^0_{V\setminus\{v\}}$ is a zero measure; see Lemma~\ref{lem:single_v} below.
If the measure $\Lambda$ is homogeneous, it suffices to consider a fixed $\eps>0$ instead of all $\eps>0$, e.g., $\eps=1$.
\end{remark}

\subsection{Probability kernel}
We verify first that the measure $\Lambda$ can be represented in terms of a probability kernel~\citep[p.\ 20]{kallenberg} when restricted to the set $\{y \in \cE : y_C\neq 0\}$.
To this end we use the notation that a set $R$ is of \emph{product form} if it can be represented as $R = \bigtimes_{v\in V}R_v = R_{A \cup B} \times R_C$ with $R_{v} \in \cB(\R)$ and $0_C \not\in R_C$. We postpone the technical proof to Appendix~\ref{sec:CIproofs}.

\begin{lemma}[Probability kernel representation]\label{lem:kernel}
There exists a $\Lambda_C$-unique probability kernel $\nu_C$ from $\cE^C$ to $\R^{A\cup B}$ such that for all $R$ of product form we have
\[\Lambda(R)=\int_{y_C\in R_C}\nu_C(y_C,R_{A\cup B})\Lambda_C(\D y_C)\in[0,\infty],\qquad \text{if }0_C\notin R_C.\]
Moreover, 
for any $c\in C$ and $\eps>0$ such that $\Lambda(R_{c,\eps})>0$ there is the identity
\[\nu_C(y_C,R_{A\cup B})=\p(Y_{A\cup B}\in R_{A\cup B}|Y_C=y_C),\qquad Y\sim \L_{R_{c,\eps}}\]
for $\Lambda_C$-almost all $y_C$ with $|y_c|\geq \eps$.
\end{lemma}

We can now characterize conditional independence as given by Definition~\ref{def:CILambda} in terms of this kernel representation.

\begin{theorem}[CI via kernel factorization]\label{thm:kernel}
Let $(A,B,C)$ be a partition of~$V$. Then
$A \indep B \mid C \; [\Lambda]$ if and only if the probability kernel factorizes,
\[\nu_C(y_C,R_{A\cup B})=\nu_C(y_C,R_{A}\times \R^B)\nu_C(y_C,R_{B}\times \R^A)\]
 for $\Lambda_C$-almost all $y_C\neq 0_C$ and any $R_{A\cup B}$ of product form, and additionally $A \indep B \; [\Lambda_{A\cup B}^0]$.
\end{theorem}
\begin{proof}
According to Theorem~\ref{thm:test_class} (ii) it is sufficient to show that  
$Y_A\indepp Y_B \mid Y_C$ with $Y\sim \L_{R_{c,\eps}}$ for all $c\in C,\,\eps>0$ if and only if the kernel factorizes for $\Lambda_C$-almost all $y_C\neq 0_C$.
Fix $c\in C,\,\eps>0$ and note that $Y_A\indepp Y_B \mid Y_C$ is equivalent to 
\[\p(Y_{A\cup B}\in R_{A\cup B}|Y_C=y_C)=\p(Y_{A}\in R_{A}|Y_C=y_C)\p(Y_{B}\in R_{B}|Y_C=y_C)\]
for $\p_{Y_C}$-almost all $y_C$ and any product form~$R_{A\cup B}$. These are $\Lambda_C$-almost all $y_C$ satisfying $|y_c|\geq \eps$ and the factorization can be stated in terms of the kernel $\nu_C$ according to its representation in Lemma~\ref{lem:kernel}. 
This factorization readily extends to $\Lambda_C$-almost all $y_C\neq 0_C$, since $c\in C,\,\eps>0$ are arbitrary, which completes the proof.
\end{proof}

\subsection{Density factorization}
An important case concerns the measure $\Lambda$ having density $\lambda$ on $\cE$ with respect to some product measure 
\[\mu=\bigotimes_{v\in V}\mu_v,\]
where each $\mu_v$ is a $\sigma$-finite measure on $(\R,\mathcal B(\R))$.
For any non-empty $D\subset V$, the marginal density $\lambda_D$, the density of $\Lambda_D$ in~\eqref{eq:Lambda0}, with respect to $\mu_D=\bigotimes_{v\in D}\mu_v$ is 
\[\lambda_D(y_D)=\int \lambda(y_D,y_{V\setminus D})\mu_{V\setminus D}(\D y_{V\setminus D}),\qquad y_D\neq 0_D,\]
which must be finite for $\mu_D$-almost all $y_D\neq 0_D$. We also note that the density $\lambda^0_{D}$ of $\Lambda_{D}^0$ with respect to $\mu_{D}$ is given by
\[\lambda^0_{D}(y_{D})=\lambda(y_{D},0_{V\setminus D})\mu_{V\setminus D} (\{0_{V\setminus D}\}),\qquad y_{D}\neq 0_D.\]
%%Note that $\lambda^0_{D}$ vanishes, when $\mu_{V\setminus D}(\{0_{V\setminus D}\})=0$.

\begin{lemma}[Kernel density]\label{lem:density}
Assume that $\Lambda$ has a $\mu$-density $\lambda$. Then the probability kernel $\nu(y_C,\cdot)$ has density
\[f_{y_C}(y_{A\cup B})=\frac{\lambda(y)}{\lambda_C(y_C)}\]
with respect to $\mu_{A\cup B}$ for $\Lambda_C$-almost all~$y_C\neq 0_C$. 
\end{lemma}
\begin{proof}
Take any $R$ of product form such that $0_C\notin R_C$, and observe that
\[\Lambda(R)=\int_{y_C\in R_C}\Big( \int_{y_{A\cup B}\in R_{A\cup B}} \frac{\lambda(y)}{\lambda_C(y_C)}\mu_{A\cup B}(\D y_{A\cup B})\Big)\Lambda_C(\D y_C).\]
Hence the expression in brackets coincides with $\nu(y_C,R_{A\cup B})$ for all $R_{A\cup B}$ and $\Lambda_C$-all $y_C\neq 0_C$, see Lemma~\ref{lem:kernel}. 
\end{proof}

This readily yields another characterization result of our conditional independence (Definition~\ref{def:CILambda} with $C \neq \emptyset$).

\begin{theorem}[CI via density factorization]\label{thm:density_factorization}
Let $(A,B,C)$ be a partition of $V$ and assume that $\Lambda$ has a $\mu$-density $\lambda$. Then $A \indep B \mid C \; [\Lambda]$ if and only if the factorization
\[\lambda(y)\lambda_C(y_C)=\lambda_{A\cup C}(y_{A\cup C})\lambda_{B\cup C}(y_{B\cup C})\]
holds for $\mu$-almost all $y\in \cE,\,y_C\neq 0_C$, and additionally,
$A \indep B\; [\Lambda^0_{A\cup B}]$.
\end{theorem}
\begin{proof}
According to Theorem~\ref{thm:kernel} it will suffice to establish equivalence to the kernel factorization.
By Lemma~\ref{lem:density} kernel factorization is equivalent to $\mu_{A\cup B}$-almost everywhere factorization of the density~$f_{y_C}$ for $\Lambda_C$-almost all $y_C\neq 0_C$.
The marginal density of $f_{y_C}$ with respect to $A$ is given by $\lambda_{A\cup C}(y_{A\cup C})/\lambda_C(y_C)$ and the analogous result holds for the set~$B$.
Thus, we have 
\[\frac{\lambda(y)}{\lambda_C(y_C)}=\frac{\lambda_{A\cup C}(y_{A\cup C})}{\lambda_C(y_C)}\,\frac{\lambda_{B\cup C}(y_{B\cup C})}{\lambda_C(y_C)}\]
for $\mu_{A\cup B}\otimes \Lambda_C$-almost all $(y_{A\cup B},y_C)$ with $y_C\neq 0_C$.
It is left to note that the set $\{y:\lambda_C(y_C)=0,\lambda_{A\cup C}(y_{A\cup C})>0\}$ is $\mu$-negligible.
\end{proof}

Importantly, we do not claim that the independence statement $A\indep B\; [\Lambda]$ implies the factorization $\lambda(y)=\lambda_A(y_A)\lambda_B(y_B)$. The marginals $\lambda_A$ and $\lambda_B$ are not even defined for $y_A=0_A$ or $y_B=0_B$, respectively, and considering a restriction to $y_A\neq 0_A$, $y_B\neq 0_B$ does not resolve the matter either; see Section~\ref{subsec:backtodensity} below. 
Instead, as will be also shown in Section~\ref{subsec:backtodensity},
one needs to consider a modified density  
in order to achieve an equivalence to a (modified) density factorization.

\section{Independence and the support of $\Lambda$}\label{sec:independence}
\subsection{Independence characterization via the support of $\Lambda$}
Throughout this section we make the following explosiveness assumption
\begin{align}\tag{E0}\label{eq:infinite0} 
\Lambda\big(y_v\neq 0\big)=\Lambda_{\{v\}}\big(\cE^{\{v\}}\big)
=\Lambda\big(\cE^{\{v\}} \times \RR^{V \setminus \{v\}}\big) \in\{0,\infty\}
\qquad \text{for all }v\in V,
\end{align} 
which means that all one-dimensional marginal measures $\Lambda_{\{v\}}$ of $\Lambda$ are either infinite or zero measures. 
It is much less restrictive than imposing homogeneity on the measure $\Lambda$; however note that \eqref{eq:infinite0} is always met for a $-\alpha$-homogeneous $\Lambda$, since $\cE^{\{v\}} \times \RR^{V \setminus \{v\}}$ is a scale-invariant set.

\begin{proposition}[Independence]\label{prop:indep}
Consider a partition $(A,B)$ of $V$ and assume~\eqref{eq:infinite0}.
Then
\[\indepAB\qquad  \iff \qquad \Lambda(y_A\neq 0_A \text{ and }\, y_B\neq 0_B )=0.\]
\end{proposition}
\begin{proof}
The if statement follows trivially from the definition, since  any $Y \sim \PP_R$ for an admissible $R \in \cR(\Lambda)$ is such that either $Y_A=0_A$ a.s.\ or $Y_B=0_B$ a.s. (the assumption~\eqref{eq:infinite0} is not needed for this direction).

Now assume $\indepAB$ and that the support of $\Lambda$ contains some $y$ with $y_a\neq 0$ and $y_b\neq 0$ for some $a\in A$, $b\in B$. Note that $\Lambda_{\{a\}}$ is infinite by~\eqref{eq:infinite0}, since it can not be a zero measure.
Furthermore, we may take $\delta>0$ small such that $\Lambda(R_{a,\delta}\cap R_{b,\delta})>0$.
Letting $Y\sim \PP_{R_{a,\eps}}$ with $\eps\in(0,\delta)$ we observe that
\[\frac{\Lambda(R_{a,\delta}\cap R_{b,\delta})}{\Lambda(R_{a,\eps})}=\p(|Y_a|\geq \delta,|Y_b|\geq \delta)=\p(|Y_a|\geq \delta)\p(|Y_b|\geq \delta)
=\frac{\Lambda(R_{a,\delta})}{\Lambda(R_{a,\eps})}\frac{\Lambda(R_{b,\delta}\cap R_{a,\eps})}{\Lambda(R_{a,\eps})}.\]
Multiply both sides by $\Lambda(R_{a,\eps})\in(0,\infty)$ and observe that $\Lambda(R_{b,\delta}\cap R_{a,\eps})$ stays bounded, whereas $\Lambda(R_{a,\eps})\to\Lambda(y_a\neq 0)=\infty$ as $\eps\downarrow 0$. Hence $\Lambda(R_{a,\delta}\cap R_{b,\delta})=0$, a contradiction.
\end{proof}

Proposition~\ref{prop:indep} is closely related to the concept of asymptotic independence in extremes, which we discuss in~Section~\ref{sec:extremes}.
We may rephrase the condition in Proposition~\ref{prop:indep} in a number of ways, for instance in terms of the sets 
\begin{align}\label{eq:face}
\cH_D = \{y \in \cE \,:\, y_{V\setminus D} = 0_{V \setminus D}\} = \bigcup_{\emptyset \neq A \subset D} \cE_A,
\end{align}
where 
\begin{align}\label{eq:elementary-face}
\cE_A = \{y\in \cE \,:\, y_i\neq 0 \,\,\forall\, i\in A\, \text{ and }\,  y_i=0\,\forall\, i\in V \setminus A\}=(\RR \setminus \{0\})^A \times \{0_{V \setminus A}\}
\end{align}
is the sub-face of $\cE$ corresponding to the subset $A\subset V$; it should not be confused with $\cE^A= \RR^A\setminus \{0_A\}$.
%%=\{y \in \RR^V\setminus \{0_V\} \,:\, y_i= 0 \,\,\forall\, i\notin D \} 
%%can then be expressed as the disjoint union
%%\begin{align}\label{eq:face}
 %% \cH_D = \bigcup_{\emptyset \neq A \subset D} \cE_A.
%%\end{align}
%%Naturally, for $A \subset D$, we have $\cH_A \subset \cH_D$.
For convenience we summarize a few of these reformulations as follows. Figure~\ref{fig:EH} provides an illustration.
\begin{corollary} \label{cor:indepRephrase}
Consider a partition $(A,B)$ of $V$ and assume~\eqref{eq:infinite0}. The following statements are equivalent. 
\begin{enumerate}[label=(\roman*)]
\item $\indepAB$
\item $\Lambda(\cE^A \times \cE^B)=\Lambda(y_A\neq 0_A \text{ and }\, y_B\neq 0_B )=0$
\item ${\rm supp}(\Lambda)\subset \cH_A\cup \cH_B$
\item $\Lambda_A=\Lambda_A^0$
\item $\Lambda_B=\Lambda_B^0$
\end{enumerate}
\end{corollary}

The independence characterization of Proposition~\ref{prop:indep} (or its extended reformulation Corollary~\ref{cor:indepRephrase}) only treats the case where $(A,B)$ is a partition of $V$. Suppose that this is not the case and denote by $(A,B,C)$ a partition of $V$. We recall that by~\eqref{subset_equivalenece}, $\indepAB$ is equivalent to $\indepABnonpar$ and that \eqref{eq:infinite0} for $\Lambda$ implies \eqref{eq:infinite0} for $\Lambda_{A \cup B}$. Thus, applying Proposition~\ref{prop:indep} to the measure $\Lambda_{A \cup B}$, we obtain 
\begin{align} \label{eq:Inopartition} \indepAB \,\,\,\, \iff  \,\,\,\, \Lambda(y_A\neq 0_A, \, y_B\neq 0_B )=0 \,\,\,\, \iff \,\,\,\, {\rm supp}(\Lambda)\subset \cH_{A \cup C}\cup \cH_{B \cup C}.
\end{align}
Due to Lemma~\ref{lem:Lambda0}, we will often work with the statement $\indepABzero$. 
If \eqref{eq:infinite0} holds for the reduced measure $\Lambda^0_{A \cup B}$,
Proposition~\ref{prop:indep} translates into the equivalence
\begin{align*}
\indepABzero \quad &\iff \quad \Lambda(y_A\neq 0_A, \, y_B\neq 0_B, \, y_C=0_C )=0\\
&\iff \quad {\rm supp}(\Lambda)
\subset 
%%%\cH_{A \cup C} \cup \cH_{B \cup C} \cup (\cE \setminus \cH_{A \cup B})
%%%=
\cH_A\cup \cH_B \cup (\cE \setminus \cH_{A \cup B}).
\end{align*}
When translating items (iv) and (v) from Corollary~\ref{cor:indepRephrase} to these situations, where $A$ and $B$ are not a partition of $V$, but only disjoint, it is useful to take Lemma~\ref{lemma:marginalreducedcompatibility} from Appendix~\ref{app:basic} into account.

\begin{figure}[b!]
  \mbox{\small
%    \begin{tikzpicture}[scale = 1.5, >=stealth]
%      %---------------------------------
%      \coordinate (O) at (0,0,0);
%      \coordinate (e1+) at (0,0,1);
%      \coordinate (e2+) at (1,0,0);
%      \coordinate (e3+) at (0,1,0);
%      \coordinate (e1-) at (0,0,-1);
%      \coordinate (e2-) at (-1,0,0);
%      \coordinate (e3-) at (0,-1,0);
%      %---------------------------------
%      \coordinate (e1+2+) at (1,0,1);
%      \coordinate (e1-2-) at (-1,0,-1);
%      \coordinate (e1+2-) at (-1,0,1);
%      \coordinate (e1-2+) at (1,0,-1);
%      %---------------------------------
%      \coordinate (e2+3+) at (1,1,0);
%      \coordinate (e2-3-) at (-1,-1,0);
%      \coordinate (e2-3+) at (-1,1,0);
%      \coordinate (e2+3-) at (1,-1,0);
%      %---------------------------------
%      \filldraw[dashed, color=red, fill=red!10] (e1+2+) -- (e1+2-) -- (e1-2-) -- (e1-2+) -- (e1+2+) node[anchor=north east, xshift=1cm]{$\cE_{\{1,2\}}$};
%      %---------------------------------
%      \draw[thick, blue, ->] (e1-) to (e1+) node[anchor=north east,xshift=2mm]{$\cE_{\{1\}}$} ;
%      \draw[thick, blue, ->] (e2-) to (e2+) node[anchor=west, yshift=-2mm]{$\cE_{\{2\}}$} ;
%      \draw[thick, blue, ->] (e3-) to (e3+) node[anchor=west]{$\cE_{\{3\}}$} ;
%      %---------------------------------
%      \filldraw(O) circle (1pt) node[anchor=south east] {$0$};
%    \end{tikzpicture}  
%    \hspace{5mm}    
        \begin{tikzpicture}[scale = 1.5, >=stealth]
              %---------------------------------
      \coordinate (O) at (0,0,0);
      \coordinate (e1+) at (0,0,1);
      \coordinate (e2+) at (1,0,0);
      \coordinate (e3+) at (0,1,0);
      \coordinate (e1-) at (0,0,-1);
      \coordinate (e2-) at (-1,0,0);
      \coordinate (e3-) at (0,-1,0);
      %---------------------------------
      \coordinate (e1+2+) at (1,0,1);
      \coordinate (e1-2-) at (-1,0,-1);
      \coordinate (e1+2-) at (-1,0,1);
      \coordinate (e1-2+) at (1,0,-1);
      %---------------------------------
      \coordinate (e2+3+) at (1,1,0);
      \coordinate (e2-3-) at (-1,-1,0);
      \coordinate (e2-3+) at (-1,1,0);
      \coordinate (e2+3-) at (1,-1,0);
      %---------------------------------
      \filldraw[dashed, color=red, fill=red!10] (e1+2+) -- (e1+2-) -- (e1-2-) -- (e1-2+) -- (e1+2+) node[anchor=north east, xshift=1cm]{$\cH_{\{1,2\}}$};
      %---------------------------------
      \draw[thick, red, ->] (e1-) to (e1+);
      \draw[thick, red, ->] (e2-) to (e2+);
      \draw[thick, blue, ->] (e3-) to (e3+) node[anchor=west]{$\cH_{\{3\}}$} ;
      %---------------------------------
      \filldraw(O) circle (1pt) node[anchor=south east] {$0$};
        \end{tikzpicture}
        \hspace{5mm}
     \begin{tikzpicture}[scale = 1.5, >=stealth]
       %---------------------------------
       \filldraw[dashed, color=red, fill=red!10]   (e1-2+) -- (e2+) -- (e2-) -- (e1-2-) --  (e1-2+)  node[anchor=west]{$\cH_{\{1,2\}}$};
       \draw[thick, red] (e1-) to (O);
       %---------------------------------
       \filldraw[dashed, color=blue, fill=blue!10, opacity=0.75] (e2+3+) -- (e2-3+) -- (e2-) -- (e2+) -- (e2+3+);
              \filldraw[dashed, color=blue, fill=blue!10, opacity=0.9] (e2+3-) -- (e2-3-) -- (e2-) -- (e2+) -- (e2+3-) node[anchor=west]{$\cH_{\{2,3\}}$};
              \draw[thick, blue, ->] (e3-) to (e3+);
              %---------------------------------
              \filldraw[dashed, color=red, fill=red!10, opacity=0.75]  (e1+2+) -- (e2+) -- (e2-) -- (e1+2-) --  (e1+2+);
              \draw[thick, blue!50!red, ->] (e2-) to (e2+);
      %---------------------------------
      \draw[thick, red, ->] (O) to (e1+);
          %---------------------------------
      \filldraw(O) circle (1pt) node[anchor=south east] {$0$};
            \end{tikzpicture}
 }   

  \caption{\small  Illustration of the sets $\cH_D$ for $V=\{1,2,3\}$. No mass of $\Lambda$ outside of $\cH_{\{1,2\}}$ and $\cH_{\{3\}}$ corresponds to $\{1,2\} \indep \{3\}\; [\Lambda]$, see Proposition~\ref{prop:indep}. No mass of $\Lambda$ outside of $\cH_{\{1,2\}}$ and $\cH_{\{2,3\}}$ corresponds to $\{1\} \indep \{3\}\; [\Lambda]$, see~\eqref{eq:Inopartition}.}
  \label{fig:EH}
\end{figure}
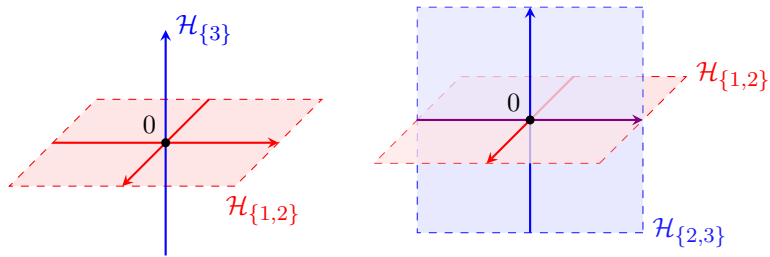

\subsection{Contraction semi-graphoid property (L4)}
We are now in a position to establish the remaining contraction semi-graphoid property (L4). In order to do so, we need to strengthen the explosiveness assumption~\eqref{eq:infinite0}, so that it also holds for all underlying restricted measures:
\begin{align}\tag{E1}\label{eq:infinite}
\Lambda^0_D\big(y_d\neq 0\big)&=(\Lambda^0_D)_{\{d\}}\big(\cE^{\{d\}}\big)\\
\notag
&=\Lambda\big(\cE^{\{d\}} \times \RR^{D \setminus \{d\}} \times \{0_{V \setminus D}\}\big) \in\{0,\infty\}
\qquad \text{for all }d \in D\subset V.
\end{align} 
Assumption~\eqref{eq:infinite} means that all one-dimensional marginals $(\Lambda^0_D)_{\{d\}}$ of all restricted measures $\Lambda^0_D$  (including $\Lambda=\Lambda^0_V$) are either infinite or zero measures. 
This necessarily entails that also all one-dimensional marginals 
$(\Lambda_D)_{\{d\}}=\Lambda_{\{d\}}$
of all marginal measures $\Lambda_D$ are infinite or zero measures, hence \eqref{eq:infinite0}. 

\begin{theorem}\label{thm:L4}
Under Assumption~\eqref{eq:infinite} conditional independence with respect to $\Lambda$ as defined in Definition~\ref{def:CILambda} satisfies all four semi-graphoid properties (L1)--(L4). 
\end{theorem}
\begin{proof}
Properties (L1)--(L3) are already established in Proposition~\ref{prop:LauritzenEasyPart}.
To prove (L4), it is sufficient to assume that the union of $A,B,C,D$ is~$V$; see Lemma~\ref{lemma:marginalcompatibility}.

First we establish (L4) for $C=\emptyset$. That is, we show that for a partition $(A,B,D)$ of $V$
the statements $\indepAB$ and $A \indep D \mid B\;[\Lambda]$
imply $A \indep B\cup D\;[\Lambda]$.
According to Proposition~\ref{prop:indep} we need to show that $\Lambda$ has no mass outside of $\cH_A \cup \cH_{B\cup D}$.
Note that $\indepAB$ implies that there is no mass outside of $\cH_{A\cup D}\cup \cH_{B\cup D}$, whereas 
$A \indep D \mid B\;[\Lambda]$ implies $A \indep D\; [\Lambda^0_{A\cup D}]$, see Lemma~\ref{lem:Lambda0}, and thus all the mass of $\Lambda$ on $\cH_{A\cup D}$ must be contained in $\cH_A\cup \cH_D$. Summarizing, ${\rm supp} ({\Lambda}) \subset (\cH_A \cup \cH_D) \cup \cH_{B\cup D}= \cH_A \cup \cH_{B \cup D}$, completing the first part of the proof. Here we used the explosiveness assumption \eqref{eq:infinite0} for the measures $\Lambda_{A \cup B}$ and $\Lambda^0_{A\cup D}$, which follows from~\eqref{eq:infinite}.

Finally we prove the case $C\neq \emptyset$.
As a consequence of Lemma~\ref{lem:Lambda0} and Lemma~\ref{lemma:marginalreducedcompatibility} we see that $A \indep B \mid C\;[\Lambda]$ implies $A \indep B\;[\Lambda^0_{A\cup B \cup D}]$, whereas $A \indep D \,|\, B \cup C \, [\Lambda]$ implies $A \indep D \,|\, B\, [\Lambda^0_{A\cup B \cup D}]$ (see~Remark~\ref{rk:CIreduced}). Hence from the first step we readily get $A \indep (B\cup D) \, [\Lambda^0_{A\cup B \cup D}]$.
According to Theorem~\ref{thm:test_class}~(ii) it is left to consider $Y\sim \p_{R_{c,\eps}}$ for $c\in C$, $\eps>0$, and to establish $Y_A\indepp Y_{B\cup D}\mid Y_C$. But the assumption in (L4) implies $Y_A\indepp Y_{B}\mid Y_C$ (see Lemma~\ref{lemma:marginalcompatibility}) and $Y_A\indepp Y_{D}\mid Y_{B\cup C}$, and we apply the classical analogue of~(L4) to complete the proof.
\end{proof}

\begin{remark} \label{rk:A1}
At first sight, the explosiveness assumption~\eqref{eq:infinite} may seem restrictive. However, it is a very natural condition for exponent measures and L\'evy measures in the context of infinitely divisible distributions; see~Section~\ref{sec:motivation}. 
Importantly, for a $-\alpha$-homogeneous $\Lambda$ as in~\eqref{def:homog}, all measures $\Lambda^0_D$ are $-\alpha$-homogeneous and so Assumption~\eqref{eq:infinite} is satisfied.
In addition, the necessity of this assumption (see~Example~\ref{ex:violation}) underlines the special role of the origin in our setting, which is not apparent through the first three properties (L1)--(L3).
\end{remark}

%%So taken together with Proposition~\ref{prop:LauritzenEasyPart}, we can finally conclude that, under Assumption~\eqref{eq:infinite},
%%the conditional independence introduced here  in Definition~\ref{def:CILambda} satisfies all four semi-graphoid properties (L1)--(L4).

\subsection{A modified density} \label{subsec:backtodensity}

Here we revisit the case when $\Lambda$ has a density $\lambda$ with respect to a product measure~$\mu=\bigotimes_{v \in V} \mu_v$.
Let us first review the independence situation.
 Under assumption~\eqref{eq:infinite0} the independence $A\indep B\; [\Lambda]$ is equivalent to $\Lambda(y_A\neq 0_A, y_B\neq 0_B)=0$.
 In this case 
\begin{align*}
\lambda_D(y_D) &= \lambda_D^0(y_D), && \text{for $\mu_D$-almost all } y_D \in \cE^D \text{ for } D \in \{A,B\},
%%\lambda_B(y_B) &= \lambda_B^0(y_B), &&  
%%\text{for $\mu_B$-almost all } y_B \in \cE^B,
\end{align*}
whereas $\lambda(y)=0$ for $\mu$-almost all $y=(y_A,y_B) \in \cE^A \times \cE^B$.
So, unless we have the degenerate situation that the full  measure $\Lambda$ is entirely concentrated either on $\cH_A$ only or on $\cH_B$ only, a factorization of the form
$\lambda(y) = \lambda_A(y_A)\lambda_B(y_B)$ for $\mu$-almost all $y \in \cE^A \times \cE^B$ cannot hold true. 
 Hence the density does not factorize on the product set where the marginals are defined.

{However, with a slight modification of the density, we can achieve a factorization, as long it is tested only on the correct subsets of $\RR^V$.}
To this end, we will assume $\mu_v(\{0_v\})>0$ for all $v\in V$, and we define for a non-empty $D\subset V$ the modified densities
\begin{align}\label{extended_dens}
\overline \lambda_D(y_D) = 
\begin{cases}
\lambda_D(y_D) & \text{for } y_D \in \cE^D=\R^D \setminus\{0_D\},\\
1/\mu_{D}(\{0_D\}) & \text{for } y_D=0_D,
\end{cases}
\end{align}
whereas $\olambda_\emptyset = 1$. In particular, $\lambda(y)=\olambda(y)$ for all $y \in \cE=\RR^V\setminus \{0_V\}$.
As a special case of the following result we will see that $A\indep B\;[\Lambda]$ is equivalent to the factorization
\[\overline\lambda(y)=\overline\lambda_{A}(y_{A})\overline\lambda_{B}(y_{B})\]
for $\mu$-almost all $y\in\cE$ with either $y_A=0_A$ or $y_B=0_B$. 

\begin{proposition}\label{prop:trick}
Assume~\eqref{eq:infinite} and that $\Lambda$ has a $\mu$-density $\lambda$, where $\mu_v(\{0_v\})>0$ for all $v\in V$.
Then for a partition $(A,B,C)$ of $V$ with possibly empty $C$ the conditional independence $A\indep B\mid C\; [\Lambda]$ is equivalent to the factorization 
\begin{align}\label{eq:lambdabar-CI}
\overline\lambda(y)\overline\lambda_{C}(y_{C})=\overline\lambda_{A\cup C}(y_{A\cup C})\overline\lambda_{B\cup C}(y_{B\cup C})
\end{align}
for $\mu$-almost all $y\in\RR^V \setminus\{y:y_{A}\neq 0_A, y_B\neq 0_B,y_C=0_C\}$.  
\end{proposition}
\begin{proof}
For $y=0_V$ the statement \eqref{eq:lambdabar-CI} is trivial; it follows simply from the factorization of $\mu(\{0_V\})$ into the factors $\mu_{v}(\{0_v\})$, $v \in V$.
On the set $\{y_C\neq 0_C\}$ the factorization \eqref{eq:lambdabar-CI} is that of~$\lambda$ as in Theorem~\ref{thm:density_factorization}.
In view of Theorem~\ref{thm:density_factorization}
it is therefore sufficient to check that $A \indep B\; [\Lambda^0_{A\cup B}]$ is equivalent to the factorization of $\overline\lambda$ for $\mu$-almost all $y\neq 0_V$ with $y_{A\cup C}=0_{A\cup C}$ or $y_{B\cup C}=0_{B\cup C}$.

Indeed, according to Corollary~\ref{cor:indepRephrase} the independence $\indepABzero$ is equivalent to $\Lambda(y_A\neq 0_A,y_B\neq 0_B,y_C=0_C)=0$, which holds if and only if
\begin{align}\label{eq:lambdaAC0}
\lambda_{A \cup C}(y_{A \cup C})= \lambda_A^0(y_A)/\mu_C(\{0_C\}) &= \lambda(y) \mu_{B}(\{0_{B}\})
\end{align}
for $\mu$-almost all $y\neq 0_V$ with $y_{B\cup C}=0_{B\cup C}$, and the analogous statement with exchanged roles of $A$ and $B$ holds.
The equality \eqref{eq:lambdaAC0} can be rewritten as
\[\lambda(y)/\mu_C(\{0_C\})=\lambda_{A\cup C}(y_{A\cup C})/\mu_{B\cup C}(\{0_{B\cup C}\}),\]
which is the claimed factorization of $\overline\lambda$ on $\{y \in \cE : y_{B \cup C}=0_{B \cup C}\}$. The same argument applies with roles of $A$ and $B$ interchanged, and the proof is complete.
\end{proof}

\begin{remark}
If $\indepABC$, then $\Lambda(\{y:y_{A}\neq 0_A, y_B\neq 0_B,y_C=0_C\})=0$.
Proposition~\ref{prop:trick} does not need the entire extended explosiveness assumption~\eqref{eq:infinite}.
It is sufficient to assume \eqref{eq:infinite0} for the reduced measure $\Lambda^0_{A \cup B}$ instead, which is guaranteed by \eqref{eq:infinite}.
\end{remark}

\subsection{A generic construction} \label{subsec:CI3constr}
To conclude this section we give a constructive example of a measure $\Lambda$ in $d=3$ dimensions with index set $V = \{1,2,3\}$ such that the conditional independence $\{1\} \indep \{3\} \mid \{2\}\; [\Lambda]$ holds. As a dominating product measure $\mu$ we take
\begin{align}\label{product_measure}
  \mu(\D x_1,\ldots,\D x_d)=\bigotimes_{i=1}^d \mu_i(\D x_i)=
  \bigotimes_{i=1}^d(\D x_i+\delta_0(\D x_i)),
\end{align}
where $\delta_0$ is the point mass at~0. 
As building blocks, we employ two bivariate densities $\kappa_{ij}$ for $(i,j) \in \{ (1,2), (2,3)\}$ defined on $(0,\infty)^2$ with identical univariate marginal densities $\kappa_i(y_i)=m(y_i)$, $y_i \in (0,\infty)$, for all $i \in V$.  As in~\eqref {eq:Lambda} we further assume that the marginal measure on $(0,\infty)$ defined by the density $m$ puts finite mass on sets bounded away from 0.
For each of these bivariate densities $\kappa_{ij}$ we can add mass on the axes to define a density with respect to the measure $\mu_{ij}=\mu_i  \otimes \mu_j$ on $\cE_+^{ij}=[0,\infty)^{\{i,j\}} \setminus \{0_{\{i,j\}}\}$ by
\begin{align}\label{ext_density}
  \lambda_{ij}(y_{ij})=p_{ij}\ind{y_i,y_j\neq 0} \kappa_{ij}(y_{ij})+q_{ij}\ind{y_j=0}m(y_i)+q_{ij}\ind{y_i=0}m(y_j), \quad y_{ij} \in \cE_+^{ij},
\end{align}
where $p_{ij} \in [0,1]$ is a mixture probability and $q_{ij}=1-p_{ij}$; the case $p_{ij}=0$ implies independence between components $i$ and $j$ in the sense of Proposition~\ref{prop:indep}.

 By construction, the marginal densities of $\lambda_{ij}$ also satisfy $\lambda_i(y_i)=m(y_i)$ for $i \in V$. 
We can combine these two bivariate densities using the modified densities in~\eqref{extended_dens} into a trivariate density
\begin{align}\label{eq:lambdaCIconstr} 
  \lambda(y) = \overline \lambda_{12}(y_{12})\overline \lambda_{23}(y_{23})/ \overline \lambda_2(y_2),
\end{align}
for any $y\in\cE_+\setminus\{y:y_{1}\neq 0, y_3\neq 0,y_2=0\}$, and $\lambda(y) = 0$ otherwise. 
To be more precise, we can write this out in terms of the densities $\lambda_{ij}$ and obtain for any $y \in \cE_+$
\begin{align*}
  \lambda(y)&=\lambda_{12}(y_{12})\lambda_{23}(y_{23})/ m(y_2), && y_2\neq 0,\\
  \lambda(y_1,0,0)&=\lambda_{12}(y_{1},0), && y_1 \neq 0,\\ 
  \lambda(0,0,y_3)&=\lambda_{23}(0,y_{3}), && y_3 \neq 0,\\
  \lambda(y_1,0,y_3)&= 0, && y_1 \neq 0,\, y_3 \neq 0.
\end{align*}
The corresponding measure is determined for any Borel set $A\subset \cE_+$, bounded away from the origin, by
\[ \Lambda(A) = \int_A  \lambda(y) \mu(\D y).\]
By construction, $\Lambda$ has identical marginals $\Lambda(y_i>u) = \int_u^\infty m(y) \D y$, $u>0$, $i \in V$, and it satisfies the finiteness condition~\eqref{eq:Lambda}. Indeed, for any set $A\subset \mathcal E_+$ bounded away from the origin there exists $\epsilon>0$ such that $A \subset L_\epsilon = \mathcal E_+ \setminus [0,\epsilon]^d$, and therefore
\[\Lambda(A) \leq \Lambda(L_\epsilon) \leq \sum_{i\in V} \Lambda(y_i > \epsilon) = d  \int_\epsilon^\infty m(y) \D y < \infty. \]
The measure $\Lambda$ also satisfies the conditional independence $\{1\} \indep \{3\} \mid \{2\}\; [\Lambda]$ according to~Proposition~\ref{prop:trick}.
If $0 < p_{12}, p_{23} < 1$ then the density of $\Lambda$ puts mass on every sub-face of $\cE_+=[0,\infty)^3\setminus\{0\}$ except for $\{y : y_1>0, y_2=0, y_3>0\}$. Indeed, the latter is the only sub-face that cannot contain mass in this situation if $\Lambda$ satisfies the explosiveness assumption \eqref{eq:infinite}; see~Lemma~\ref{lem:Lambda0} and  Proposition~\ref{prop:indep}. Further details and an illustration of this construction are provided in Appendix~\ref{app:CIconstr}.
Concrete examples of this construction involving either the homogeneous measure from Example~\ref{ex:HR} or the non-homogeneous measure from Example~\ref{ex:gauss} will be given in Sections~\ref{sec:CIextremes} and~\ref{sec:AI}, respectively.

\section{Undirected graphical models}\label{sec:undirected}

\subsection{Fundamentals}
For the fundamental definitions we follow again  the axiomatic approach to conditional independence as in~\citet{Lauritzen}, from which we also adopt relevant notions that characterize relations in a graph.
Let our index set $V$ coincide with the vertex set (node set) of an undirected graph $\cG = (V,E)$ with set of edges $E$. 
The subset $C$ is said to \emph{separate} $A$ from $B$ if all paths from any of the nodes of $A$ to any of the nodes of $B$ necessarily intersect with $C$.

\begin{definition}\label{def:global}
The measure $\Lambda$ satisfies the \emph{global Markov property} with respect to an undirected graph $\mathcal G$ if $A\indep B\mid C\;[\Lambda]$ for any 
disjoint subsets $A,B,C$ where (possibly empty) $C$ separates $A$ from $B$. In this case we say that $\Lambda$ is an undirected graphical model on $\cG$. 
\end{definition}

\begin{remark}\label{rk:GMP}
In Definition~\ref{def:global}, it is sufficient to consider only partitions $(A,B,C)$ of $V$. If $(A,B,C)$ is instead a triplet of disjoint subsets  of $V$ not forming a partition, one can enlarge the subsets $A$ and $B$ so that separation still holds, i.e.\ there exist subsets $A',B'$, such that $(A',B',C)$ forms a partition of $V$ with $A \subset A'$, $B \subset B'$ and $C$ separates $A'$ and $B'$; and according to the semi-graphoid property (L2) $A'\indep B'\mid C\;[\Lambda]$ implies $A\indep B\mid C\;[\Lambda]$.
\end{remark}

Throughout this section we will further assume that the measure $\Lambda$ satisfies our fundamental explosiveness assumption~\eqref{eq:infinite}. This allows us to apply the independence characterization in Proposition~\ref{prop:indep} to any restricted measure $\Lambda^0_D$ for $\emptyset \neq D \subset V$ and is important as the condition $A \indep B  \, [\Lambda^0_{A\cup B}]$ appears always when checking conditional independence statements; see Section~\ref{sec:alternative}. It has critical implications on the support of~$\Lambda$.  The following simple result is a direct consequence of Lemma~\ref{lem:Lambda0} and Proposition~\ref{prop:indep} and helps understanding where $\Lambda$ is allowed to have mass.

\begin{corollary}\label{cor:mass0}
Assume~\eqref{eq:infinite} and that $\Lambda$ is globally Markov with respect to an undirected graph $\mathcal G$.  
Then
$\Lambda(y_A\neq 0_A, \, y_B\neq 0_B, \, y_C=0_C )=0$
for any 
disjoint subsets $A,B,C$ where (possibly empty) $C$ separates $A$ from $B$. 
\end{corollary}

Assumption~\eqref{eq:infinite} ensures the validity of all four semi-graphoid properties (L1)--(L4); see~Theorem~\ref{thm:L4}. This implies that various standard results for conditional independence in the probabilistic setting carry over automatically to the $\Lambda$-based conditional independence, as they only rely on these properties. For example, the global Markov property implies the \emph{local Markov property}, which in turn implies the \emph{pairwise Markov property} \citep[see][Prop.\ 3.4]{Lauritzen}. 

%%\begin{lemma}\label{lem:mass}
%%Assume~\eqref{eq:infinite} and that $\Lambda$ admits a global Markov property with respect to an undirected graph $\mathcal G$. Then for any $D\subset V$, such that $\mathcal G$ restricted to $D$ is not connected, it holds that $\Lambda(y_v=0, v\notin D\text{ and }y_v\neq 0,v\in D)=0$.
%%\end{lemma}
%%\begin{proof}
%%Since $D$ is not connected, we can find a partition  $(A, B)$ of $D$, such that there are no edges between $A$ and $B$ in $\mathcal G$ restricted to $D$. Hence, the complement $D^c$ separates $A$ from $B$, which gives $A\indep B\mid D^c\;[\Lambda]$, which implies $\indepABzero$ according to Lemma~\ref{lem:Lambda0}.
%%By Proposition~\ref{prop:indep} we obtain  $\Lambda^0_D(y:y_v\neq 0\,\forall v\in D)=0$, which is the assertion.
%%\end{proof}

The theory of conditional independence and graphical models in the case where $\Lambda$ has no mass on any of the lower-dimensional sets $\cH_D$ defined in~\eqref{eq:face} is similar to the classical theory of probabilistic graphical models, and issues arising from the special role of the origin can then be avoided.
Our main focus in this section is therefore on the general case with mass on several lower-dimensional sub-faces, which is highly desirable in applications for instance in extremes; see Section~\ref{sec:extremes}.

\subsection{Decomposable graphs} \label{sec:DecGraphs}

Decomposable (or chordal) graphs are an important sub-class of general graphical structures. A triple of disjoint subsets $(A,B,S)$ is a \emph{decomposition} of $\mathcal G$ if 
$A\cup B\cup S = V$ and a fully connected $S$ separates $A$ from~$B$. 
A graph is then called \emph{decomposable} if it can be reduced to a set of cliques $\mathcal C$ (i.e., maximal fully connected sub-graphs) via decompositions.
Importantly, a graph is decomposable if and only if it has  the following \emph{running intersection property} (e.g. \citet{Lauritzen}~Prop.~2.17 or \citet{maathuis2018handbook}~Section~1.6):
there exists an ordering of its cliques $C_1,\ldots, C_n$ such that for all $j<n$:
\begin{itemize}
\item $S_j:= C_{j+1}\cap D_j \subset C_i$ for some $i\leq j$, where $D_j = \bigcup_{i\leq j} C_i$;
\item the multiset $\cS= \{S_1,\ldots, S_{n-1}\}$ is independent of the ordering,\\ and it is comprised of all minimal fully connected separators of~$\mathcal G$.
\end{itemize}
We note that $\mathcal S$ is a multiset since it may contain the same set several times, and it contains the empty set if $\cG$ is disconnected.
Figures~\ref{fig:chordal} and \ref{fig:forest} illustrate two examples of such  graphs with their cliques and separators.

\begin{figure}[b!]
\tikzstyle{vertex}=[circle,fill=gray!10,draw=black,thin,minimum size=12pt,inner sep=0pt]
\tikzstyle{edge} = [draw,thick,-]

\mbox{\small   
\begin{tikzpicture}[scale=1]
    % vertices
    \foreach \pos/\name in {{(0.33,0.33)/1}, {(1,-1)/2},{(-0.33,-1)/3},{(2,0)/4},{(3,1)/5},{(1.66,1.33)/6},{(3,-1)/7}, {(4,0)/8}}
    %%%%    (5.66,0)/9,(6.5,-1)/12,(8.16,-1)/13, (6.5,1)/10, (7.33,0)/11, (9.33,0)/14}
        \node[vertex] (\name) at \pos {\scriptsize $\name$};
    % edges
    \foreach \source / \dest  in {1/2, 2/3, 2/4, 4/5, 4/6, 5/6, 4/7, 7/8, 5/8, 5/7, 4/8} 
%%%%    9/10, 10/11, 11/12, 11/13}
        \path[edge] (\source) -- (\dest);
\end{tikzpicture}
}

         \caption{\small  
A decomposable graph with (maximal) cliques  $\cC=\{ \{1,2\}, \{2,3\}, \{2,4\}, \{4,5,6\}, \{4,5,7,8\}\}$ 
They are separated by the (minimal) separators  $\cS=\{\{2\},\{2\},\{4\}, \{4,5\}\}$. 
Note that $\cS$ is a multiset.
}
  \label{fig:chordal}
\end{figure}
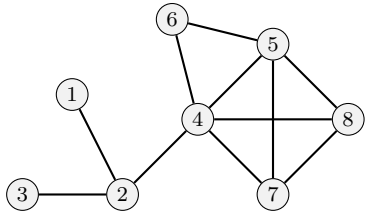

For classical probabilistic graphical models, the Hammersley--Clifford theorem \citep[][Theorem 3.9]{Lauritzen} states that the global Markov property on a decomposable graph is equivalent to the factorization of the density of the model into marginal densities on the cliques.
Considering instead the measure $\Lambda$ with density $\lambda$, we obtain a similar result. If all mass of $\Lambda$ is concentrated on $\cE_V=\{y : y_{i}\neq 0 \, \forall i \in V\}$, then the global Markov property in Definition~\ref{def:global} is equivalent to
\begin{align}\label{HC}
\lambda(y)\prod_{S\in\mathcal S}\lambda_S(y_S) = \prod_{C\in\mathcal C}\lambda_C(y_C),
\end{align}
for $\mu$-almost all $y\in \cE$ with $y_{i}\neq 0$ for all $i\in V$.
This statement is a special case of the much more general Theorem~\ref{thm:decomposable} below;
see also \citet[][Theorem 1]{eng2018} who showed such as factorization for homogeneous measures $\Lambda$ in the context of extreme value theory.

As we turn to the general case with possible mass on several lower-dimensional subsets $\cH_D$, $D \subset V$, an interesting twist of the $\Lambda$-conditional independence theory appears compared to the classical probabilistic theory.
The factorization~\eqref{HC} for $\lambda$ is in general no longer equivalent to the global Markov property and the modified densities $\olambda$ defined in~\eqref{extended_dens} become instrumental.
Our main theorem of this section is a characterization of the global Markov property of $\Lambda$ in the spirit the classical Hammersley--Clifford theorem. For its proof, we need an auxiliary result, Proposition~\ref{prop:recursive_global} in Appendix~\ref{app:recursive_global}, which allows us to derive the global Markov property from a decomposition into a globally Markov sub-graph and a clique. Its proof is based on the four semi-graphoid properties (L1)--(L4), and therefore Assumption~\eqref{eq:infinite} is critical again.
We define for an undirected graph $\cG=(V,E)$ the set
\[\mathcal Z = \mathcal Z(\mathcal G) = \bigcup_{\substack{
\text{$a,b \in V,\, S \in \cS$:}\\ \text{$S$ separates  $a$ from $b$}}} \{y \in \R^V : y_a \neq 0, y_b\neq 0, y_S=0_S\},\]
and observe that $\Lambda$ being globally Markov for $\cG$ implies $\Lambda(\mathcal Z) = 0$ by Corollary~\ref{cor:mass0}.

\begin{theorem}\label{thm:decomposable}
Assume that $\Lambda$ satisfying \eqref{eq:infinite} has a density $\lambda$ with respect to a product measure~$\mu$ with $\mu_v(\{0_v\})>0$, and the graph $\mathcal G$ is decomposable. 
Then $\Lambda$ is globally Markov if and only if
\begin{align}\label{eq:u_dec}
\olambda(y)\prod_{S\in\mathcal S}\olambda_S(y_S) = \prod_{C\in\mathcal C}\olambda_C(y_C)
\end{align} 
for $\mu$-almost all $y\notin \mathcal Z(\mathcal G)$, and $\olambda_\emptyset=1$ by convention. 
\end{theorem}

The proof of this result is based on an inductive argument using the running intersection property for decomposable graphs and given in Appendix~\ref{app:HC}.

\begin{remark}
In the setting of Theorem~\ref{thm:decomposable}, if $\Lambda$ is globally Markov for $\cG$, we have that $\Lambda(\mathcal Z) = 0$ by Corollary~\ref{cor:mass0}. Hence $\lambda(y)=0$ for $\mu$-almost all $y\in \mathcal Z$ in this situation. 
However, it is important to note that this is not(!) true for the right hand side of~\eqref{eq:u_dec}. A simple example of this phenomenon becomes apparent from the construction of Section~\ref{subsec:CI3constr}; see Table~\ref{tab:CIconstr} (the sixth row) in Appendix~\ref{app:CIconstr}.
\end{remark}

\begin{remark}\label{rem:caveat}

Alternatively to the factorization of $\olambda$ in Theorem~\ref{thm:decomposable}, one may be tempted to consider instead a characterization of the global Markov property via the factorization of $\lambda$ as in~\eqref{HC} on the set where all the relevant marginals are defined together with respective support assumptions as in Corollary~\ref{cor:mass0}.
However, first, such a factorization cannot hold for a globally Markov $\Lambda$ on a disconnected graph as discussed in Section~\ref{subsec:backtodensity}. Second, although the factorization of $\lambda$ does hold for a connected graph when $\Lambda$ is globally Markov, one needs to be cautious with the converse implication; one may not deduce the global Markov property from the factorization of $\lambda$ alone.
In Appendix~\ref{app:caveat} we give an example that illustrates this caveat. 
\end{remark}

Collectively, this shows that factorization of $\lambda$ is not a suitable characterization of Markov properties of $\Lambda$ on an undirected graph, and that the modified densities $\olambda$ are the more natural object for this purpose.

%\newpage
\section{Directed graphical models}\label{sec:directed}

\subsection{Fundamentals}
In this section we consider a \emph{directed acyclic graph (DAG)} $\mathcal G$ with vertex set $V$.
The classical definition of the directed Markov properties require the following graph notions; see~\citet{Lauritzen} for details.
Let $\mathrm{pa}(v)$, $\mathrm{an}(v)$ and $\mathrm{de}(v)$ denote \emph{parents}, \emph{ancestors} and \emph{descendants} of $v\in V$, respectively.
The \emph{ancestral set} of a subset $A\subset V$ in a DAG $\cG$ is defined as $A$ together with all the ancestors of every vertex in~$A$. The \emph{moral graph} of $\mathcal G$ is an undirected graph obtained by connecting the parents of every vertex in $\mathcal G$ and dropping the directions of the original edges. The following directed Markov properties are again in line with the axiomatic approach to conditional independence.

\begin{definition}[Directed Markov properties]\label{def:directedMarkov}
\textbf{(DL)} We say that $\Lambda$ satisfies the \emph{directed local Markov property} with respect to the DAG $\cG$ if for every vertex $v\in V$
\[v \indep V\setminus (\{v\}\cup \mathrm{de}(v) \cup \mathrm{pa}(v))\mid \mathrm{pa}(v) \;[\Lambda].\] 
\textbf{(DG)} We say that $\Lambda$ satisfies the \emph{directed global Markov property} with respect to $\cG$ if for every triplet of disjoint subsets $(A,B,S)$ of $V$  it holds that  $A\indep B \mid S \;[\Lambda]$ if $A$ and $B$ are separated by $S$ in the moral graph of the ancestral set of $A\cup B\cup S$.  
\end{definition}

An alternative equivalent characterization of the global Markov property via the concept of \emph{$\mathrm{d}$-separation} is also possible \citep[e.g.,][]{Lauritzen}. Further, it is shown in~\citet[Prop.~4]{lauritzen1990independence} that in the classical setting the two properties (DL) and (DG) are equivalent. The proof only requires properties (L1)--(L4) of the classical conditional independence, and so due to Theorem~\ref{thm:L4} this is also true in our setting as long as $\Lambda$ satisfies our basic explosiveness assumption~\eqref{eq:infinite}.

\begin{corollary}\label{cor:DGDL}
Assume $\Lambda$ satisfies \eqref{eq:infinite} and $V$ is the vertex set of a DAG $\cG$. Then  $\Lambda$  satisfies (DG) with respect to $\cG$ if and only if $\Lambda$ satisfies (DL) with respect to $\cG$. 
\end{corollary}

\subsection{Recursive max-linear models}
\label{sec:claudia}

%%%For a max-stable random vector $X$ with continuous positive density and disjoint non-empty sets $A,B,C \subset V$, the conditional independence $A \indep B \mid C$ implies the independence $A \indep B$ \citep[][Theorem 1]{papastathopoulos2016conditional}. This implies that graphical modeling for max-stable distributions is only possible for models that are degenerate in the sense that the corresponding exponent measure $\Lambda$ does not admit Lebesgues densities on all sub-faces; see \citet[][Proposition 2.1]{dom2016b} for the connection between the existence of the densities of $X$ and $\Lambda$.

Recursive max-linear models introduced by \citet{gis2018} received a lot of attention in recent years \citep[e.g.,][]{gis2018a, gis2019, amendola2022conditional}. They are defined on a DAG $\mathcal G$ by
\begin{align}\label{eq:recmaxlin}
X_i=\bigvee_{j\in \pa(i)}\beta_{ij}~X_j\vee\beta_{ii}~\eps_{i},\qquad i \in V,
\end{align}
where $\beta_{ij}>0$ and $\eps_i\sim F_i$ are independent non-negative random variables. It is a special case of the classical recursive structural equation model and as such 
$X$ satisfies the classical probabilistic (DL) property~\citep[cf.,][Thm.\ 1.4.1]{pea2009}.
Note that $\beta_{ij}$ corresponds to the arc $j \rightarrow i$ in the graph $\mathcal G$, which is a common notation in the literature.

The recursive equation~\eqref{eq:recmaxlin} can be rewritten as
\begin{align}\label{eq:gamma}X_i=\bigvee_{j \in V}\gamma_{ij}~\eps_{j}, \quad i \in V, \qquad
\text{with} \qquad \gamma_{ij}=\bigvee_{p\in \mathrm{paths}_{ji}} \,\,
\prod_{(k \rightarrow l)\in p}
\beta_{lk}, \quad i,j \in V,
\end{align}
where $\mathrm{paths}_{ji}$ is the set of paths from $j$ to $i$ with the additional convention that each path starts with the $(j \rightarrow j)$ transition, and $\gamma_{ij}=0$ for empty $\mathrm{paths}_{ji}$.

Such a model without a restriction on the non-negative $\gamma_{ij}$ is called \emph{max-linear}. It is immediate that {max-linear} models are always max-infinitely divisible since every univariate random variable $\eps_j$ is max-infinitely divisible.  We additionally assume that
\begin{align}\label{eq:assumpF}
\inf\{x:F_j(x)>0\} = 0,\qquad j \in V,
\end{align} 
which guarantees that   $X$ satisfies~\eqref{eq:support} and so the exponent measure $\Lambda$ is finite away from the origin.
Let us further recall that the support of this measure consists of $|V|$ rays, each spanned by a column $(\gamma_{\cdot j})$ \citep[e.g.,][]{yue2014}. 

\begin{lemma} \label{lemma:Lambda_maxlinear}
Let $\Lambda$ be the exponent measure of the recursive max-linear model \eqref{eq:recmaxlin} satisfying \eqref{eq:assumpF}.
Then $\Lambda$ is given by
\[
\Lambda(E)=\sum_{j \in V} \eta_j \big( \{ t > 0 \,:\, t \gamma_{\cdot j} \in E\} \big)
\]
for measurable $E \subset \cE_+$,
where $\gamma_{\cdot j}$ denotes the $j$-th column of the $\lvert V \rvert \times \lvert V \rvert$ matrix $(\gamma_{ij})$ and $\eta_j$ denotes the exponent measure of $\eps_j$.
\end{lemma}
\begin{proof}
From \eqref{eq:exp_measure}, it is immediate that
\[
\Lambda(\cE_+\setminus [0,x])=\sum_{j \in V \,:\, \gamma_{\cdot j} \neq 0_V} \eta_j ((a_j,\infty))  \quad \text{with} \quad a_j=\bigwedge_{i\in V \,:\, \gamma_{ij}>0} {x_i}/{\gamma_{ij}} .
\]
It suffices to note that the case $\gamma_{\cdot j}= 0_V$ cannot appear due to $\gamma_{jj}>0$ and that for $t >0$ we have that $t > a_j$ if and only if $t \gamma_{\cdot j} \in \cE_+ \setminus [0,x]$, while sets of the form $\cE_+\setminus [0,x]$ form a separating measure-determining class.
\end{proof}

The restricted measures $\Lambda^0_D$ for $\emptyset \neq D \subset V$ are also readily obtained from Lemma~\ref{lemma:Lambda_maxlinear} and given by 
\[
\Lambda^0_D(E)=
\sum_{\substack{\text{$j \in V \,:\, \gamma_{D j} \neq 0_D$} \\ \text{$\gamma_{(V\setminus D) j} = 0_{V\setminus D}$}}} 
\eta_j \big( \{ t > 0 \,:\, t \gamma_{D j} \in E\} \big)
\]
for measurable $E \subset \cE^D_+$. The measure $\Lambda^0_D$ is supported on the rays spanned by the columns $(\gamma_{\cdot j})$ (restricted to the $D$-components), whose complementary $V\setminus D$-components are all equal to zero, if such columns exist; otherwise $\Lambda^0_D=0$. In particular, if not null,  $\Lambda^0_D$ corresponds itself to a max-linear model for a subset of the innovations $(\eps_j)$.

\begin{lemma} \label{lem:epsilon_positive}
Let $\Lambda$ be the exponent measure of the recursive max-linear model \eqref{eq:recmaxlin} satisfying \eqref{eq:assumpF}.
If $F_j$ has no mass at 0 for all $j \in V$, then $\Lambda$ satisfies~\eqref{eq:infinite}.
\end{lemma}
\begin{proof}
Since each of the restricted measures $\Lambda^0_D$ corresponds itself to another max-linear model for a subset of $(\eps_j)$ if not null, it is sufficient to prove~\eqref{eq:infinite0}. Now for any $v\in V,t>0$
\[\Lambda(y_v>t)=\sum_{j\in V \,:\, \gamma_{vj}>0}\eta_j((t/\gamma_{vj},\infty))=-\sum_{j\in V \,:\, \gamma_{vj}>0}\log F_j(t/\gamma_{vj}),\]
which either converges to $\infty$ as $t\downarrow 0$ (because $F_j(0+)=0$ by assumption) or remains 0 when all $\gamma_{vj}=0$.
\end{proof}

The main result of this section shows that the exponent measure $\Lambda$ of a recursive max-linear model \eqref{eq:recmaxlin} satisfies also the properties of a directed graphical model with respect to its underlying graph $\mathcal G$  in the sense of Definition~\ref{def:directedMarkov}
(in addition to $X$ being a classical probabilistic graphical model with respect to $\cG$).

\begin{theorem}\label{thm:DAG}
Consider the recursive max-linear model in~\eqref{eq:recmaxlin} defined on a DAG $\mathcal G$ and satisfying~\eqref{eq:assumpF}. Then the corresponding exponent measure $\Lambda$ satisfies the directed local Markov property (DL) on $\mathcal G$. It also satisfies the directed global Markov property (DG) on $\mathcal G$ if $\eps_j>0$ for all $j \in V$.
\end{theorem}
\begin{proof}
In view of Lemma~\ref{lem:epsilon_positive} and Corollary~\ref{cor:DGDL}, it is sufficient to show (DL). In other words, we need to establish $A\indep B\mid S\;[\Lambda]$ for an arbitrary vertex $v\in V$ and $A=\{v\}$, $S=\mathrm{pa}(v)$, $B=V\setminus (\{v\}\cup \mathrm{de}(v) \cup \mathrm{pa}(v))$.
This is the same as conditional independence for the marginal measure $\Lambda_{A \cup B \cup S}=\Lambda_{V\setminus \mathrm{de}(v)}$, which is the exponent measure of the original max-linear model \eqref{eq:recmaxlin} with $(X_i,\eps_i)$, $i\in \mathrm{de}(v)$ removed. Thus, we may assume that $v$ is a terminal node (it has no children), and hence $(A,B,S)$ is a partition of~$V$.

Let us first consider the measure $\Lambda^0_{A\cup B}=\Lambda^0_{V \setminus \mathrm{pa}(v)}$, which is supported by the rays corresponding to the columns $\gamma_{\cdot j}$ with the index $j$ such that $\gamma_{pj}=0$ for all parent nodes $p\in \mathrm{pa}(v)$. Now either $\gamma_{vj}=0$ or $v=j$, and in the latter case $\gamma_{iv}=0$ for all $i\neq v$ since $v$ is terminal,  in particular for $i\in B$.
Thus, we obtain $A\indep B\; [\Lambda^0_{A\cup B}]$.

Next, suppose that there is a pair of indices $(i,j)$ such that  
\[\gamma_{S i}=c\gamma_{S j}\neq 0_S\] for some $c>0$. Then, since $S=\mathrm{pa}(v)$, both $i$ and $j$ are ancestors of~$v$ and it holds that
\[\gamma_{vi}=\bigvee_{p\in pa(v)}\beta_{vp}\gamma_{pi}=c\bigvee_{p\in pa(v)}\beta_{vp}\gamma_{pj}=c\gamma_{vj}.\]
Hence the two rays (corresponding to $\gamma_{\cdot i}$ and $\gamma_{\cdot j}$) projected on $\mathrm{pa}(v)\cup \{v\}$ result in the same ray.
Choose any $p\in \mathrm{pa}(v)$ and consider the random vector $Y\sim\L_{R_{p,\eps}}$ with $R_{p,\eps}$ as in \eqref{eq:Rvep}.
It is clear from the above proportionality that conditioning on $Y_S$ renders $Y_v$ a constant.
Hence $Y_v$ is independent of $Y_B$ given $Y_S$ and the proof is complete in view of Theorem~\ref{thm:test_class} (ii).
\end{proof}

\begin{remark}
Beyond establishing the (DL) property for $\Lambda$, the proof of Theorem~\ref{thm:DAG}  also reveals a certain degeneracy of the recursive max-linear model. That is, conditioning on the parents identifies their common child. More precisely, the kernel $\nu_{\mathrm{pa}(v)}(y_{\mathrm{pa}(v)},\cdot)$ has a deterministic $v$-th component for $\Lambda_{\mathrm{pa}(v)}$-all values of $y_{\mathrm{pa}(v)}\neq 0_{\mathrm{pa}(v)}$.

The recursive max-linear models are not \emph{faithful} as a classical directed graphical model~\citep[see][]{amendola2022conditional}. That is, $X$ typically will satisfy additional conditional independence properties which can not be inferred from the graph.
This continues to be true for the respective $\Lambda$-graphical model as can be verified for instance by using the same diamond graph example as in~\citet{amendola2022conditional}.
\end{remark}

\subsection{Recursive infinitely divisible models}\label{sec:rec_id}
By replacing the maximum operation with the sum operation we obtain another basic example of a structural equation model:
\begin{align}\label{eq:Xid}
X_i=\sum_{j\in \pa(i)}\beta_{ij}X_j+\beta_{ii}\eps_{i}=\sum_{j \in V}\gamma_{ij}\eps_{j},\qquad i \in V,
\end{align}
where $\beta_{ij}\neq 0$ and $\gamma_{ij}$ is given by~\eqref{eq:gamma}  but with maximum replaced by the sum over the paths from $j$ to~$i$.
This is again a directed graphical model with respect to the underlying graph~\cite[Thm.\ 1.4.1]{pea2009}.
\citet{stableModels} assume that $\eps_i$ are mutually independent $\alpha$-stable random variables, making $X$ a multivariate $\alpha$-stable vector. More generally, we may consider $\eps_j$ to be independent infinitely divisible random variables, and so $X$ is infinitely divisible.

We remark that, while every univariate distribution is max-infinitely divisible, it is not necessarily infinitely divisible. This explains the need for the extra assumption on the innovation terms $\eps_j$ in order to employ the framework of infinitely divisible vectors and respective L\'evy measures.

\begin{lemma}\label{lem:Levyrays}
If all $\eps_j$ are infinitely divisible, then the L\'evy measure of the infinitely divisible $X$ defined in~\eqref{eq:Xid} is supported by the rays spanned by $\gamma_{\cdot j}$.
\end{lemma}
\begin{proof}
Consider the respective L\'evy processes $\eps_j(t)$ and $X(t)$, and recall that their independence implies that they do not jump simultaneously a.s.~\citep{sato}. Thus, every jump of $X(t)$ arises from a jump $\Delta\eps_j(t)$ of a single~$\eps_j(t)$, and therefore it has a form $\gamma_{\cdot j}\Delta\eps_j(t)$. These jumps belong to the ray spanned by $\gamma_{\cdot j}$ and the proof is complete.
\end{proof}

Thus, the L\'evy measure $\Lambda$ has the same form as the exponent measure $\Lambda$ corresponding to the recursive max-linear model, and so we arrive at the analogue of Theorem~\ref{thm:DAG}.
\begin{corollary}
Consider a recursive linear model in~\eqref{eq:Xid} defined on a DAG $\mathcal G$ such that all $\eps_j$ are infinitely divisible. Then the L\'evy measure $\Lambda$ corresponding to infinitely divisible vector~$X$ satisfies the directed local Markov property (DL) with respect to $\mathcal G$. It also satisfies the directed global Markov property (DG) with respect to $\mathcal G$ when the L\'evy measures of $\eps_j$ explode at~0 for all $j \in V$.
\end{corollary}
\begin{proof}
 In light of Lemma~\ref{lem:Levyrays}, the proof of Theorem~\ref{thm:DAG} can be repeated almost verbatim for the present situation (replace maximum by sum when expressing $\gamma_{vi}$) to establish (DL).
Note that the arguments do not depend on $\Lambda$ being supported by the positive orthant. 
Furthermore, we may establish \eqref{eq:infinite} similarly to the proof of Lemma~\ref{lem:epsilon_positive}, where we note that 
\[\Lambda(|y_v|>t)=\sum_{j \in V} \eta_j(|y\gamma_{vj}|>t),\] where $\eta_j$ is the L\'evy measure of $\eps_j$. If all $\eta_j$ are exploding at 0, the limit (as $t \to 0$) is either 0 or $\infty$. 
\end{proof}

\section{Relation to conditional independence in extremes}\label{sec:extremes}

In this section we discuss the implications of our results on several aspects in the study of sparsity in multivariate extreme value models. When studying limit theorems in this field, the sub-class of exponent measures $\Lambda$ on the space $\cE_+ = [0,\infty)^d\setminus \{0\}$ plays a central role.
To any such measure we can associate a max-infinitely divisible random vector $Z = (Z_i:i \in V)$ with distribution function
\begin{align}\label{maxid_cdf}
  G_\Lambda(x)=\exp\{-\Lambda(\cE_+ \setminus [0,x])\}, \quad x\in \cE_+.
\end{align}
Up to a marginal standardization, these distributions $G_\Lambda$ are the only possible limits of triangular arrays of maxima \citep{balres77,dHres77}.
In the following we assume that $\Lambda$ has the same marginal measures $\Lambda(u):=\Lambda( y_i > u) > 0$ for all $u \in (0,\infty)$ and $i \in V$, so that each  univariate survival function $\PP( Z_i > u)=1-\exp(-\Lambda(u))$ is asymptotically equivalent to $\Lambda(u)$ in the sense that
\begin{align}\label{eq:SurvivalAsympEquiv}
\lim_{u \to \infty} \frac{\PP( Z_i > u)}{\Lambda(u)}=1.
\end{align}
The strength of dependence between the largest observations of the components $i$-th and $j$-th component of $Z$ can be summarized by the \emph{extremal correlation coefficient}
\begin{align}\label{eq:chidef}
\chi_{ij} &:= \lim_{u \to \infty}  \mathbb P( Z_i > u \mid Z_j > u ) = \lim_{u \to \infty} \frac{\Lambda(y_i>u, y_j>u)}{\Lambda(u)} \in [0,1],
\end{align}
whenever the limit exists and where the second equation follows from a simple Taylor expansion.
Two dependence regimes are usually distinguished: if $\chi_{ij} > 0$ we speak of \emph{asymptotic dependence} between $i,j\in V$, and if $\chi_{ij} = 0$, we say the components are \emph{asymptotically independent}. Accordingly, this section is structured as follows.

In Section~\ref{sec:CIextremes} we focus on the case of asymptotic dependence, where a homogeneous measure $\Lambda$ can be used to fully describe the extremal dependence properties.
Asymptotically independent models are far more complex to describe mathematically. Section~\ref{sec:AI} provides a concrete example of such a max-infinitely divisible distribution that shows how our theory on density factorizations of general $\Lambda$ measures opens the door to a new theory of asymptotically independent graphical models.
Finally, Section~\ref{sec:sparsity} connects extremal graphical models to other sparsity notions and the field of concomitant extremes.

\subsection{Extremal graphical models}
\label{sec:CIextremes}

Let $X\geq 0$ be a random vector and assume for simplicity that it has heavy-tailed marginal distributions with common tail-index $\alpha>0$. 
In the case of asymptotic dependence, there are two different, but closely related classical approaches for describing the extremes of the multivariate distribution of $X$. 

The first approach is concerned with scale-normalized componentwise maxima \[a(n)^{-1} \max_{i=1,2,\dots,n} X^{(i)}\] of independent copies $X^{(i)}$ of $X$, where $a(n) > 0$. The only possible limit laws of such maxima as $n \uparrow \infty$ are max-stable with distribution function~\eqref{maxid_cdf},
where the exponent measure $\Lambda$ satisfies \eqref{eq:Lambda} and is $-\a$-homogeneous as in~\eqref{def:homog}. In particular, the distribution $G_\Lambda$ has necessarily $\a$-Fr\'echet marginals.

The second approach studies the distribution of the scale-normalized exceedances
\[ u^{-1} X \,\Big\vert\, \max_{j=1,\dots,d} X_j  > u\]
of the random vector $X$, conditioning on the event that at least one component $X_i$ exceeds a large threshold $u$. The only possible limits of these peaks-over-threshold as $u\uparrow \infty$ are  multivariate Pareto distributions \citep{roo2006}, whose probability laws are induced by a homogeneous measure $\Lambda$ on the (non-rectangular) set $\cL = \cE_+ \setminus [0,1]^d$ and take the form $\p_{\cL}(\D y)=\Lambda(\D y)/\Lambda(\cL)$.
%%%\{x \geq 0: \max_{i=1,\dots,d}(x_i) > 1\}=
%%%\[\p_{\cL}(dy)=\frac{\Lambda(dy)}{\Lambda(\cL)}.\]

An apparent connection between these two approaches is the exponent measure $\Lambda$, which characterizes distribution functions of both, multivariate max-stable distributions and multivariate Pareto distributions. In fact, the connection is due to a fundamental limiting result, which links the two approaches via regular variation. This has been neatly summarized and extended in \citet[Thm.~1]{dombryribatet2015}. We only recall here that
\begin{align}\label{MEVD}
a(n)^{-1} \max_{i=1,2,\dots,n} X^{(i)} \to G_\Lambda \quad \iff \quad u^{-1} X \,\Big\vert\, \max_{j=1,\dots,d} X_j  > u \to \p_{\cL},
\end{align}
where $a(n)$ is the $(1-1/n)$-quantile of $\max_{j=1,\dots,d} X_j$. In addition each of these limiting statements is equivalent to the \emph{regular variation} of the random vector $X$ with limiting measure $\Lambda$ in the sense that  the measure $u \p(a(u)^{-1} X \in \cdot )$
converges vaguely to $\Lambda$ on $\overline \cE_+ = [0,\infty]^d\setminus\{0\}$, denoted by $X \sim {\rm RV}^+_\alpha(\Lambda)$. In other words, for a given $X \sim {\rm RV}^+_\alpha(\Lambda)$ the corresponding max-stable distribution $G_\Lambda$ and multivariate Pareto distribution $\p_{\cL}$ are associated via the same exponent measure $\Lambda$.

It seems therefore natural to approach conditional independence for multivariate extremes in terms of the exponent measure $\Lambda$. 
We start by linking known versions of {extremal (conditional) independence} to the $\Lambda$-based conditional independence introduced here in Definition~\ref{def:CILambda}, where $\Lambda$ assumes now the role of the associated exponent measure. Recall that such exponent measures naturally satisfy our key explosiveness assumption~\eqref{eq:infinite} due to their homogeneity; see Remark~\ref{rk:A1}. When writing $(X_A,X_B) \in {\rm RV}^+_\alpha(\Lambda)$ below, we tacitly assume that $(A,B)$ constitutes a partition of $V$, and analogously $(X_A,X_B,X_C) \in {\rm RV}^+_\alpha(\Lambda)$ implies $(A,B,C)$ being a partition of $V$.

%%
%%\begin{theorem}
%%Let $X_1,X_2,\dots$ be a sequence of independent copies of a random vector $X$ with values in $\cE_+$ a.s.\ and $a(u)=\inf\{ t \geq 0 \,:\, \p(\lVert X \rVert_\infty \leq t) \leq 1-1/u )\}$. Then the following statements are equivalent:
%%\begin{enumerate}[label=(\roman*)]
%%\item $X$ is regularly varying in the sense that the measure $u \p(a(u)^{-1} X \in \cdot )$ converges vaguely to $\Lambda$ on $\overline \cE_+ = [0,\infty]^d\setminus\{0\}$.
%%\item As $n \to \infty$, the law of rescaled maxima $a(n)^{-1} \max_{i=1,2,\dots,n} X_i$ converges weakly to the max-stable law with distribution function $G(x)=\exp(-\Lambda(\cE_+ \setminus [0,x]))$, $x \geq 0$.
%%\item As $u \to \infty$, the law of the rescaled exceedance $u^{-1} X \vert \lVert X \rVert_\infty > u$ converges weakly to the multivariate Pareto distribution $\p_{\cL}(dy)=\Lambda(dy)/\Lambda(\cL)$.
%%\end{enumerate}
%%\end{theorem}

\begin{definition}[Traditional extremal independence] \label{def:TradEI}
Let $(X_A,X_B) \in {\rm RV}^+_\alpha(\Lambda)$, then $X_A$ and $X_B$ are said to exhibit \emph{extremal independence}, denoted by
\[X_A \,{\perp_{\rm tr}}\, X_B,\]
if we have classical independence $Z_A \indepp Z_B$ for $Z \sim G_\Lambda$ .
\end{definition}

A first observation is that the notion of traditional extremal independence coincides with our new notion of independence with respect to the exponent measure $\Lambda$, since both statements are equivalent to $\Lambda(y_A \neq 0_A, y_B\neq 0_B)=0$; see~Proposition~\ref{prop:indep} and, for instance, \citet{stro2020}.

\begin{corollary}
Let $(X_A,X_B) \in {\rm RV}^+_\alpha(\Lambda)$, then $X_A \perp_{\rm tr} X_B$ if and only if $\indepAB$.
\end{corollary}

\begin{remark}
If we consider more generally $Z\sim G_\Lambda$, where $\Lambda$ is not necessarily homogeneous, but still satisfies the basic assumptions~\eqref{eq:Lambda} and \eqref{eq:infinite0}, then also $Z_A \indepp Z_B$ if and only if $\indepAB$.
\end{remark}

One might be tempted now to extend Definition~\ref{def:TradEI} to extremal conditional independence in a similar manner. However, it was shown in \citet{papastathopoulos2016conditional} that this leads to a flawed approach if one considers models $G_\Lambda$ with a positive continuous density. Instead two other routes have been successfully pursued. \citet{gis2018} study spectrally discrete structural models in the maxima setting where $G_\Lambda$ does not admit a density; see also Section~\ref{sec:claudia}. In contrast, \citet{eng2018} study extremal conditional independence for multivariate Pareto distributions and propose in essence the following definition.

\begin{definition}[Extremal conditional independence based on \citet{eng2018}] \label{def:TradCEI}
Let $(X_A,X_B,X_C) \in {\rm RV}^+_\alpha(\Lambda)$, then $X_A$ and $X_B$ are said to exhibit \emph{extremal conditional independence} given $X_C$, denoted by
\[X_A \perp_e X_B \mid X_C,\]
if we have classical conditional independence $Y_A \indepp Y_B \,|\, Y_C$ for $Y \sim \p_{R_{v,1}}=\Lambda(\D y)/\Lambda(R_{v,1})$ for all $v \in V$.
\end{definition}

\begin{remark}
  \citet{eng2018} only consider the case where $\Lambda$ has a Lebesgue density on $\cE_+$ and therefore no mass on any of the sub-faces $\cE_D$, $D\subset V$ for $D\neq V$. They show in this case that it suffices to verify classical conditional independence in the above definition for one $v \in C$ in the conditioning set.
  Our Lemma~\ref{lem:single_v} in the Appendix shows that indeed any $v \in V$ is sufficient, even if $v \in V \setminus C$.
\end{remark}

Since the exponent measure $\Lambda$ is homogeneous, it is an immediate consequence of Theorem~\ref{thm:test_class} (see also Remark~\ref{rk:test_class}) that the extremal conditional independence from Definition~\ref{def:TradCEI} 
and our $\Lambda$-based conditional independence coincide. To be precise, we may formulate this finding as follows. In particular, we would like to stress that it is valid for any homogeneous exponent measure $\Lambda$ without further assumptions; $\Lambda$ need not have a Lebesgue density and $C=\emptyset$ is allowed.

\begin{corollary} \label{cor:ECI}
Let $(X_A,X_B,X_C) \in {\rm RV}^+_\alpha(\Lambda)$, then 
 \[X_A \perp_e X_B \mid X_C \quad \iff \quad \indepABC.\]
 In particular, this includes the case $C=\emptyset$.
\end{corollary}

%%\begin{proof}
%%The proof is an immediate consequence of Theorem~\ref{thm:test_class}, cf.\ also Remark~\ref{rk:test_class}, since $\Lambda$ is a homogeneous measure.
%%\end{proof}

Originally, \citet{eng2018} excluded the independence case $C=\emptyset$, as their approach is based on working with a Lebesgue density of $\Lambda$ and does not allow for positive mass on the lower-dimensional subsets $\cH_A \cup \cH_B$; in view of~Proposition~\ref{prop:indep} we know that mass on such subsets is crucial for the independence case. In our new setup, Definition~\ref{def:TradCEI} and Corollary~\ref{cor:ECI} naturally encapsulate the case $C=\emptyset$ and it is furthermore in line with the traditional notion of extremal independence. We have seen now that for $(X_A,X_B) \in {\rm RV}^+_\alpha(\Lambda)$
\[ X_A \perp_{\rm tr} X_B \qquad \iff \qquad X_A \perp_{e} X_B \qquad \iff \qquad A \indep B \,\, [\Lambda].\]
The first equivalence has already been shown in \citet{stro2020}.

In Corollary~\ref{cor:ECI} it is also possible to consider discrete spectral measures,  which builds a bridge to the structural max-linear models~\eqref{eq:recmaxlin} from \citet{gis2018}; see Section~\ref{sec:claudia}. Theorem~\ref{thm:DAG} shows that we recover all the conditional independencies that are encoded in the structural model equations also within our $\Lambda$-based notion of conditional independence.

\begin{remark}
Instead of restricting oneself to the cone $\cE_+ = [0,\infty)^d \setminus \{0\}$, we can also consider a random vector $X$ on $\cE$ that is regularly varying in the sense that there exists a scaling function $a(u)$, such that $u\p(a(u)^{-1}X \in \cdot)$ converges vaguely to a measure $\Lambda$. 
Definition~\ref{def:TradCEI}  and Corollary~\ref{cor:ECI} also carry over naturally to the situation when $(X_A,X_B,X_C)\in {\rm RV}_\alpha(\Lambda)$.
\end{remark}

Extremal graphical models have been applied to assess flood risk \citep{roettger2021, ase2021}, financial risks \citep{engelke2022a} and large delays in flight networks \citep{hen2022}.
The parametric family of H\"usler--Reiss distributions \citep{Husler1989} can be seen as the analogue of the multivariate Gaussian distribution in extremes, and it is so far the only one used in these kind of statistical applications.
One advantage of our new theory here in the context of extremal graphical models is that we lay the foundation to construct much more general models than in \citet{eng2018}, overcoming key limitations that were pointed out in its discussion part. In particular, we do not require Lebesgue densities and the exponent measure $\Lambda$ can have mass on lower-dimensional sub-faces of $\cE_+$. Importantly, the graph can therefore have unconnected components, which corresponds to asymptotically independent groups of variables.
The following example provides a generalization of the widely used H\"usler--Reiss distributions that makes use of all of the above features.

\begin{figure}[hbt]
\tikzstyle{vertex}=[circle,fill=gray!10,draw=black,thin,minimum size=12pt,inner sep=0pt]
\tikzstyle{edge} = [draw,thick,-]

\mbox{\small   
\begin{tikzpicture}[scale=1]
    % vertices
    \foreach \pos/\name in {
    (1.43,-1)/1,(2.43,0)/2,(2.1,1.33)/3, (4,0)/4,
    (5.66,0)/5,(6.5,-1)/8,(8.16,-1)/9, (6.5,1)/6, (7.33,0)/7, (9.33,0)/10}
        \node[vertex] (\name) at \pos {\scriptsize $\name$};
    % edges
    \foreach \source / \dest  in {1/2, 2/3, 2/4, 5/6, 6/7, 7/8, 7/9}
        \path[edge] (\source) -- (\dest);
\end{tikzpicture}
}

         \caption{\small  
A forest with cliques  $\cC=\{ \{1,2\}, \{2,3\}, \{2,4\}, \{5,6\}, \{6,7\}, \{7,8\}, \{7,9\}, \{10\} \}$ (edges and isolated singleton \{10\}) separated by separators  $\cS=\{\{2\}, \{2\}, \emptyset, \{6\}, \{7\}, \{7\}, \emptyset \}$. The empty set appears twice in $\cS$, as the forest consists of three connected components. For each node $i \in V=\{1,\dots,10\}$, the singleton $\{i\}$ appears $d_i-1$ times in $\cS$, where $d_i$ is the degree of $i$ in the forest.
}
  \label{fig:forest}
\end{figure}
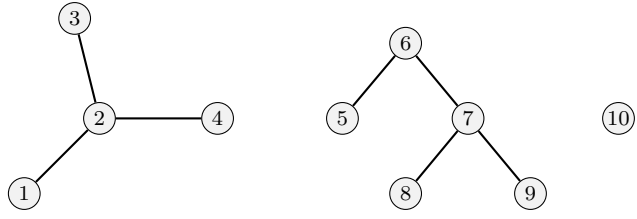

 \begin{example}\label{HR_forest}

  Let $V = \{1,\dots, d\}$ and suppose that $\cG = (V,E)$ is a forest, that is, a graph where each connected component is a tree. The (maximal) cliques of such a graph are its edges together with its isolated points, cf.~also Figure~\ref{fig:forest}, i.e.~$\cC=E \cup \mathcal I$, where $\mathcal I$ consists of isolated singletons $\{i\}$. The separator multiset $\cS$ is also comprised of two kinds of sets; on the one hand empty sets separating connected components and on the other hand singletons $\{i\}$ that are separating edges. The multiplicity of each singleton $\{i\}$ in the multiset $\cS$ is $d_i-1$ if $d_i$ denotes the degree of $i$ in the forest. Similarly, the empty set appears once less than the number of trees.
  
  We work with the same setup as in the example in Section~\ref{subsec:CI3constr} and assume a dominating product measure of the form~\eqref{product_measure}.
  As building blocks for any edge $\{i,j\} \in E$ we use the bivariate H\"usler--Reiss densities~\citep[e.g.,][]{Engelke2015} given by
  \begin{align}\label{eq:HRbivdensity}
 \kappa_{ij}(y_{ij})=\frac{1}{\sqrt{2\pi\Gamma_{ij}}}y_i^{-2}y_j^{-1}\exp\{-[\log(y_j/y_i)+\Gamma_{ij}/2]^2/(2\Gamma_{ij})\},\,\,\, y_{ij} \in (0,\infty)^2,
\end{align}
where $\Gamma_{ij} \in (0,\infty)$ are the corresponding dependence parameters on the edges as in Example~\ref{ex:HR}; note that these densities are $-3$-homogeneous and have univariate marginal densities $m(y) = y^{-2}$.
As in Section~\ref{subsec:CI3constr}, we let $\lambda_{ij}$ be the corresponding densities on $\cE_+^{ij}$ with masses on the axes as in~\eqref{ext_density} with mixture parameters $p_{ij}\in (0,1]$, $\{i,j\} \in E$. Recall that the  marginal densities of $\lambda_{ij}$ are by construction $m(y_i)$ and $m(y_j)$, respectively.

The exponent measure density of a H\"usler--Reiss forest is defined as
\begin{align}\label{forest_dens}
  \lambda(y) = \frac{\prod_{\{i,j\} \in E} \overline \lambda_{ij}(y_i, y_j)}{\prod_{\{i\} \in \mathcal S} \overline \lambda_i(y_i)} \prod_{\{i\} \in  \mathcal I} \overline \lambda_i(y_i),
\end{align}
for $\mu$-almost all $y\in \cE_+ \setminus \mathcal Z(\mathcal G)$, and it is implied that $\lambda(y)=0$ otherwise. 
By construction, the corresponding measure $\Lambda$ on $\cE_+ =[0,\infty)^V\setminus\{0_V\}$ is $-1$-homogeneous with marginal densities $\lambda_i(y_i)=m(y_i)=y_i^{-2}$. By Theorem~\ref{thm:decomposable} it is an undirected graphical model on the forest $\cG$.

We write $i \leftrightarrow j$ if $i$ and $j$ are connected, and $i \nleftrightarrow j$ otherwise. In the former case, we denote by $\mathrm{path}_{ij}$ the set of all edges on the unique shortest path between $i$ and $j$ in $\cG$. We can explicitly express the extremal correlation between arbitrary nodes $i,j\in V$ as
\[ \chi_{ij} =
\begin{cases}
\big\{2 - 2\Phi\big(\sqrt{\Gamma_{ij}} /2\big)\big\}{\displaystyle \prod_{(s,t) \in \mathrm{path}_{ij}} p_{st}  }  & \text{for } i\leftrightarrow j,\\
0  & \text{for } i\nleftrightarrow j,
\end{cases}
\]
where $\Phi$ is the standard normal distribution function, and the $\Gamma_{ij} = \sum_{(s,t) \in \mathrm{path}_{ij}} \Gamma_{st}$ are the tree-completed H\"usler--Reiss coefficients for all $i\leftrightarrow j$~\citep[e.g.,][]{eng2020, asenova2021extremes}; see Appendix~\ref{HR_forest_proof} for the proof. This means that nodes in different connected components are asymptotically independent.
\end{example}

\subsection{Asymptotic independence}\label{sec:AI}

In the previous section we have discussed the implications of our results for asymptotically dependent models, where a homogeneous measure  $\Lambda$ characterizes
the extremal limits.
The theory of this paper is much more general and allows, for the first time, to define non-trivial
$\Lambda$-based graphical models in the regime of asymptotic independence.
In general, there is no unified way of describing the dependence structures of all asymptotically independent distributions, and different approaches exist, including hidden regular variation \citep{res2002}, conditional extreme value models \citep{HeffernanTawn2004} or scale mixtures \citep{Wadsworthetal2017, eng2018a}.

Another possibility to construct distributions that exhibit asymptotic independence is given by a max-infinitely divisible distribution $Z$ with distribution function~\eqref{maxid_cdf} with suitable non-homogeneous exponent measure $\Lambda$; here we also assume equal marginal measures with $\Lambda(u)>0$ for all $u\in(0,\infty)$.
 In this case, the extremal correlations in~\eqref{eq:chidef} are not informative since they satisfy $\chi_{ij}=0$ for all $i,j\in V$.
Following the approach of \citet{LedTawn96,LedTawn97}, a refined \emph{residual tail dependence coefficient} $\eta_{ij}\in (0,1]$ can be defined through
\begin{align}\label{eq:LT-residual-tail-dep-coeff}
  \mathbb P( Z_i > u, Z_j > u ) &= \ell_{ij} (\PP(Z_i>u)^{-1}) \PP(Z_i>u)^{1/\eta_{ij}},
\end{align}
where $\ell_{ij}$ is a slowly varying function; see~Appendix~\ref{app:GaussianExample} for a review of some fundamental properties of slowly varying functions. Due to \eqref{eq:SurvivalAsympEquiv},
the relation \eqref{eq:LT-residual-tail-dep-coeff} is equivalent to
 \begin{align}\label{eq:residual-tail-dep-coeff}
  \mathbb P( Z_i > u, Z_j > u ) &= \widehat \ell_{ij} (\Lambda(u)^{-1}) \Lambda(u)^{1/\eta_{ij}}
\end{align} 
for a  slowly varying function $ \widehat \ell_{ij}$, cf.~also~Lemma~\ref{lemma:SV}.
The coefficient $\eta_{ij}$ characterizes the decay rate of the joint exceedance probability relative to the univariate survival function.
We speak about \emph{positive} and \emph{negative extremal association} between $Z_i$ and $Z_j$ if $\eta_{ij} \in (1/2, 1]$ and $\eta_{ij}\in(0,1/2)$, respectively, and about \emph{near independence} if $\eta_{ij} = 1/2$.
For max-infinitely divisible distributions we cannot obtain negative extremal association and 
the joint survival function $\PP( Z_i > u, Z_j > u )$ is asymptotically equivalent to the joint survival function of the exponent measure $\Lambda(y_i>u, y_j>u)$ for positive extremal association. A proof for this result is given in Appendix~\ref{app:etaOneHalf}.

\begin{lemma}\label{lemma:etaOneHalf}
  Let $Z$ be a max-infinitely divisible distribution with residual tail dependence coefficient $\eta_{ij}$ as in \eqref{eq:LT-residual-tail-dep-coeff} or \eqref{eq:residual-tail-dep-coeff}. Then $\eta_{ij} \geq 1/2$ for all $i,j \in V$.
   If $\eta_{ij} > 1/2$ we have that
  \begin{align}\label{eq:joint-survival-tail-equiv}
  \lim_{u \to \infty}\frac{\Lambda(y_i>u, y_j>u)}{\PP( Z_i > u, Z_j > u )} =1
\end{align} 
and   
\begin{align}\label{eq:eta-viaLambda}
  \Lambda(y_i>u, y_j>u) &= \widetilde{\ell}_{ij} (\Lambda(u)^{-1}) \Lambda(u)^{1/\widetilde{\eta}_{ij}},
\end{align} 
for a slowly varying function $\widetilde \ell_{ij}$ and $\widetilde{\eta}_{ij}=\eta_{ij}$. 

Conversely, if \eqref{eq:eta-viaLambda} holds for some $\widetilde{\eta}_{ij} > 1/2$, this entails \eqref{eq:joint-survival-tail-equiv} as well as \eqref{eq:LT-residual-tail-dep-coeff} and \eqref{eq:residual-tail-dep-coeff} with $\eta_{ij}=\widetilde{\eta}_{ij}$.
\end{lemma}

\begin{remark}
  It is known that max-infinitely divisible distributions are positively associated in the usual sense \citep[Prop.~5.29]{res2008}. The above result confirms that this translates into extremal positive association. In many practical examples we have indeed $\eta_{ij} > 1/2$, however  for $\eta_{ij} = 1/2$ (the near independence case) one has to be careful. Consider for instance the situation where $Z_i$ and $Z_j$ are independent. 
 Then the limit in \eqref{eq:joint-survival-tail-equiv} will be zero (instead of 1) and the joint survival functions are no longer tail-equivalent.
\end{remark}

An example for an exponent measure that induces an asymptotically independent max-infinitely divisible distribution $Z$ exhibiting positive association is the bivariate Gaussian type measure $\Lambda_\rho$ in Example~\ref{ex:gauss} with correlation $\rho > 0$. By construction, in this case $(Z_1,Z_2)$ is a bivariate log-Gaussian distribution, which is asymptotically independent with residual tail dependence coefficient
\[ \eta = \frac{1+\rho}{2}.\]
In order to construct examples in higher dimensions $d>2$ with graphical structure, we can draw on the density factorization in Theorem~\ref{thm:decomposable}.

Here we consider the most basic construction on the chain graph $\cG=(V,E)$ with nodes $V=\{1,2,3\}$ and edges $E = \{(1,2), (2,3)\}$. Following the example in Section~\ref{subsec:CI3constr} with no mass on the axes, we consider bivariate measures $\lambda_{12}$ and $\lambda_{23}$ that are equal to the densities of the log-Gaussian exponent measures $\Lambda_\rho$ with correlations $\rho_{12},\rho_{23} > 0$, respectively. We note that they share the same univariate marginal density $\lambda_1=\lambda_2=\lambda_3$. Setting
\begin{align}\label{3d_example}
  \lambda(y) = \lambda_{12}(y_{12}) \lambda_{23}(y_{23})/  \lambda_2(y_2), \quad y \in \mathcal E_+,
\end{align}
we obtain a trivariate density that induces a valid exponent measure $\Lambda$ and a corresponding
max-infinitely divisible distribution $Z = (Z_1, Z_2, Z_3)$. By construction,   the exponent measure $\Lambda$ with density~\eqref{3d_example} is an undirected graphical model on the chain graph $\cG$.

\begin{proposition}\label{prop:DG}
 The max-infinitely divisible distribution $Z$ 
  with exponent measure $\Lambda$ with density~\eqref{3d_example}
 is asymptotically independent with residual tail dependence coefficients
  \[ \eta_{12} = \frac{1+\rho_{12}}{2}, \qquad \eta_{23} = \frac{1+\rho_{23}}{2}, \qquad \eta_{13} = \frac{1+\rho_{12}\rho_{23}}{2}.\]
\end{proposition}

We provide a proof of Proposition~\ref{prop:DG} in Appendix~\ref{app:GaussianExample}.
This result resembles the corresponding residual joint asymptotic tail behavior that we would also have detected in  a trivariate (log-)Gaussian distribution with the same graphical structure. However, what is quite subtle in this context, is that the distribution of $Z$ is in fact not(!) log-Gaussian,  but only arises by combining two bivariate log-Gaussian exponent measure densities.

 This construction principle can be extended to more general graphs such as trees. Moreover, any valid exponent measure density of a bivariate max-infinitely divisible distribution can be used; for more examples we refer to \citet{hus2021}.
This provides a powerful route for statistical modeling of high-dimensional data which exhibit asymptotic independence to be addressed in future research.

\subsection{Sparsity}\label{sec:sparsity}
Detecting meaningful and interpretable sparse structures in multivariate extremal dependence is an active field of research. Most approaches are formulated in terms of the exponent measure $\Lambda$ of the respective limit models in~\eqref{MEVD}. The review article \citet{eng2021} distinguishes three groups of methods and formalizes the underlying notions of sparsity. In this section we show that our theory reveals an intimate connection between two of them. First, let us recall the relevant sparsity notions.

The research on \emph{concomitant extremes} is concerned with the identification of elementary sub-faces $\mathcal E_D$, $\emptyset \neq D\subset V$, that the exponent measure charges with mass, that is, $\Lambda(\mathcal E_D) > 0$. In this context, a sub-face is defined as
\begin{align}\label{eq:face2}
\mathcal E_D = \{y \in \cE_+ \,:\, y_{D} > 0, y_{V\setminus D} = 0_{V \setminus D}\}.
\end{align}
This is of interest since in the case $\Lambda(\mathcal E_D) > 0$, underlying variables indexed by the set $D$ can be jointly extreme.
The set $\mathcal E_+$ can be partitioned in the $2^d-1$ disjoint sub-faces $\mathcal E_D$ for all non-empty subsets $D\subset V$. The sparsity notion 2(a) in \citet{eng2021} requires that there is only a small number of groups of variables that can be concomitantly extreme, that is,
\begin{align}\label{sparse2}
  \left|\{ D\subset V : \Lambda(\mathcal E_D) > 0 \} \right| \ll 2^d - 1.
\end{align}
The sparsity notion 2(b) then states that the maximal cardinality of such groups is much smaller than the dimension $d = |V|$.
Many statistical approaches have been proposed to tackle the difficult problem of determining those sets \citep[e.g.,][]{sim2020, chi2017, chi2019, mey2019}.

A more classical notion of sparsity is in terms of conditional independence relations. A sparse graph $\mathcal G = (V,E)$, directed or undirected, will induce a distribution that can be explained by lower-dimensional objects; see Sections~\ref{sec:undirected} and \ref{sec:directed} here. Accordingly, sparsity notion 3 in~\citet{eng2021} is the requirement that a  graphical model on $\cG=(V,E)$ is sparse if the number of edges $|E|$ is much smaller than the number of all possible edges $d(d-1)/2$, that is,
\begin{align}\label{sparse3}
  |E| \ll d^2. 
\end{align}

Methods for sparsity in terms of concomitance as in~\eqref{sparse2} on the one hand, and in terms of sparse graphs as in~\eqref{sparse3} on the other hand, have so far not been considered together. In the present conditional independence framework for an exponent measure $\Lambda$, we can now establish a link between them.
More precisely, we draw on Corollary~\ref{cor:mass0}, which gives insight how the topology of the graph $\cG$ underlying an undirected extremal graphical model limits the number of potentially charged sub-faces  $\mathcal E_D$. It implies that $\Lambda(\mathcal E_D) = 0$ for any index set $D$ such that $\cG$ restricted to $D$ is disconnected. This gives an \emph{a priori}  upper bound on the number of charged sub-faces.

\begin{corollary}\label{cor:sparse}
   For an exponent measure $\Lambda$ corresponding to an extremal graphical model on the undirected graph $\cG=(V,E)$, an upper bound on the number of sub-faces charged with mass by $\Lambda$ is given by
  \begin{align}\label{sparse_bound}
    \left|\{ D\subset V :\, \Lambda(\mathcal E_D) > 0 \} \right| \leq \left|\{ D \subset V:\, \cG_D \text{ is connected}  \} \right|,
  \end{align}
  where the right-hand side is the cardinality of all connected sub-graphs of $\cG$. Moreover, if $\Lambda(\cE_D)>0$, then the cardinality $|D|$ is bounded from above by the number of nodes in the largest connected component of $\cG$. 
\end{corollary}

 Hence, sparsity in the sense of few edges in the graph (requirement \eqref{sparse3})
 may also enforce sparsity in terms of fewer potential concomitant extremes (requirement \eqref{sparse2}). 
However, the precise topology matters as well, as the following example illustrates.

 \begin{example}
Consider the three graphs with $d=|V| = 10$ nodes in Figure~\ref{fig:graphs}.  Table~\ref{tab:counts} displays the counts of their edges and connected sub-graphs.
The graph on the left corresponds to asymptotic independence between all $d$ nodes. There are no edges, i.e.\ the graph is sparse in the sense of \eqref{sparse3}. Trivially, we can only have $d$ sub-faces with mass in this example, which is certainly also a sparse model in the sense of concomitant extremes, i.e., \eqref{sparse2}, since $d  \ll 2^d - 1$.
The ring graph in the center is a connected, non-decomposable graph with $d$ nodes and $d$ edges, hence also \eqref{sparse3}. Using simple combinatorics, one can see that  the number of connected sub-graphs of a ring is $(d-1)d+1 \ll  2^d - 1$ and therefore by Corollary~\ref{cor:sparse} it also induces a sparse model for concomitance, cf.~\eqref{sparse2}.

However, the star graph on the right-hand side of Figure~\ref{fig:graphs} shows that the number of edges, here $d-1$, is not a sufficient indicator of whether only few sub-faces are charged with mass. Indeed, the star graph has $2^{d-1} +d -1$ connected sub-graphs, which is the largest number among any tree structure \cite[][Theorem 3.1]{sze2005}. While this tree model is simple in terms of the graph sparsity notion \eqref{sparse3}, the corresponding exponent measure $\Lambda$ may have mass on more than half of its sub-faces.
\end{example}

\begin{table}[htb]
\caption{\small Counts of edges and connected sub-graphs for graphs with $d=10$ nodes displayed in Figure~\ref{fig:graphs}. The fully connected graph (not displayed in Figure~\ref{fig:graphs}) corresponds to the maximal possible counts.}\label{tab:counts}

{\small
\renewcommand{\arraystretch}{1.2}
\begin{tabular}{|l|c|c|c|c|c}
\hline
\cellcolor{gray!30}
type of graph &  empty & ring & star & 
\cellcolor{gray!5} fully connected  \\
\hline
\cellcolor{gray!30}
\# edges  & $0$ & $d=10$ &  $d-1=9$ & 
\cellcolor{gray!5} $d(d-1)/2=45$ \\
\cellcolor{gray!30}
\# connected sub-graphs  & $d=10$ &  $(d-1)d+1=91$ &  $2^{d-1}+d-1 = 521$ & \cellcolor{gray!5} $2^d-1=1023$\\
\hline
\end{tabular}
}
\end{table}

From these examples it can be seen that connectivity properties of the graph $\cG=(V,E)$ are more important for the number of sub-faces with $\Lambda$-mass than the cardinality of the edge set $|E|$. Since our extended definition of extremal graphical models for $\Lambda$ allows for unconnected components in a graph, this enables particularly sparse models in the sense of concomitant extremes.

\begin{figure}
\includegraphics[width=0.3\textwidth]{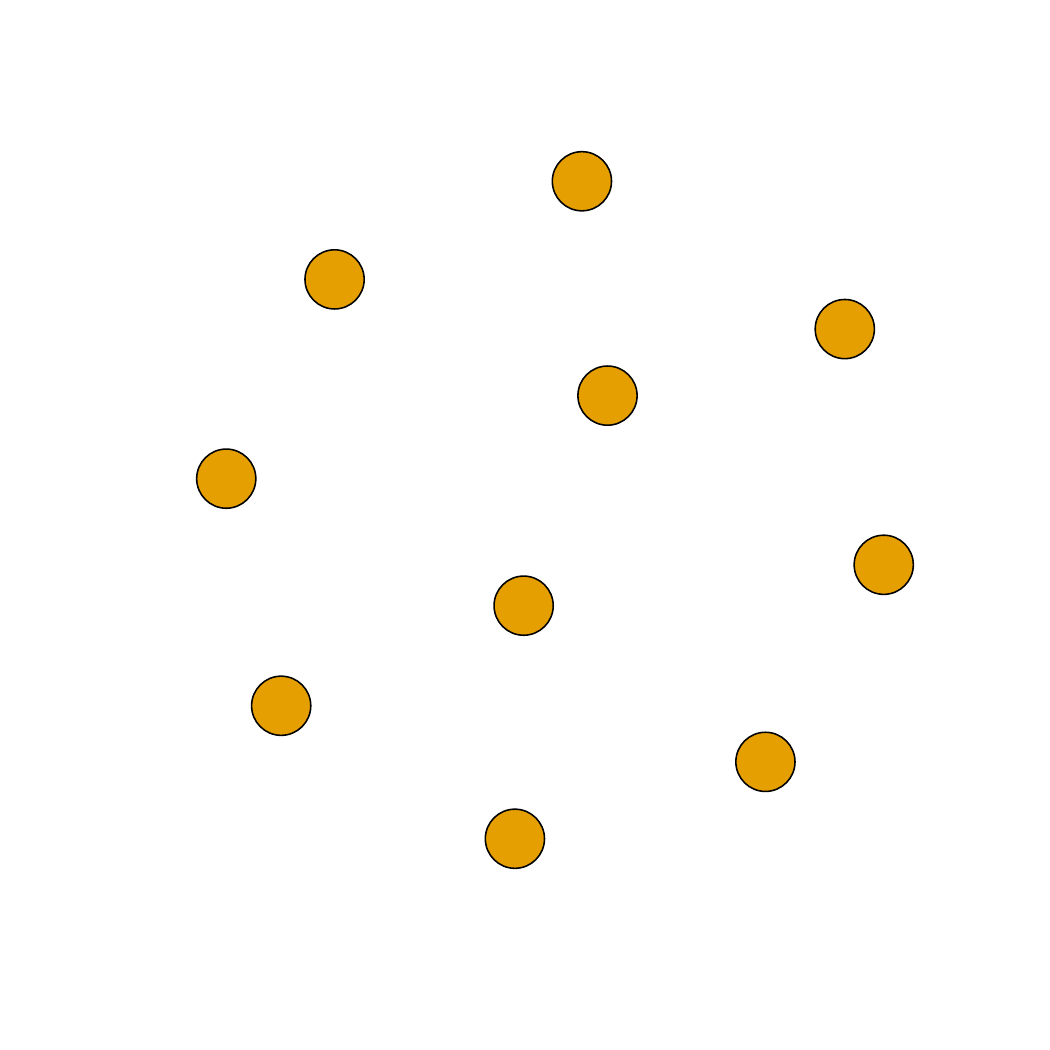}
\includegraphics[width=0.3\textwidth]{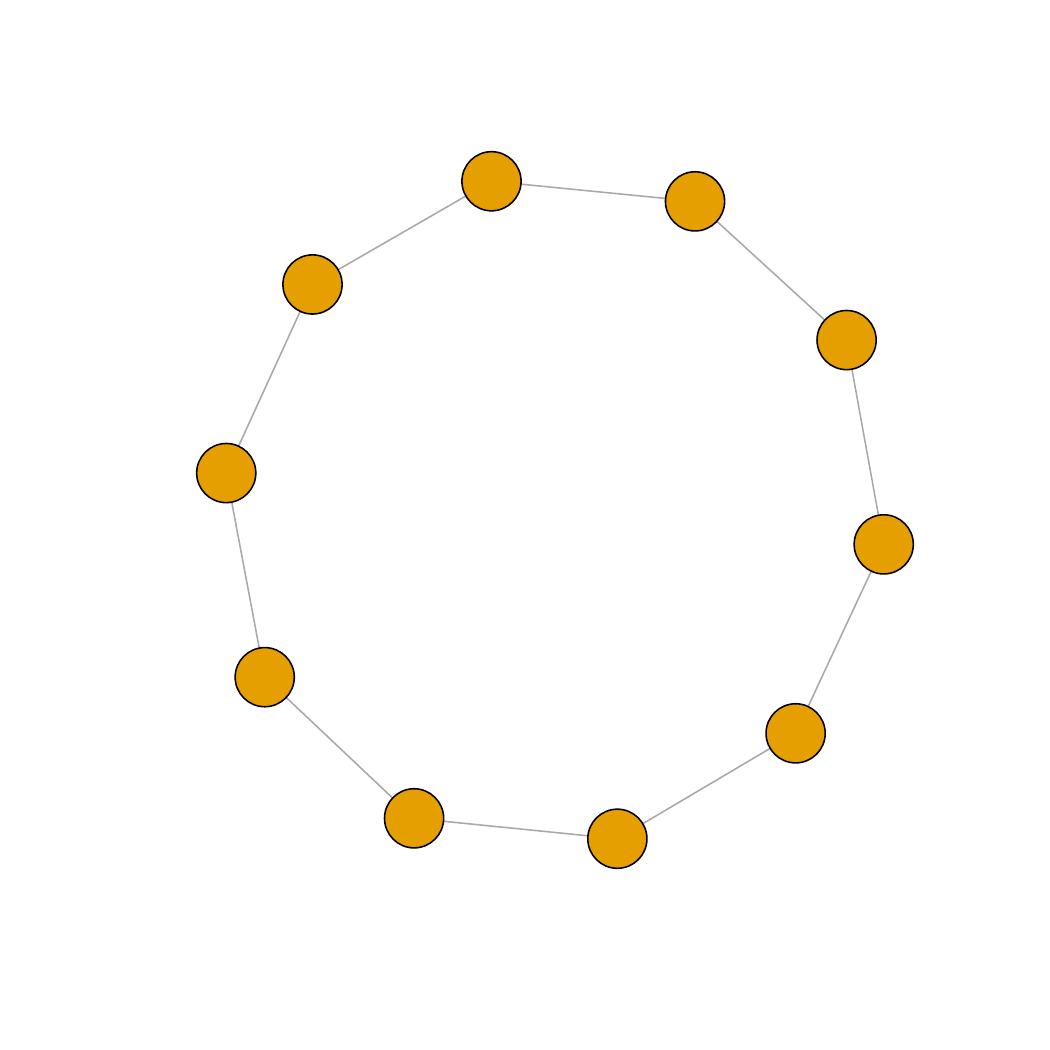}
\includegraphics[width=0.3\textwidth]{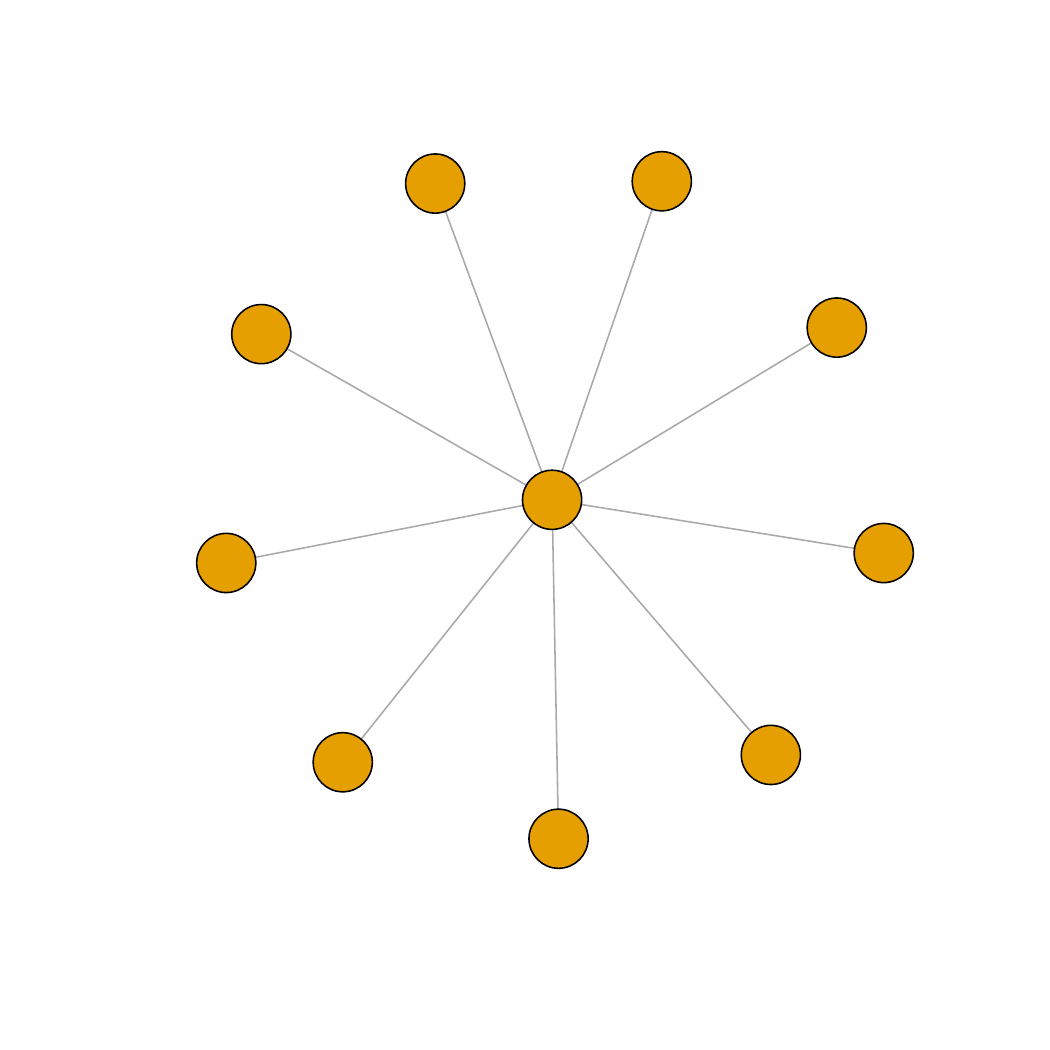}
\caption{Three graphs with $d = 10$ nodes: empty (left), ring (center) and star (right) graph. }\label{fig:graphs}
\end{figure}

\section{Outlook on applications to L\'evy processes}\label{sec:outlook}

The theory of conditional independence and graphical models developed in this paper applies to general infinite measures $\Lambda$. We concentrated on discussing the implications of our results in the field of extreme value theory, since this field has been very active recently; see Sections~\ref{sec:claudia} and \ref{sec:extremes}.

Another important instance of a measure $\Lambda$ studied in this paper is the L\'evy measure of an infinitely divisible random vector $X$ as introduced in Section~\ref{sec:motivation}. In this case, $\Lambda$ appears in the characteristic function of $X$ and it is not obvious which probabilistic interpretation a conditional independence statement $\indepABC$ has. In the particular case where the L\'evy measure $\Lambda$ concentrates on a finite number of rays this statement implies the classical conditional independence $X_A \indepp X_B \mid X_C$; see Section~\ref{sec:rec_id}. For general L\'evy measures this turns out to not be true.

Ongoing research shows that in this case a different and surprisingly intuitive probabilistic interpretation is available in terms of L\'evy processes. A L\'evy process $\{X(t): t\geq 0\}$ with values in $\mathbb R^d$ is a stochastic process with independent and stationary increments \citep[e.g.,][]{bertoin1996levy, sato}. Moreover, there is a one-to-one mapping that takes an infinitely divisible random vector $X$ with characteristic triplet $(0, 0, \Lambda)$ to the associated L\'evy process satisfying $X(1) \stackrel{d}{=} X$; here we assume the drift and Gaussian component in Section~\ref{sec:motivation} to be zero since they are well-understood.
It turns out that our $\Lambda$ conditional independence characterizes conditional independence on the level of the sample paths of the L\'evy process. Indeed, under the explosiveness assumption~\eqref{eq:infinite} it can be shown that
\[ \{X_A(t): t\geq 0\} \indepp \{X_B(t): t\geq 0\} \mid \{X_C(t): t\geq 0\} \quad \iff \quad \indepABC \]
This result enables the definition of graphical models for L\'evy processes with large potential for new theory, statistical methods and applications.

%%  MR numbers will be added by VTeX.                      %%
%%                                                         %%
%%  Use \cite{...} to cite references in text.             %%
%%                                                         %%
%%%%%%%%%%%%%%%%%%%%%%%%%%%%%%%%%%%%%%%%%%%%%%%%%%%%%%%%%%%%%

%% if your bibliography is in bibtex format, uncomment commands:
\bibliographystyle{imsart-nameyear} % Style BST file (imsart-number.bst or imsart-nameyear.bst)
%%\bibliography{references}       % Bibliography file (usually '*.bib')

%% or include bibliography directly:
% \begin{thebibliography}{}
% \bibitem{b1}
% \end{thebibliography}

\begin{appendix}

\section{
Additional proofs for Sections~\ref{sec:theory} and \ref{sec:alternative}:
Conditional independence and additional characterization results}\label{sec:CIproofs}
\subsection{Consistency properties of the basic operations with measures}\label{app:basic}
\begin{lemma}\label{lemma:marginalreducedcompatibility} 
Let $(A,B,C)$ be a partition of $V$ with $C \neq \emptyset$. Then
\[(\Lambda_{A \cup C})_C =\Lambda_C,
\qquad (\Lambda^0_{A \cup C})^0_C =\Lambda^0_C,
\qquad (\Lambda^0_{A \cup C})_C = (\Lambda_{B \cup C})^0_C.
\]
\end{lemma}
\begin{proof}
The domain of these measures is $\cE^C = \R^{C}\setminus\{0_C\}$, and expressed in terms of $\Lambda$ they coincide with $\Lambda(y_C \in \cdot)$, $\Lambda(y_C \in \cdot, y_B=0_B, y_A=0_A)$ and $\Lambda(y_C \in \cdot, y_B=0_B)$, respectively. 
\end{proof}
    \begin{lemma} \label{lemma:marginalcompatibility} For a non-empty $D \subset V$ consider $Y\sim \L_{R_D \times \R^{V\setminus D}}$ with $R_D\in \mathcal R(\Lambda_D)$.  Then
    \begin{enumerate}[label=(\alph*)]
      \item $Y_D\sim \L_{R_D}$.
      \item For any $R=R_D \times R_{V\setminus D}$ it holds that $\L_{R}$ is the conditional law of $Y$ given $Y_{V\setminus D}\in R_{V\setminus D}$.
      \end{enumerate}
    \end{lemma}
    \begin{proof}
    
    \begin{enumerate}[label=(\alph*),wide=0mm]
    \item  follows immediately from~\eqref{eq:probabR} and the definition of the marginal measure in~\eqref{eq:Lambda0}.
    \item  is seen from 
    \[\frac{\Lambda(\D y)}{\Lambda(R)}=\frac{\Lambda(\D y)}{\Lambda(R_D \times \R^{V\setminus D})}\frac{\Lambda(R_D \times \R^{V\setminus D})}{\Lambda(R)}=\frac{\p(Y\in \D y)}{\p(Y\in R)}\]
    for $y \in R$ and $Y_D \in R_D$ a.s. \qedhere
    \end{enumerate}
    \end{proof}

\subsection{Restriction and extension}
We start with a few fundamental results which follow directly from basic probability theory. We provide short arguments for completeness.
Consider a partition $(A,B,C)$ of $V$. For two product-form sets $R,R'\in \cR(\Lambda)$ the corresponding random vectors 
as in \eqref{eq:probabR}
are denoted by $Y\sim \L_R$ and $Y'\sim \L_{R'}$ in the following.

\begin{lemma}\label{lem:condProbs}
Let $C \neq \emptyset$.
Assume that $R_{A\cup B}\supset R'_{A\cup B}\supset E_{A\cup B}$, where the latter is some Borel set. Then
\[\p(Y'_{A\cup B}\in E_{A\cup B}|Y'_C=y_C)=\frac{\p(Y_{A\cup B}\in E_{A\cup B}|Y_C=y_C)}{\p(Y_{A\cup B}\in R'_{A\cup B}|Y_C=y_C)}\]
for $\p_{Y'_C}$-almost all $y_C\in R_C\cap R'_C$. 
\end{lemma}
\begin{proof}
Choose a Borel set $E_C\subset R_C\cap R'_C$ and observe that
\[\p(Y'_{A\cup B}\in E_{A\cup B},Y'_C\in E_C)=\p(Y_{A\cup B}\in E_{A\cup B},Y_C\in E_C)\frac{\Lambda(R)}{\Lambda(R')}.\]
The right hand side can be rewritten as
\[\int_{E_C}\frac{\p(Y_{A\cup B}\in E_{A\cup B}|Y_C=y_C)}{\p(Y_{A\cup B}\in R'_{A\cup B}|Y_C=y_C)}\p(Y_{A\cup B}\in R'_{A\cup B},Y_C\in \D y_C)\frac{\Lambda(R)}{\Lambda(R')},\]
where the denominator is non-zero for $\p_{Y'_C}$-almost all $y_C\in E_C$.
Note that the scaled measure corresponds to $\p(Y_C'\in \D y)$ and thus we obtain
\[\p(Y'_{A\cup B}\in E_{A\cup B},Y'_C\in E_C)=\int_{E_C}\frac{\p(Y_{A\cup B}\in E_{A\cup B}|Y_C=y_C)}{\p(Y_{A\cup B}\in R'_{A\cup B}|Y_C=y_C)}\p(Y_C'\in \D y_C),\]
and the result follows.
\end{proof}

\begin{remark}
\begin{enumerate}[label=(\alph*),wide=0mm]
\item If in addition $Y_A \indepp Y_B \mid Y_C$  in Lemma~\ref{lem:condProbs}, by setting $E_{A \cup B}=E_A \times R'_B$, we can further conclude that for $\p_{Y'_C}$-almost all $y_C\in R_C\cap R'_C$ 
\[\p(Y'_{A}\in E_{A}|Y'_C=y_C)=\frac{\p(Y_{A}\in E_{A}|Y_C=y_C)}{\p(Y_{A}\in R'_{A}|Y_C=y_C)},\]
(and the analogous identity with reversed roles for $A$ and $B$). 
\item
We also note that, when $R'_{A\cup B}=R_{A\cup B}$ in Lemma~\ref{lem:condProbs}, the two conditional probabilities $\p(Y'_{A\cup B}\in E_{A\cup B}|Y'_C=y_C)$ and $\p(Y_{A\cup B}\in E_{A\cup B}|Y_C=y_C)$
coincide for almost all $y_C\in R_C\cap R'_C$ with respect to both probability measures $\p_{Y_C}$ and $\p_{Y'_C}$.
\item For $C=\emptyset$ the analogous identity
\[\p(Y'_{A \cup B}\in E_{A \cup B})=\frac{\p(Y_{A \cup B}\in E_{A \cup B})}{\p(Y_{A \cup B}\in R'_{A \cup B})}\]
holds trivially.
\end{enumerate} 
\end{remark}

\begin{lemma}[Restriction]\label{lem:restrict}
Assume that $R'\subset R$. Then $Y_A \indepp Y_B \mid Y_C$ implies  $Y'_A \indepp Y'_B \mid Y'_C$.
This result is also true for $C=\emptyset$ and the respective independence statements. 
\end{lemma}
\begin{proof}
Consider Borel subsets $E_A\subset R'_A$ and $E_B\subset R'_B$. 
According to Lemma~\ref{lem:condProbs} and conditional independence $Y_A \indepp Y_B \mid  Y_C$ we have
\[\p(Y'_A\in E_A,Y'_B\in E_B|Y'_C=y_C)=\frac{\p(Y_A\in E_A|Y_C=y_C)}{\p(Y_A\in R'_A|Y_C=y_C)}\,\frac{\p(Y_B\in E_B|Y_C=y_C)}{\p(Y_B\in R'_B|Y_C=y_C)}\]
for $\p_{Y'_C}$-almost all $y_C\in R'_C$. By considering the same identity with $E_B=R'_B$ and then $E_A=R'_A$ and applying Lemma~\ref{lem:condProbs} again, we find that the two ratios on the right-hand side are $\p(Y'_A\in E_A|Y'_C=y_C)$ and $\p(Y'_B\in E_B|Y'_C=y_C)$, respectively, and the first claim is proven.
The case $C=\emptyset$ follows analogously from the trivial identity 
\[\p(Y'_A\in E_A,Y'_B\in E_B)=\frac{\p(Y_A\in E_A,Y_B\in E_B)}{\p(Y_A\in R'_A,Y_B\in R'_B)}. \qedhere\]
\end{proof}

\begin{lemma}[Extension via $C$]\label{lem:extension}
Let $C\neq \emptyset$ and assume that $R'_{A\cup B}=R_{A\cup B}$.\\ Then $Y_A \indepp Y_B \mid Y_C$ and $Y'_A \indepp Y'_B \mid Y'_C$ imply $Y''_A \indepp Y''_B \mid Y''_C$ with $Y''\sim \L_{R\cup R'}$. 
\end{lemma}
\begin{proof}
Lemma~\ref{lem:condProbs} shows that for any Borel subsets $E_A\subset R_A,E_B\subset R_B$ we have
\begin{align*}
&\p(Y''_A\in E_A,Y''_B\in E_B|Y''_C=y_C)=\p(Y_A\in E_A,Y_B\in E_B|Y_C=y_C)\\
&=\p(Y_A\in E_A|Y_C=y_C)\p(Y_B\in E_B|Y_C=y_C)\\
&=\p(Y''_A\in E_A|Y''_C=y_C)\p(Y''_B\in E_B|Y''_C=y_C)
\end{align*} for $\p_{Y''_C}$-almost all $y_C\in R_C$.
Analogously, we get a factorization for $\p_{Y''_C}$-almost all $y_C\in R'_C$ when considering~$Y'$ instead of $Y$. This shows the stated result. 
\end{proof}
In Lemma~\ref{lem:extension} we verify an extension result for conditional independence from two sets $R$ and $R'$ with identical $A\cup B$ component to their union $R \cup R'$. Naturally, the analogous extension result holds also for extending conditional independence from countably many product-form sets $R^{(n)}$ with $R^{(n)}_{A \cup B}=R_{A \cup B}$ for all $n$, resulting in the conditional independence statement on the union $\bigcup_n R^{(n)}$.

\subsection{Main proofs for characterization results}
We are now ready to prove our first characterization result.

\begin{proof}[Proof of Theorem~\ref{thm:test_class}]
By definition $A \indep B\mid C\,[\Lambda]$ implies $Y_A \indepp Y_B\mid Y_C$ with $Y\sim\L_{R_{v,\eps}}$ for any $v\in V$ and $\eps>0$ such that $\Lambda(R_{v,\eps})>0$.
Additionally, it implies $A \indep B  \, [\Lambda^0_{A\cup B}]$, see Lemma~\ref{lem:Lambda0}.
Thus, $A \indep B\mid C\,[\Lambda]$ implies both (i) and (ii), also when $C=\emptyset$.

(i) In the opposite direction, we take any $R=\bigtimes R_v \in \cR(\Lambda)$ and note that there exists $v\in V$ such that the closure of $R_v$ does not contain 0. Hence we may choose $\eps>0$ such that $R\subset R_{v,\eps}$. 
According to Lemma~\ref{lem:restrict} we have $Y_A \indepp Y_B\mid Y_C$ with $Y\sim\L_{R}$. Hence $A \indep B\mid C\,[\Lambda]$. The argument also applies when $C=\emptyset$. 

(ii) It is sufficient to establish (i), and so we take any $v\in A\cup B$, $\eps>0$ and consider the set $R_{v,\eps}$. We partition this set according to $y_C\neq 0_C$ and $y_C= 0_C$, letting $Y$ and $Y'$ be the respective random vectors. Here we assume that both sets are charged by $\Lambda$, since otherwise the proof is only simpler. 
According to Lemma~\ref{lem:extension} it is sufficient to show both $Y_A \indepp Y_B \mid Y_C$ and $Y'_A \indepp Y'_B \mid Y'_C$.  

The latter follows directly from $A \indep B  \, [\Lambda^0_{A\cup B}]$, see also the proof of Lemma~\ref{lem:Lambda0}.
The former follows from the countably-infinite version of Lemma~\ref{lem:extension} by considering the sets $R^{(n,c)}=\{|y_v|\geq \eps, |y_c|\geq 1/n\}$ for all $c\in C$ and natural~$n$.
The respective conditional independence statements are implied by the assumption in (ii) together with the restriction Lemma~\ref{lem:restrict}.
\end{proof}

As a side result the same argument that we used in proof of Theorem~\ref{thm:test_class} (ii) above reveals that a single $v\in V$ can be taken when checking for conditional independence under the condition  that $\Lambda(y\neq 0_V,y_v=0)=0$. To be precise we formulate the following lemma.

\begin{lemma}[Single $v$ test]\label{lem:single_v}
Assume that a fixed $v\in V$ is such that $\Lambda(y\neq 0_V,y_v=0)=0$. Then 
$Y_A\indepp Y_B\mid Y_C$ for all $Y\sim\L_{R_{v,\eps}}, \eps>0$ implies $\indepABC$.
This result is also true for $C=\emptyset$ and the respective independence statements. 
\end{lemma}

Next we prove uniqueness and representation of the kernel.

\begin{proof}[Proof of Lemma~\ref{lem:kernel}]
Let $R^1_C$ be some product form set bounded away from $0_C$ and define $R^1=\R^{A\cup B}\times R^1_C$. 
Suppose for a moment that $R_C\subset R^1_{C}$.
Letting $Y\sim\L_{R^1}$ (the case of zero mass being trivial) we observe that
\[\Lambda(R)=\p(Y_{A\cup B}\in R_{A\cup B},Y_C\in R_C)\Lambda(R^1)=\int_{y_C\in R_C}\nu_{C}^{1}(y_C,R_{A\cup B})\p(Y_C\in \D y_C)\Lambda(R^1),\]
where $\nu_{C}^{1}(y_C,\cdot)$ is a probability kernel corresponding to $Y$, see~\citet[Thm.\ 6.3]{kallenberg}. Note that $\Lambda_C(\D y_C)=\p(Y_C\in \D y_C)\Lambda(R^1),y_C\in R_C$ giving the stated representation of $\Lambda(R)$.
In general, we partition $\cE^C$ into countably many product form sets $R^i_C$ and define $\nu_C(y_C,\cdot)=\nu^i_C(y_C,\cdot)$ for $y_C$ in the respective partition.
The stated form of $\nu_C$ and $\Lambda_C$-uniqueness follow from Lemma~\ref{lem:condProbs}.
\end{proof}

\section{Further details for the generic construction of Section~\ref{subsec:CI3constr}}\label{app:CIconstr}

Here we provide further details on the constructive example of a measure $\Lambda$ in $d=3$ dimensions with index set $V = \{1,2,3\}$ such that the conditional independence $\{1\} \indep \{3\} \mid \{2\}\; [\Lambda]$ holds. 
Table~\ref{tab:CIconstr} contains the reformulation of the factorization of the modified density~\eqref{eq:lambdaCIconstr} in terms of original building blocks, that is, the densities $\kappa_{ij}$, their joint marginal density $m$ and the mixture probabilities $p_{ij}$ with $q_{ij}=1-p_{ij}$. Figure~\ref{fig:CIexample} illustrates them for faces up to dimension 2.

We note that if the starting density $\kappa$ is $-\alpha$-homogeneous for some $\alpha>2$, it is easily checked that its marginal density $m$ will be $-(\alpha-1)$-homogeneous and the resulting measure $\Lambda$ will be $-(\alpha-2)$-homogeneous; see Example~\ref{HR_forest}.

%%% An example of a density that one could plug in for $\kappa_{ij}$ is the bivariate H\"usler-Reiss density \eqref{eq:HRbivdensity}.
%%% It has identical univariate margins $m(y)=y^{-2}$ and, in addition, it satisfies the homogeneity $\kappa_{ij}(t y) = t^{-3} \kappa_{ij}(y)$ for $t > 0$ and $y \in (0,\infty)^2$. This leads to a situation, where the resulting measure $\Lambda$ becomes $-1$-homogeneous with margins normalized as desired, i.e.\ $\Lambda(y_i>1)=1$ for $i \in V$. More generally, if the starting density $\kappa$ is $-\alpha$-homogeneous for some $\alpha>2$, it is easily checked that its marginal density $m$ will be $-(\alpha-1)$-homogeneous and the resulting measure $\Lambda$ will be $-(\alpha-2)$-homogeneous. 

%%%
%%% In general, the resulting measure $\Lambda$ has bivariate joint survival functions as given by 
%%%\begin{align*}
%%%\Lambda(y_i>u,y_j>u) &= p_{ij}\int_u^\infty \int_{u}^\infty \kappa_{ij}(y_{ij}) \D y_i \D y_i, \qquad (i,j) \in \{(1,2),(2,3)\},\\
%%%\Lambda(y_1>u,y_3>u) &= p_{12}p_{23}\int_u^\infty \int_{u}^\infty \bigg[ \int_{0}^\infty \frac{\kappa_{12}(y_{12})\kappa_{23}(y_{23})}{m(y_2)}  \D y_2  \bigg] \D y_1  \D y_3.
%%%\end{align*}

\begin{table}[htb]
\caption{\small Modified densities of the  trivariate example from Section~\ref{subsec:CI3constr}; see also Figure~\ref{fig:CIexample}. The resulting measure $\Lambda$ satisfies $\{1\} \indep \{3\} \mid \{2\}\; [\Lambda]$ and has identical one-dimensional marginals. It is globally Markov for the chain graph $1-2-3$ with zero mass on $\cZ=\{y_1\neq 0,y_2=0,y_3\neq 0\}$; see Theorem~\ref{thm:decomposable}.
}\label{tab:CIconstr}
{\small
\renewcommand{\arraystretch}{1.5}
\begin{tabular}{|l|ccc||ccc|c|}
\hline
\rowcolor{gray!30}
\textbf{Dimen-} & \multicolumn{3}{c||}{\textbf{Faces}} & 
\multicolumn{4}{c|}{\textbf{Modified densities}}
\\ 
\rowcolor{gray!30}
\textbf{sion} & $y_1$ & $y_2$ & $y_3$ 
& $\overline \lambda_{12}(y_{12})$ 
& $\overline \lambda_{23}(y_{23})$ 
& $\overline \lambda_{2}(y_2)$
& $\overline \lambda(y) = \lambda(y)$  
\\
\hline
%%0 & 0 & 0 & 0
%%& 1 & 1 & 1 & 1
%%\\
%%\hline
& 0 & 0 & $\neq 0$ 
& 1 & $q_{23} m(y_3)$ & 1 & {\color{blue} $q_{23} m(y_3)$}
\\
1 & $\neq 0$ & 0 & 0 
& $q_{12} m(y_1)$  & 1 & 1 & {\color{red} $q_{12} m(y_1)$}
\\
& 0 & $\neq 0$ & 0
& $q_{12} m(y_2)$  & $q_{23} m(y_2)$  & $m(y_2)$ & {\color{ blue!50!red} $q_{12}q_{23} m(y_2)$}
\\
\hline
& $\neq 0$ & $\neq 0$ & 0
& $p_{12} \kappa_{12}(y_{12})$  & $q_{23} m(y_2)$  & $m(y_2)$ & {\cellcolor{red!10} $p_{12}q_{23} \kappa_{12}(y_{12})$}
\\
2 & 0 & $\neq 0$ & $\neq 0$
&$q_{12} m(y_2)$  & $p_{23} \kappa_{23}(y_{23})$ & $m(y_2)$ & {\cellcolor{blue!10}$q_{12}  p_{23} \kappa_{23}(y_{23})$}
\\
& $\neq 0$ & 0 & $\neq 0$ 
& $q_{12} m(y_1)$  & $q_{23} m(y_3)$  & 1 & $0$
\\
\hline
3 & $\neq 0$ & $\neq 0$ & $\neq 0$
&$p_{12} \kappa_{12}(y_{12})$ & $p_{23} \kappa_{23}(y_{23})$ & $m(y_2)$ & $\frac{p_{12} p_{23}  \kappa_{12}(y_{12})  \kappa_{23}(y_{23})}{m(y_2)}$
\\
\hline
\end{tabular}
}
\end{table}

\begin{figure}[htb]
\mbox{\small   
\begin{tikzpicture}[scale = 2.5, >=stealth]
              %---------------------------------
      \coordinate (O) at (0,0,0);
      \coordinate (e2+) at (1,0,0);
      \coordinate (e3+) at (0,1,0);
      \coordinate (R) at (0,0.5,0);
      %---------------------------------
      \coordinate (e2+3+) at (1,1,0);
       \coordinate (M) at (0.5,0.5,0);
      %---------------------------------
    \filldraw[dashed, color=blue, fill=blue!10]   (e2+3+) -- (e3+) --  (O) -- (e2+) --  (e2+3+);   
    \node at (M) {\textcolor{blue}{$p_{ij}\kappa_{ij}(y_{ij})$ }};      
       %--------------------------------- 
         \draw[thick, blue, ->] (O) to node[midway, anchor=north] {$q_{ij} m(y_i)$} (e2+) node[anchor=north west]{\textcolor{black}{$y_{i}$}};     
    \draw[thick, blue, ->] (O) to (e3+) node[anchor=south east]{\textcolor{black}{$y_{j}$}} ;
    \node[rotate=90,yshift=4mm] at (R) {\textcolor{blue}{$q_{ij} m(y_j)$}};
      %---------------------------------
      \filldraw(O) circle (1pt) node[anchor=north east] {$0$};
        \end{tikzpicture} 
        \hspace*{3mm}
    \begin{tikzpicture}[scale = 2.5, >=stealth]
              %---------------------------------
      \coordinate (O) at (0,0,0);
      \coordinate (e1+) at (0,0,1);
      \coordinate (e2+) at (1,0,0);
      \coordinate (e3+) at (0,1,0);
      %---------------------------------
      \coordinate (e1+2+) at (1,0,1);
      \coordinate (e2+3+) at (1,1,0);
           %---------------------------------
    \filldraw[dashed, blue, fill=blue!10]   (e2+3+) -- (e3+) --  (O) -- (e2+) --  (e2+3+) node[anchor=south west]{$q_{12}p_{23} \kappa_{23}(y_{23})$ };     
       \filldraw[dashed, color=red, fill=red!10]   (e1+2+) -- (e2+) --  (O) -- (e1+) --  (e1+2+) node[anchor=north west]{$p_{12}q_{23} \kappa_{12}(y_{12})$ };      
       %---------------------------------      
      \draw[thick, red, ->] (O) to node[midway, left]{$q_{12} m(y_1)$} (e1+) node[anchor=north]{\textcolor{black}{$y_{1}$}};
      \draw[thick, blue!50!red, ->] (O) to node[midway, anchor=south, xshift=0cm]{$q_{12}q_{23} m(y_2)$} (e2+) node[anchor=north west]{\textcolor{black}{$y_{2}$}};
      \draw[thick, blue, ->] (O) to node[midway, left]{$q_{23} m(y_3)$} (e3+) node[anchor=south east]{\textcolor{black}{$y_{3}$}} ;
      %---------------------------------
      \filldraw(O) circle (1pt) node[anchor=south east] {$0$};
        \end{tikzpicture}             
        }

\caption{Illustration of the generic construction of a measure $\Lambda$ satisfying $\{1\} \indep \{3\} \mid \{2\}\; [\Lambda]$ from Section~\ref{subsec:CI3constr}. Left: measure $\lambda_{ij}$, which serves as building block. Right: resulting densities $\lambda$ on the faces up to dimension 2; see Table~\ref{tab:CIconstr}.}\label{fig:CIexample}
\end{figure}
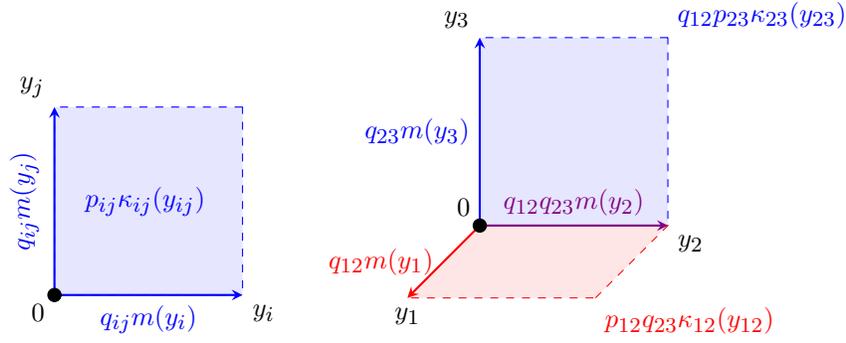

\section{Proofs and examples for Section~\ref{sec:undirected}: 
Undirected graphical models}

\subsection{Auxiliary result for the Hammersley-Clifford-type theorem}
\label{app:recursive_global}

\begin{proposition}\label{prop:recursive_global}
Assume that $\Lambda$ satisfies \eqref{eq:infinite} and that $V=D \cup C$ is the vertex set of a graph $\cG$, where $C$ is a clique and  $\widetilde S=D \cap C$ separates $D\setminus \widetilde S$ from $C \setminus \widetilde S$.
Then $\Lambda$ satisfies the global Markov property with respect to $\mathcal G$ if and only if
\begin{align}\label{eq:recursive_global}
D\setminus \widetilde S \indep C\setminus \widetilde S \mid \widetilde S\; [\Lambda]
\end{align}
and $\Lambda$ is globally Markov when restricted to $D$.
\end{proposition}

\begin{proof}[Proof of Proposition~\ref{prop:recursive_global}]
 The global Markov property for $\Lambda$ implies \eqref{eq:recursive_global} and, since $\widetilde S \subset C$ is fully connected, the global Markov property for $\Lambda$ restricted to $D$, cf.~\citet[Prop.~3.17 (start of proof)]{Lauritzen}. So we only need to establish the other direction, i.e.\ show the global Markov property for $\Lambda$ from these two properties.
 
To this end, consider disjoint sets $A,B,S\subset V$ with $S$ separating $A$ from $B$.
As outlined in Remark~\ref{rk:GMP}, it suffices to consider the partition situation, where $A\cup B\cup S=V$. Since $C$ is a clique, the sets $A$ and $B$ cannot both intersect $C$, otherwise they could not be separated by any $S$, so we may assume that $B\cap C=\emptyset$.

Let us first consider the case that $S\subset D$.
Since $\Lambda$ satisfies \eqref{eq:infinite}, we can employ the semi-graphoid property (L4). We do so with the four sets $(B, A\cap D, S, A\setminus D)$ as the sets $(A,B,C,D)$ from the original formulation of (L4).
The first condition there reads as $B\indep (A\cap D)\mid S$, which is implied by the global Markov property for $D$. The second condition reads as $B\indep (A\setminus D)\mid S \cup (A\cap D)$, which follows from the other assumption, $ (D\setminus {\widetilde S})\indep (A\setminus D)\mid {\widetilde S}$, by means of~(L3).
As a result we get $A\indep B\mid S$ as required.

If instead $E=S\cap (C\setminus \widetilde S) \neq \emptyset$, set $A'=A \cup E$ and $S'= S \setminus E$.
Then we can finish the proof by showing that $S'$ still separates $A'$ from $B$, but  with $S' \subset D$, so that we can employ the above argument again to obtain $A'\indep B\mid S'$, which subsequently implies $A\indep B\mid S$ by (L3). 
It remains to be seen that indeed $S'$ separates $A'$ from $B$. Let $b \in B$ and $a \in A'=A \cup E$. We need to show that any path $\pi$ from $b$ to $a$ (if such a path exists) necessarily visits $S'=S\setminus E$.\\
\textbf{Case 1:} $a \in E$. Since $\widetilde S$ separates $D \setminus \widetilde S$ from $C \setminus \widetilde S$, $\widetilde S$ also separates their subsets $B$ and $E$. In particular, $\pi$ must pass through an element $v \in \widetilde S$ before reaching $C \setminus \widetilde S$ for the first time. Since $\widetilde S \subset A \cup S'$, either $v \in S'$ (and we are done) or $v \in A$. However, if $v \in A$, recall that $S$ separates $B$ from $A$, so $\pi$ must first pass through an element $s \in S$ before reaching $v\in A$. Here, $s$ cannot belong to $E$, since $\pi$ has not yet arrived in $C\setminus  \widetilde S$ for the first time. Hence $s \in S'=S\setminus E$.\\
\textbf{Case 2:} $a \in A$. Since $S$ separates $B$ from $A$, $\pi$ needs to visit some $s \in S$. If $s \in S'=S\setminus E$, we are done. Else $s \in E$. But then the argument from Case 1 shows that on the way between $b$ and $s \in E$, the path $\pi$ needs to visit $S'$.
\end{proof}

\subsection{Proof of the Hammersley-Clifford-type theorem}\label{app:HC}
\begin{proof}[Proof of Theorem~\ref{thm:decomposable}] 
Let us prove the `if' and `only if' directions separately.

``Global Markov property $\Rightarrow$ Factorization:''
First we prove that the global Markov property of $\Lambda$ implies the factorization \eqref{eq:u_dec} by induction on the number of cliques.  The statement holds trivially for a single clique, as there are no separators, so the factorization reads $\olambda(y)=\olambda(y)$ for all $y \in \RR^{C_1}$, which is always true.  We let $\cG_{j}$ be the decomposable graph that arises from restricting $\cG$ to $D_{j}$ and abbreviate $\cZ_{j}=\cZ(\cG_j)$ henceforth.
Let us assume that the implication holds for $n-1$ cliques, that is, $\Lambda_{D_{n-1}}$ being globally Markov with respect to $\cG_{n-1}$ implies
\begin{align}\label{eq:IA}
\olambda(y_{D_{n-1}})\prod_{j < n-1}\olambda_{S_j}(y_{S_j}) = \prod_{j < n}\olambda_{C_j}(y_{C_j}).
\end{align} 
for $\mu_{D_{n-1}}$-almost all $y_{D_{n-1}}\notin \cZ_{n-1}$. For the induction step, suppose now that $\Lambda$ is globally Markov with respect to $\cG_n$. We need to  establish \eqref{eq:u_dec} for $\mu$-almost all $y\notin \cZ_{n}$.

The global Markov property of $\Lambda$ for $\cG_n$ implies also the global Markov property of
$\Lambda_{D_{n-1}}$ for $\cG_{n-1}$ and we obtain \eqref{eq:IA} as well by the induction assumption.
It is easily seen that 
\[\RR^V \setminus \cZ_{n} \subset (\RR^{D_{n-1}} \setminus \cZ_{n-1}) \times \RR^{V \setminus D_{n-1}}\]
and so \eqref{eq:IA} holds for $\mu$-almost all $y\notin \mathcal Z_n$.
According to the running intersection property we have
$C_n\setminus S_{n-1} \indep D_{n-1}\setminus S_{n-1} \mid S_{n-1} \;[\Lambda]$, which in view of Proposition~\ref{prop:trick} is equivalent to
\[\olambda(y)\olambda_{S_{n-1}}(y_{S_{n-1}}) = \olambda_{C_n}(y_{C_n})\olambda_{D_{n-1}}(y_{D_{n-1}})\]
for $\mu$-almost all $y \in \RR^V$ 
 excluding those with $y_{S_{n-1}}=0_{S_{n-1}}$ and 
 $y_{C_n\setminus S_{n-1}} \neq 0_{C_n\setminus S_{n-1}}$
 and $y_{D_{n-1}\setminus S_{n-1}} \neq 0_{D_{n-1}\setminus S_{n-1}}$. The factorization therefore holds for $\mu$-almost all $y\notin \mathcal Z_n$, since for any partition $(A,B,S)$ of $V$, where $S \in \cS$ separates $A$ from $B$ in the graph $\cG$ we have
\begin{align*}
\{ y \,:\, y_A\neq 0_A, \, y_B\neq 0_B, \, y_S=0_S\}
= \bigcup_{a \in A, \, b \in B} \{y  \,:\, y_a \neq 0, y_b\neq 0, y_S=0_S\} \subset \cZ(\cG).
\end{align*}
 Here $S_{n-1}$ is also allowed to be an empty set and we recall that $\olambda_\emptyset=1$.
Multiplying both sides by $\prod_{j<n-1}\olambda_{S_j}(y_{S_j})$ and using \eqref{eq:IA} yields \eqref{eq:u_dec} for $\mu$-almost all $y\notin \mathcal Z_n$, and taken together this completes the induction step.

``Factorization $\Rightarrow$ Global Markov property:''
The proof in the other direction is more subtle.
We need to show that, if the factorization \eqref{eq:u_dec} holds  for $\mu$-almost all $y\notin \mathcal Z(\mathcal G)$, then  $\Lambda$ is globally Markov with respect to $\cG$.
Again, we use induction on the number of cliques, where the case of a single clique is trivial, as there are no separators, so $\Lambda_{C_1}$ is always globally Markov with respect to $\cG_1$, a true conclusion.  Our induction assumption is now that \eqref{eq:IA} for $\mu_{D_{n-1}}$-almost all $y_{D_{n-1}}\notin \cZ_{n-1}$ implies that $\Lambda_{D_{n-1}}$ is globally Markov with respect to $\cG_{n-1}$. We need to show that
$\Lambda$ is globally Markov with respect to $\cG_{n}$, whilst assuming  \eqref{eq:u_dec} for $\mu$-almost all $y\notin \cZ_{n}$.

Note that $\{y_{D_{n-1}}=0_{D_{n-1}}\} \subset \RR^V \setminus \cZ_n$. 
So we obtain from \eqref{eq:u_dec} that for $\mu$-almost all $y \in \cE$ with $y_{D_{n-1}}=0_{D_{n-1}}$ 
\[
\lambda(0_{D_{n-1}},y_{C_n\setminus S_{n-1}}) \mu_{D_{n-1}}(\{0_{D_{n-1}}\}) 
 = \lambda_{C_n}(0_{S_{n-1}},y_{C_n\setminus S_{n-1}}) \mu_{S_{n-1}}(\{0_{S_{n-1}}\}),
\]
where the left-hand side and right-hand side are the densities of the reduced measures 
\[\Lambda^0_{C_n\setminus S_{n-1}}=(\Lambda^0_{V\setminus S_{n-1}})^0_{C_n\setminus S_{n-1}}
\quad \text{and} \quad (\Lambda_{C_n})^0_{C_n\setminus S_{n-1}}=(\Lambda^0_{V\setminus S_{n-1}})_{C_n\setminus S_{n-1}},
\] 
respectively, where the last two identities follow from Lemma~\ref{lemma:marginalreducedcompatibility} with $A=D_{n-1}\setminus S_{n-1}$, $B=S_{n-1}$ and $C=C_n\setminus S_{n-1}$ therein. 
By Corollary~\ref{cor:indepRephrase} we obtain the independence statement for the restricted measure $\Lambda^0_{V\setminus S_{n-1}}$, that is,
\begin{align}\label{eq:indepPartDone}
C_n\setminus S_{n-1} \indep D_{n-1}\setminus S_{n-1} \; [\Lambda^0_{V\setminus S_{n-1}}],
\end{align}
which further implies by Corollary~\ref{cor:indepRephrase} and Lemma~\ref{lemma:marginalreducedcompatibility} that
\[
\Lambda^0_{D_{n-1}\setminus S_{n-1}}
= (\Lambda^0_{V\setminus S_{n-1}})^0_{D_{n-1}\setminus S_{n-1}}
= (\Lambda^0_{V\setminus S_{n-1}})_{D_{n-1}\setminus S_{n-1}}
= (\Lambda_{D_{n-1}})^0_{D_{n-1}\setminus S_{n-1}}.
\]
Therefore, we have for $\mu$-almost all $y\in \cE$ with $y_{C_n}=0_{C_n}$ that
\begin{align}\label{eq:DS0}
\lambda(y_{D_n\setminus S_{n-1}},0_{C_n}) \mu_{C_n}(\{0_{C_n}\}) = 
\lambda_{D_{n-1}}(y_{D_{n-1}\setminus S_{n-1}},0_{S_{n-1}}) \mu(\{0_{S_{n-1}}\}).
\end{align}
Further we obtain from \eqref{eq:u_dec} that for $\mu$-almost all $y \in \cE \setminus \cZ_n$ with $y_{C_{n}}=0_{C_{n}}$ 
\[
\Big[\lambda(y_{D_n\setminus S_{n-1}},0_{C_n}) \mu_{C_n\setminus S_{n-1}}(\{0_{C_n \setminus S_{n-1}}\})\Big] \prod_{j < n-1}\olambda_{S_j}(y_{S_j}) = \prod_{j < n}\olambda_{C_j}(y_{C_j}),
\]
where the term in brackets can be replaced by $\lambda_{D_{n-1}}(y_{D_{n-1}\setminus S_{n-1}},0_{S_{n-1}})$ due to \eqref{eq:DS0}.
Note that $y \in \cE \setminus \cZ_n$ with $y_{C_{n}}=0_{C_{n}}$ implies $y_{D_{n-1}} \in \cE^{D_{n-1}} \setminus \cZ_{n-1}$ with $y_{S_{n-1}}=0_{S_{n-1}}$, and conversely, if $y_{D_{n-1}} \in \cE^{D_{n-1}} \setminus \cZ_{n-1}$ with $y_{S_{n-1}}=0_{S_{n-1}}$, then $(y_{D_{n-1}},0_{C_n\setminus S_{n-1}}) \in \cE \setminus \cZ_n$ and $y_{C_{n}}=0_{C_{n}}$. Hence, we recover \eqref{eq:IA} for $\mu_{D_{n-1}}$-almost all $y_{D_{n-1}}\notin \cZ_{n-1}$ with $y_{S_{n-1}}=0_{S_{n-1}}$ (note that it is trivially satisfied for $y_{D_{n-1}}=0_{D_{n-1}}$). 
If $S_{n-1}=\emptyset$, the same line of reasoning 
(with $\mu_{S_{n-1}}(\{0_{S_{n-1}}\})$ replaced by $1$ above) even recovers  \eqref{eq:IA} for $\mu_{D_{n-1}}$-almost all $y_{D_{n-1}}\notin \cZ_{n-1}$. 

Else $S_{n-1}\neq\emptyset$ and we consider 
$y \not\in \cZ_n$ with $y_{S_{n-1}}\neq 0_{S_{n-1}}$, so that $y_{D_{n-1}} \not\in \cZ_{n-1}$ with $y_{S_{n-1}}\neq 0_{S_{n-1}}$. Conversely, such a $y_{D_{n-1}}$ satisfies $y \not\in  \cZ_n$ with $y_{S_{n-1}}\neq 0_{S_{n-1}}$ for any(!) $y_{C_n\setminus S_{n-1}} \in \RR^{C_n\setminus S_{n-1}}$. So we can integrate out the $C_n\setminus S_{n-1}$ components in \eqref{eq:u_dec} for such $y$ to obtain
\[\lambda_{D_{n-1}}(y_{D_{n-1}})\lambda_{S_{n-1}}(y_{S_{n-1}})\prod_{j<n-1}\olambda_{S_j}(y_{S_j}) = \lambda_{S_{n-1}}(y_{S_{n-1}})\prod_{j<n}\olambda_{C_j}(y_{C_j})\]
for $\mu_{D_{n-1}}$-almost all $y_{D_{n-1}}\notin \cZ_{n-1}$ with $y_{S_{n-1}}\neq 0_{S_{n-1}}$.
This formula is still true when the term $\lambda_{S_{n-1}}(y_{S_{n-1}})$ is dropped on both sides, since for $\mu$-almost all $y$ with $y_{S_{n-1}}\neq 0_{S_{n-1}}$, for which the marginal density $\lambda_{S_{n-1}}(y_{S_{n-1}})$ is zero, we also must have zero densities $\lambda_{D_{n-1}}(y_{D_{n-1}})$ and  $\lambda_{C_{n-1}}(y_{C_{n-1}})=\olambda_{C_{n-1}}(y_{C_{n-1}})$, as $\lambda_{S_{n-1}}(y_{S_{n-1}})$
arises as a marginal density from them.
So, collectively, we recover  \eqref{eq:IA} for $\mu_{D_{n-1}}$-almost all $y_{D_{n-1}}\notin \cZ_{n-1}$.

The induction assumption now gives us that $\Lambda_{D_{n-1}}$ is globally Markov with respect to $\cG_{n-1}$. Hence, by Proposition~\ref{prop:recursive_global} it suffices to establish
\[C_n\setminus S_{n-1} \indep D_{n-1}\setminus S_{n-1} \mid S_{n-1}\; [\Lambda],\]
which is our goal in what follows. 
If $S_{n-1}=\emptyset$, it follows readily from \eqref{eq:indepPartDone} and the proof is complete. 

Else $S_{n-1}\neq\emptyset$ and \eqref{eq:indepPartDone} only yields the independence of the reduced measure and by Theorem~\ref{thm:density_factorization} all that remains to be seen is that
\begin{align}\label{eq:finalproduct}
\lambda(y) \lambda_{S_{n-1}}(y_{S_{n-1}}) = \lambda_{C_n}(y_{C_n})\lambda_{D_{n-1}}(y_{D_{n-1}})
\end{align}
for $\mu$-almost all $y$ with  $y_{S_{n-1}}\neq 0_{S_{n-1}}$.

Here, we note that \eqref{eq:u_dec} for $\mu$-almost all $y\notin \cZ_{n}$ together with
\eqref{eq:IA}  for $\mu_{D_{n-1}}$-almost all $y_{D_{n-1}}\notin \cZ_{n-1}$ gives
\[
\lambda(y) \lambda_{S_{n-1}}(y_{S_{n-1}})
\prod_{j < n-1}\olambda_{S_j}(y_{S_j}) = \lambda_{C_n}(y_{C_n}) \lambda_{D_{n-1}}(y_{D_{n-1}}) \prod_{j < n-1}\olambda_{S_j}(y_{S_j})
\]
for $\mu$-almost all $y\notin \cZ_n$, $y_{S_{n-1}}\neq 0_{S_{n-1}}$. The product terms $\olambda_{S_j}(y_{S_j})$ can either be canceled by the same marginal cancellation argument as above if $y_{S_j}\neq 0_{S_j}$ or else due to $\olambda_{S_j}(0_{S_j})=1/\mu_{S_j}(\{0_{S_j}\})$. Hence we have \eqref{eq:finalproduct} for $\mu$-almost all $y\not\in\cZ_n$,  $y_{S_{n-1}}\neq 0_{S_{n-1}}$.

Finally, we note that, if $y_{S_{n-1}}\neq 0_{S_{n-1}}$, then $y\not\in\cZ_n$ is equivalent to 
$y_{D_{n-1}}\not\in\cZ_{n-1}$. By the global Markov property  on $D_{n-1}$ we know $\Lambda_{D_{n-1}}(\cZ_{n-1})=0$. Hence, both sides of \eqref{eq:finalproduct} are zero for $\mu$-almost all $y$ with $y_{S_{n-1}}\neq 0_{S_{n-1}}$ and $y_{D_{n-1}}\in\cZ_{n-1}$. Taken together, this establishes \eqref{eq:finalproduct} in the general case for $\mu$-almost all $y$ with  $y_{S_{n-1}}\neq 0_{S_{n-1}}$ and the proof is complete.
\end{proof}

\subsection{Example for Remark~\ref{rem:caveat}}\label{app:caveat}

Consider the chain graph $1-2-3-4-5$ on the vertices $V=\{1,\ldots,5\}$ where the neighboring indices are connected. We construct a measure $\Lambda$ on $\cE$ which is not(!)\ globally Markov with respect to this graph $\cG$, even though
\begin{enumerate}[label=(\roman*)]
\item \eqref{eq:infinite} holds;
\item the density $\lambda$ satisfies the factorization
\[\lambda(y)\lambda_2(y_2)\lambda_3(y_3)\lambda_4(y_4) = \lambda_{12}(y_{12})\lambda_{23}(y_{23})\lambda_{34}(y_{34})\lambda_{45}(y_{45})\]
for $\mu$-almost all $y$ such that $y_2,y_3,y_4\neq 0$;
\item $\Lambda(\cZ(\cG))=0$, that is, $\Lambda(y_k\neq 0,y_l=0, y_m\neq 0)=0$ whenever $k<l<m$.
\end{enumerate} 
As in Section~\ref{subsec:CI3constr} we consider
the same dominating measure~$\mu$ as in \eqref{product_measure} (here for $d=5$) and use some $-3$-homogeneous bivariate densities $\kappa_{ij}$ on $(0,\infty)^2$ with marginal densities $y^{-2}$ as building blocks, now for $(i,j) \in \{(1,2),(2,3),(3,4),(4,5)\}$.
In addition, we ensure that $\kappa_{12}$ and $\kappa_{23}$ are compatible with a $-4$-homogeneous trivariate density $\eta_{123}$ on $(0,\infty)^3$ as follows:
\begin{itemize}
\item $\eta_{12}=\kappa_{12}, \eta_{23}=\kappa_{23}$ and, thus, respective univariate marginals have densities $y^{-2}$,
\item $\eta_{123}(y_{123})\eta_2(y_2)\neq \eta_{12}(y_{12})\eta_{23}(y_{23})$ on a subset of $(0,\infty)^3$ of positive Lebesgue measure.
\end{itemize} 
The second condition is equivalent to saying that the measure corresponding to $\eta_{123}$ does not induce the conditional independence of $1$ and $3$ given $2$, see Theorem~\ref{thm:density_factorization}.
Such a construction is always possible, e.g.~one can start from $\eta_{123}$ with the respective univariate marginals and then work with $\kappa_{12}$ and $\kappa_{23}$ that arise from $\eta_{123}$.
From there, we define
\[\lambda(y)=\begin{cases}\kappa_{12}(y_{12})\kappa_{23}(y_{23})\kappa_{34}(y_{34})\kappa_{45}(y_{45})y_2^2y_3^2y_4^2, &y_1,y_2, y_3,y_4,y_5>0,\\
\eta_{123}(y_{123}),&y_1,y_2, y_3>0,y_4=y_5=0,\end{cases}
\]
and let it be 0 otherwise.
By construction, the measure $\Lambda$ satisfies the conditions (i) and (iii).
Furthermore, the bivariate marginals on $(0,\infty)^2$ are given by
\[ 
\lambda_{12}=\kappa_{12}+\eta_{12}=2\kappa_{12}, \quad \lambda_{23}=\kappa_{23}+\eta_{23}=2\kappa_{23}, \quad \lambda_{34}=\kappa_{34}, \quad \lambda_{45}=\kappa_{45}, 
 \]
 and the univariate ones by $\lambda_2(y_2)=2y_2^{-2}, \lambda_3(y_3)=2y_3^{-2}, \lambda_4(y_4)=y_4^{-2}$. This readily yields the factorization in (ii).
 Finally, we note that $\Lambda$ is not globally Markov for the chain graph since $\{1\} \indep \{3,4,5\}\mid \{2\}\;[\Lambda]$ is not true, because 
 \[\lambda(y)\lambda_2(y_2)=\eta_{123}(y_{123})2y_2^{-2}\neq 2\eta_{12}(y_{12})\eta_{23}(y_{23})=\lambda_{12}(y_{12})\lambda_{2345}(y_{2345})\]
on the subset $\{y:y_1,y_2, y_3>0,y_4=y_5=0\}$, which has positive $\mu$-mass. Hence, according to Theorem~\ref{thm:density_factorization} the condition $\{1\} \indep \{3,4,5\}\mid \{2\}\;[\Lambda]$ is violated.

Had we checked instead the factorization condition from Theorem~\ref{thm:decomposable} in terms of $\olambda$ on this set, we would have picked up the violation of the global Markov property. The left-hand side of \eqref{eq:u_dec} reads $\eta_{123}(y_{123})2y_2^{-2}2y_3^{-2}$ in this case, and the right-hand side is $2\eta_{12}(y_{12})2\eta_{23}(y_{23})y_3^{-2}$. They cannot be equal due to our choice of $\eta_{123}$.

\section{Proofs for Section~\ref{sec:extremes}: Relation to conditional independence in extremes}

\subsection{Derivation of extremal correlation $\chi_{ij}$ in Example~\ref{HR_forest}}
\label{HR_forest_proof}
\begin{proof}
  For $i\nleftrightarrow j$, since $\Lambda$ is a graphical model on the forest $\cG$, the global Markov property implies that $i\indep j\;[\Lambda]$. By Proposition~\ref{prop:indep} we therefore have that $\Lambda(y_i\neq 0, y_j\neq 0 )=0$, which yields $\chi_{ij} = 0$ according to~\eqref{eq:chidef}.

  For $i\leftrightarrow j$ we note that by homogeneity of $\Lambda$ we have $\chi_{ij}=\Lambda(y_i> 1, y_j> 1)/\Lambda(y_i>1)$ and due to all marginal densities being $m(y)=y^{-2}$ we have $\Lambda(y_i>1) = 1$. Because of the density factorization~\eqref{forest_dens} we can first integrate out all components $k \notin \mathrm{path}_{ij}$ to get  
  \[\Lambda(y_i> 1, y_j> 1) =  \int_1^\infty \int_1^\infty \int_{(0,\infty)^{|V(ij)|}} \frac{\prod_{(s,t) \in \mathrm{path}_{ij}} \overline \lambda_{st}(y_s, y_t)}{\prod_{s \in V(ij)} \overline \lambda_s(y_s)} \D y_i \D y_j \D y_{V(ij)} ,\] 
  where $V(ij)$ denotes all nodes on the path from $i$ to $j$, excluding $i$ and $j$. 
  %%%By Theorem~\ref{thm:decomposable}, for any $I\subset V(ij)$ we have that $\Lambda(y_i > 1, y_j>1, y_I = 0_I) = 0$, and 
  Thus
  \begin{align*}
    \chi_{ij} = \int_1^\infty \int_1^\infty \int_{(0,\infty)^{|V(ij)|}} \frac{\prod_{(s,t) \in \mathrm{path}_{ij}} p_{st} \kappa_{st}(y_s, y_t)}{\prod_{s \in V(ij)} y_s^{-2}} \D y_i \D y_j \D y_{V(ij)}
   = \chi_{ij}^{\text{HR}} \prod_{(s,t) \in \mathrm{path}_{ij}} p_{st} 
  \end{align*}
  where $\chi_{ij}^{\text{HR}} = 2 - 2\Phi(\sqrt{\Gamma_{ij}} /2)$ with $\Gamma_{ij} = \sum_{(s,t) \in \mathrm{path}_{ij}} \Gamma_{st}$ is the extremal correlation of the corresponding tree model on the connected component that contains $i$ and $j$ \citep[][Proposition 4]{eng2020}. This result is true since the domain of the above integral is contained in the set $\{y_i > 0, y_j > 0, y_{V(ij)} > 0\}$, on which the density coincides with the usual H\"usler--Reiss tree model.
  
\end{proof}

\subsection{Max-infinitely divisibility and asymptotic equivalence of the joint survival functions}
\label{app:etaOneHalf}

\begin{proof}[Proof of Lemma~\ref{lemma:etaOneHalf}]
Denoting $\Lambda_{ij}(u,u)=\Lambda_{ij}(y_i>u \text{ or } y_j>u)$, we can rewrite
 \begin{align}\label{eq:SurvivalRewrite}
    \PP( Z_i > u, Z_j > u ) &= 1 - 2e^{-\Lambda(u)} + e^{-\Lambda_{ij}(u,u)}.
  \end{align}
  Using $\Lambda_{ij}(u,u) \leq 2 \Lambda(u)$ for all $u>0$ gives
  \[
    \PP( Z_i > u, Z_j > u )
    \geq 1 - 2e^{-\Lambda(u)} + e^{-2\Lambda(u)}
= ( 1- e^{-\Lambda(u)})^2 =  \PP(Z_i>u)^2,
  \]
  and hence,
  \[ \PP( Z_i > u, Z_j > u ) \PP(Z_i>u)^{-2} \geq 1\]
  for all $u>0$. On the other hand \eqref{eq:LT-residual-tail-dep-coeff} tells us that the left-hand side is a function of the form $r(\PP(Z_i>u)^{-1})$ with $r$ regularly varying with index $-1/\eta_{ij}+2$. If we had $\eta_{ij}<1/2$, the left side would converge to 0 as $u \to \infty$ \citep[e.g.,][Appendix~B]{deh2006a}, contradicting the last inequality. 
  So we obtain $\eta_{ij} \geq 1/2$.
  
Further, recall the bounds $1-x \leq e^{-x} \leq  1-x + x^2/2$ for all $x>0$. Applied to \eqref{eq:SurvivalRewrite} we obtain upper and lower bounds for $\PP( Z_i > u, Z_j > u )$ as follows:
  \begin{align*}
2\Lambda(u) - \Lambda_{ij}(u,u) - \Lambda(u)^2 \leq     \PP( Z_i > u, Z_j > u ) 
                                  &\leq 2\Lambda(u) - \Lambda_{ij}(u,u) + \Lambda_{ij}(u,u)^2/2 \\
                                  &\leq  2\Lambda(u) - \Lambda_{ij}(u,u) +  2\Lambda(u)^2.
    \end{align*}  
    Since $ \Lambda(y_i>u, y_j>u)   = 2\Lambda(u) - \Lambda_{ij}(u,u)$, we may rewrite these bounds as
   \begin{align}\label{eq:joint-Lambda-bounds-by-PP}
   \PP( Z_i > u, Z_j > u ) - 2\Lambda(u)^2 \leq  \Lambda(y_i>u, y_j>u) \leq  \PP( Z_i > u, Z_j > u ) 
   + \Lambda(u)^2.\end{align}
        Hence, if we assume $\eta_{ij} > 1/2$, we obtain the assertion \eqref{eq:joint-survival-tail-equiv} from
    \[\lim_{u \to \infty} \frac{\Lambda(u)^2}{ \PP( Z_i > u, Z_j > u )}=0.\]
 Note that \eqref{eq:joint-survival-tail-equiv} together with \eqref{eq:residual-tail-dep-coeff} entails \eqref{eq:eta-viaLambda} with $\widetilde{\eta}_{ij}=\eta_{ij}$. 
 
 Conversely, observe that \eqref{eq:joint-Lambda-bounds-by-PP} can be rearranged to
 \begin{align*}
   \Lambda(y_i>u, y_j>u) - \Lambda(u)^2 \leq   \PP( Z_i > u, Z_j > u ) \leq  \Lambda(y_i>u, y_j>u) 
   + 2\Lambda(u)^2\end{align*}
   so that \eqref{eq:joint-survival-tail-equiv} also follows from \eqref{eq:eta-viaLambda}, as \eqref{eq:eta-viaLambda} entails
     \[\lim_{u \to \infty} \frac{\Lambda(u)^2}{  \Lambda(y_i>u, y_j>u) }=0.\]
     Finally, note that \eqref{eq:joint-survival-tail-equiv} together with  \eqref{eq:eta-viaLambda} gives \eqref{eq:residual-tail-dep-coeff} with $\eta_{ij}=\widetilde{\eta}_{ij}$.
\end{proof}

\subsection{Proof of Proposition~\ref{prop:DG}: An asymptotically independent example based on Gaussian distributions}\label{app:GaussianExample}

We provide an explicit construction of an exponent measure $\Lambda$ in dimension~3, which satisfies the conditional independence $\{1\} \indep \{3\} \mid \{2\}\; [\Lambda]$, but exhibits otherwise asymptotic independence. It is based on two bivariate Gaussian distributions and we will need some auxiliary results in order to derive and quantify the asymptotic independence in the sense of \citet{LedTawn96}.

Let $\Phi$ and $\varphi$ be the cdf and pdf of the univariate standard Gaussian distribution, and let $\overline \Phi(x) = 1 - \Phi(x)=\Phi(-x)$ be the corresponding survival function. 
It is well-known \cite[Eq.~7.8.2]{NIST:DLMF} that 
\begin{align}\label{eq:millsratioPhi}
\frac{2}{x+\sqrt{x^2+4}} < \frac{\overline \Phi(x)}{\varphi(x)} < \frac{2}{x+\sqrt{x^2+8/\pi}} < \frac{1}{x}, \qquad x > 0.
\end{align}
Using \eqref{eq:millsratioPhi}, one may easily derive the following result by showing positivity of the respective derivatives \citep[e.g.,][Lemma~5.4]{res2008}.

\begin{lemma} \label{lemma:PhiphiMonoton}
The functions $\Phi/\varphi$ and $\varphi/\overline \Phi$ are strictly monotonously increasing on its entire domain $\RR=(-\infty,\infty)$ from $0$ to $\infty$.
\end{lemma}
%%\begin{proof} Let $H=\Phi/\varphi$. Then $H'(x)=(\varphi(x) + x \Phi(x))/\varphi(x)$ is obviously positive for $x\geq 0$. For $x<0$ set $y=-x>0$ and note that $\varphi(x) + x \Phi(x) = \varphi(y) - y \overline \Phi(y)$, which is positive by \eqref{eq:millsratioPhi}. It is easily seen that $H(x)=\Phi(x)/\varphi(x)=\overline \Phi(-x)/\varphi(-x)$ converges to 0 for $x \to -\infty$ and diverges to $\infty$ for $x \to \infty$. This shows the claims for $H$. The statements for $\varphi(x)/\overline \Phi(x) = \varphi(-x)/\Phi(-x)=1/H(-x) $ follow accordingly.
%%\end{proof}

As any univariate distribution, $\Phi$ is max-infinitely divisible. Its exponent measure $\Lambda_1$ is supported on $\RR=(-\infty,\infty)$ with survival function 
\[\Lambda(u)=\Lambda_1((u,\infty)) = -\log \Phi(u), \qquad u \in \RR \]
and density
\[\lambda_1(u)=\frac{\partial}{\partial u}\log \Phi(u) = \frac{\varphi(u)}{\Phi(u)}, \qquad u \in \RR.\]
It is immediate that the survival functions of $\Lambda_1$ and $\Phi$ are asymptotically equivalent in the sense that
\begin{align}\label{eq:LambdaPhi}
\lim_{u \to \infty}  \frac{\Lambda(u)}{\overline{\Phi}(u)} = 1.
\end{align}
Let $\Phi_\rho$ and $\varphi_\rho$ be the cdf and pdf of the bivariate normal distribution with standard Gaussian margins and correlation $\rho$. Some useful identities, which we will use in several arguments below, are
\begin{align*}
\frac{\partial}{\partial x_1} \Phi_\rho (x_1,x_2) &= \Phi\bigg(\frac{x_2 - \rho x_1}{\sqrt{1-\rho^2}}\bigg) \varphi(x_1),
\qquad
\frac{\partial}{\partial x_2} \Phi_\rho (x_1,x_2) = \Phi\bigg(\frac{x_1 - \rho x_2}{\sqrt{1-\rho^2}}\bigg) \varphi(x_2) ,\\
\sqrt{1-\rho^2} \, \varphi_{\rho}(x_1,x_2) &= \varphi\bigg(\frac{x_2 - \rho x_1}{\sqrt{1-\rho^2}}\bigg) \varphi(x_1) =\varphi\bigg(\frac{x_1 - \rho x_2}{\sqrt{1-\rho^2}}\bigg) \varphi(x_2). 
\end{align*}

\begin{lemma}\label{lemma:PhiRhoBound}
For $x_1,x_2 \in \RR$ and $\rho \in (0,1)$
\[
\max \bigg[  \Phi(x_1) \Phi\bigg(\frac{x_2 - \rho x_1}{\sqrt{1-\rho^2}}\bigg),
\Phi(x_2) \Phi\bigg(\frac{x_1 - \rho x_2}{\sqrt{1-\rho^2}}\bigg)
\bigg]
\leq \Phi_\rho (x_1,x_2)
\]
and
\[
\Phi_\rho (x_1,x_2)
\leq \min \bigg[  \Phi(x_1) \Phi\bigg(\frac{x_2 - \rho x_1}{\sqrt{1-\rho^2}}\bigg) + \Phi(x_2),
\Phi(x_2) \Phi\bigg(\frac{x_1 - \rho x_2}{\sqrt{1-\rho^2}}\bigg) + \Phi(x_1)
\bigg].
\]
\end{lemma}
\begin{proof}
Consider
\[
H_\rho(x_1,x_2) = \Phi_\rho (x_1,x_2) - \Phi(x_1) \Phi\bigg(\frac{x_2 - \rho x_1}{\sqrt{1-\rho^2}}\bigg).
\]
It is easily seen that $H_\rho(x_1,x_2)$ converges to $0$ for $x_1 \to -\infty$ and to $\Phi(x_2)$ for $x_1 \to \infty$. Moreover,
\begin{align*}
\frac{\partial}{\partial x_1} H_\rho(x_1,x_2) 
%%%%%&= \Phi\bigg(\frac{x_2 - \rho x_1}{\sqrt{1-\rho^2}}\bigg) \varphi(x_1) - \bigg[\Phi\bigg(\frac{x_2 - \rho x_1}{\sqrt{1-\rho^2}}\bigg) \varphi(x_1) - 
%%\frac{\rho }{\sqrt{1-\rho^2}} \varphi\bigg(\frac{x_2 - \rho x_1}{\sqrt{1-\rho^2}}\bigg) \Phi(x_1)\bigg]\\
&=   
\frac{\rho }{\sqrt{1-\rho^2}} \varphi\bigg(\frac{x_2 - \rho x_1}{\sqrt{1-\rho^2}}\bigg) \Phi(x_1)
\end{align*}
is positive. Hence $H_\rho(x_1,x_2)$ is strictly increasing in $x_1$ and ranges between $0$ and $\Phi(x_2)$. Exchanging the roles of $x_1$ and $x_2$ establishes the assertion.
\end{proof}

We shall abbreviate the corresponding joint survival function of $\Phi_\rho$ by
\[
S_\rho(x_1,x_2) = \PP(X_1 > x_1, X_2 > x_2),
\]
where $(X_1,X_2)$ is distributed according to $\Phi_\rho$.
By an argument of Balkema \citep[Section~5.2]{res2008}, the bivariate Gaussian distribution $\Phi_\rho$ is max-infinitely divisible if and only if $\rho \geq 0$. If $\rho \in (0,1)$, the associated exponent measure $\Lambda^{(\rho)}$ has a positive continuous density on $\RR^2=(-\infty,\infty)^2$ as given by
%%\[
%%\lambda^{(\rho)}(x_1,x_2)
%%=\frac{\partial}{\partial x_1} \frac{\partial}{\partial x_2} \log \Phi_\rho (x_1,x_2)
%%= \frac{\varphi_\rho(x_1,x_2)}{\Phi_\rho(x_1,x_2)}
%%- \frac{\varphi(x_1)\varphi(x_2) \Phi\Big( \frac{x_1-\rho x_2}{\sqrt{1-\rho^2}}\Big) \Phi\Big( \frac{x_2-\rho x_1}{\sqrt{1-\rho^2}}\Big)}{\Phi^2_\rho(x_1,x_2)}.
%%\]
%%We can rewrite $\lambda^{(\rho)}$ as
\[
\lambda^{(\rho)}(x_1,x_2)
=\frac{\partial}{\partial x_1} \frac{\partial}{\partial x_2} \log \Phi_\rho (x_1,x_2)
=\frac{\varphi_\rho(x_1,x_2)}{\Phi_\rho(x_1,x_2)} \kappa^{(\rho)}(x_1,x_2),
\]
where 
\begin{align*}
\kappa^{(\rho)}(x_1,x_2)
=1-
\frac{\varphi(x_1)\varphi(x_2) \Phi\Big( \frac{x_1-\rho x_2}{\sqrt{1-\rho^2}}\Big) \Phi\Big( \frac{x_2-\rho x_1}{\sqrt{1-\rho^2}}\Big)}{\varphi_\rho(x_1,x_2)\Phi_\rho(x_1,x_2)}.
\end{align*}
Evidently, $\kappa^{(\rho)} \leq 1$. By the max-infinitely divisibility property of $\Phi_\rho$, it is clear that $\kappa^{(\rho)}\geq 0$.
It can be further uniformly bounded as follows.

\begin{lemma} 
\label{lemma:kappaBound}
Let $\rho \in (0,1)$.
For $\max(x_1,x_2)\geq 0$ we have
\begin{align}\label{eq:kappa-uniform-below-A}
0 < 1 - \sqrt{1-\rho^2} \leq \kappa^{(\rho)}(x_1,x_2).
\end{align}
Moreover, there exists a constant $C_\rho > 0$ such that for $u>0$
\begin{align}\label{eq:kappa-uniform-below-B}
1 - C_\rho u^{-1} \leq \inf_{x_1,x_2 \geq u} \kappa^{(\rho)}(x_1,x_2).
\end{align}

\end{lemma}

\begin{proof}
First note that for all $x_1,x_2 \in \RR$
\[
\frac{\varphi(x_1)\varphi(x_2)}{\varphi_\rho(x_1,x_2)}
= \frac{\sqrt{1-\rho^2}\varphi(x_1)}{\varphi\Big( \frac{x_1-\rho x_2}{\sqrt{1-\rho^2}}\Big)}
\] 
and by Lemma~\ref{lemma:PhiRhoBound}
\[
\frac{\Phi\Big( \frac{x_1-\rho x_2}{\sqrt{1-\rho^2}}\Big) \Phi\Big( \frac{x_2-\rho x_1}{\sqrt{1-\rho^2}}\Big)}{\Phi_\rho(x_1,x_2)}
\leq \frac{1}{\Phi(x_1)} \Phi\bigg( \frac{x_1-\rho x_2}{\sqrt{1-\rho^2}}\bigg), 
\]
which gives (by repeating this argument with reversed roles for $x_1$ and $x_2$)
\begin{align}
\label{eq:kappaImportantBound}
1-\kappa^{(\rho)}(x_1,x_2)
&\leq 
\sqrt{1-\rho^2}\,
\min\bigg[
\frac{\varphi(x_1)}{\Phi(x_1)} 
\frac{\Phi\Big( \frac{x_1-\rho x_2}{\sqrt{1-\rho^2}}\Big)}{\varphi\Big( \frac{x_1-\rho x_2}{\sqrt{1-\rho^2}}\Big)} ,
\frac{\varphi(x_2)}{\Phi(x_2)} 
\frac{\Phi\Big( \frac{x_2-\rho x_1}{\sqrt{1-\rho^2}}\Big)}{\varphi\Big( \frac{x_2-\rho x_1}{\sqrt{1-\rho^2}}\Big)}
\bigg]\\
\notag
&= \sqrt{1-\rho^2}\,
\min\bigg[
{H\bigg( \frac{x_1-\rho x_2}{\sqrt{1-\rho^2}}\bigg)}/{H(x_1)} ,
{H\bigg( \frac{x_2-\rho x_1}{\sqrt{1-\rho^2}}\bigg)}/{H(x_2)}
\bigg],
\end{align}
where $H=\Phi/\varphi$.

In order to establish \eqref{eq:kappa-uniform-below-A},
it suffices to show that the last minimum is not larger than 1 for $x_1\geq 0$. We distinguish three cases.\\
\textbf{Case 1:} $0\leq x_1\leq x_2$. First note that $\rho \in (0,1)$ implies $q=(1-\sqrt{1-\rho^2})/\rho \in (0,1)$ and hence $q x_1 \leq x_2$, which implies 
\[
 \frac{x_1-\rho x_2}{\sqrt{1-\rho^2}} \leq x_1.
\]
By the monotonicity of $H$ (see Lemma~\ref{lemma:PhiphiMonoton}) the assertion follows.\\
\textbf{Case 2a:} $0\leq x_2\leq x_1$. We can repeat the argument of Case 1 with reversed roles for $x_1$ and $x_2$.\\
\textbf{Case 2b:} $x_2\leq x_1$, $x_1 \geq 0$, $x_2=-y <0$. In this situation, we can rewrite
\[
H\bigg( \frac{x_2-\rho x_1}{\sqrt{1-\rho^2}}\bigg)/ H(x_2)
= \widetilde H(y) / 
\widetilde H
\bigg( \frac{y+\rho x_1}{\sqrt{1-\rho^2}}\bigg),
\]
where $\widetilde H(x)=1/H(-x)=\varphi(x)/\overline \Phi(x)$ is also monotonously increasing (see Lemma~\ref{lemma:PhiphiMonoton}). Since $y \leq (y+\rho x_1)/\sqrt{1-\rho^2}$, we can bound the right-hand side by 1 and the assertion follows.
Taken together, this shows \eqref{eq:kappa-uniform-below-A}.

In order to establish the second assertion \eqref{eq:kappa-uniform-below-B}, without loss of generality, assume $x_1 \geq x_2 \geq u > 0$ and recall from \eqref{eq:kappaImportantBound} that
\begin{align}\label{eq:kappaUbound}
\frac{\Phi(0)}{\sqrt{1-\rho^2}}
\big(1-\kappa^{(\rho)}(x_1,x_2)\big)
&\leq 
\frac{\Phi(0)}{\Phi(x_2)}
\varphi(x_2) 
\frac{\Phi\Big( \frac{x_2-\rho x_1}{\sqrt{1-\rho^2}}\Big)}{\varphi\Big( \frac{x_2-\rho x_1}{\sqrt{1-\rho^2}}\Big)}
\leq 
{\varphi(x_2)} 
\frac{\Phi\Big( \frac{x_2-\rho x_1}{\sqrt{1-\rho^2}}\Big)}{\varphi\Big( \frac{x_2-\rho x_1}{\sqrt{1-\rho^2}}\Big)}.
\end{align}
We distinguish three cases according to the value of $q=x_2/x_1 \in (0,1]$:\\
\textbf{Case 1:} $q \in (0, \rho/2]$. Then $-(x_2-\rho x_1) = (\rho-q) x_1 > \rho/2 \cdot u$ and the right-hand side of \eqref{eq:kappaUbound} is bounded by
${2 \varphi(0) \sqrt{1-\rho^2}}/{\rho} \cdot u^{-1}$ according to \eqref{eq:millsratioPhi}.\\
\textbf{Case 2:} $q \in (\rho/2,\rho)$. Then $x_2-\rho x_1 = - (\rho-q) x_1 < 0$ and $x_2=qx_1 > \rho/2 \cdot u$.  With Lemma~\ref{lemma:PhiphiMonoton}, we obtain that the right-hand side of \eqref{eq:kappaUbound} is bounded by $ \Phi(0)/\varphi(0) \varphi(\rho/2 \cdot u)$.\\
\textbf{Case 3:} $q \in [\rho,1]$. In this situation $x_2-\rho x_1=(q-\rho)x_1$ is positive. We also recall
$\rho/(1-\sqrt{1-\rho^2}) > 1 \geq q$, which implies 
\[q > \frac{q-\rho}{\sqrt{1-\rho^2}} > 0.\]
Hence, we can bound the right-hand side of \eqref{eq:kappaUbound} by $\Phi(0)$ multiplied with
\[
\sup_{q \in [\rho,1]}
\frac{\varphi(q x_1)}{\varphi\Big( \frac{(q-\rho) x_1}{\sqrt{1-\rho^2}}\Big)}
\leq
 \exp\bigg\{-\frac{1}{2}
 \inf_{q \in [\rho,1]} \Big[ q^2 - \Big( \frac{q-\rho}{\sqrt{1-\rho^2}}\Big)^2\Big] x_1^2 
\bigg\} \leq  \exp(- \rho^2 u^2/2).
\]
Taken together, all three cases establish also the assertion \eqref{eq:kappa-uniform-below-B}.
\end{proof}

\begin{lemma} \label{lem:LambdaPhiRho}
Let $\rho \in (0,1)$. The joint survival functions $\Lambda^{(\rho)}$ and $S_\rho$ are asymptotically equivalent in the sense that
\[
\lim_{u \to \infty} \frac{\Lambda^{(\rho)}( (u,\infty)^2 )}{S_\rho(u,u)} = 1.
\]
\end{lemma}
\begin{proof}
By definition, we have
\[
\Lambda^{(\rho)}( (u,\infty)^2 ) = 2 (-\log \Phi(u)) - (-\log \Phi_\rho (u,u))
\]
and 
\[
S_\rho(u,u) = 2 (1- \Phi(u)) - (1 - \Phi_\rho (u,u)).
\]
Applying L'Hopital's rule yields
\begin{align*}
\lim_{u \to \infty} \frac{\Lambda^{(\rho)}( (u,\infty)^2 )}{S_\rho(u,u)} 
& = \lim_{u \to \infty} \frac{2 {\varphi(u)}/{\Phi(u)} - {2 \varphi(u) \Phi(s u)}/{\Phi_\rho(u,u)}}{2 \varphi(u) - 2 \varphi(u) \Phi(s u )}\\
& = \lim_{u \to \infty} \frac{1}{\Phi_\rho(u,u) \Phi(u)} \cdot \lim_{u \to \infty}
\frac{\Phi_\rho(u,u) - \Phi(s u)\Phi(u)}{1 - \Phi(su)},
\end{align*}
where $s = (1-\rho)/\sqrt{1-\rho^2} = \sqrt{(1-\rho)/(1+\rho)} \in (0,1)$.
As the involved cdfs converge to 1, it remains to be seen that the second limit equals 1. Indeed,
 \begin{align*}
 \lim_{u \to \infty}
\frac{\Phi_\rho(u,u) - \Phi(s u)\Phi(u)}{1 - \Phi(su)}
&= \lim_{u \to \infty} \frac{2 \varphi(u) \Phi(su) - s \varphi(su) \Phi(u) - \varphi(u) \Phi(su)}{ - s \varphi(su) }\\
&= \lim_{u \to \infty} \Phi(u) - 
\lim_{u \to \infty}  \frac{\varphi(u)}{   \varphi(su) } \cdot
\lim_{u \to \infty} \frac{\Phi(su)}{  s }
= 1- 0/s =1,
 \end{align*}
 hence the assertion.
\end{proof}

In what follows we will need a few simple facts about regularly varying functions. An eventually positive measurable function that is defined in a neighborhood of $\infty \in \RR$, is \emph{regularly varying} with index $\alpha \in \RR$ (we write $f \in \mathrm{RV}_\alpha$) if $f(st)/f(t) \to s^{-\alpha}$ for $t \to \infty$. The case $\alpha=0$ is a special case, $f$ is then called \emph{slowly varying}. Trivially, $\log(t)$, $1/\log(t)$ or any measurable function that converges to a positive constant is slowly varying. Moreover, $f \in \mathrm{RV}_\alpha$ if and only if $f(t)t^{-\alpha}$ is slowly varying.
Some well-known closure properties are as follows \citep[e.g.,][Appendix~B]{deh2006a}.

\begin{lemma}\label{lemma:RV}
If $f_1 \in \mathrm{RV}_{\alpha_1}$, $f_2 \in \mathrm{RV}_{\alpha_2}$, then $f_1+f_2 \in \mathrm{RV}_{\max(\alpha_1,\alpha_2)}$ and $f_1f_2 \in \mathrm{RV}_{\alpha_1+\alpha_2}$. If in addition $f_2(t) \to \infty$ for $t \to \infty$, then the composition satisfies $f_1 \circ f_2 \in \mathrm{RV}_{\alpha_1\alpha_2}$.
\end{lemma}

We will need to be aware of the following consequence for $F(u)=\Phi(u)$ in this section.

\begin{lemma}\label{lemma:SV}
Let $F$ be a (univariate) montonously increasing continuous cumulative distribution function with upper endpoint $\infty$ and survival function $\overline F=1-F$. Let $\Lambda(u)=-\log F(u)$.
Let $a(u)=\overline F(u)^{-1}$ and $b(u)=\Lambda(u)^{-1}$. Then $f \circ a^{-1} \in \mathrm{RV}_\alpha$ if and only if $f \circ b^{-1} \in \mathrm{RV}_\alpha$.
In particular, $f \circ a^{-1}$ is slowly varying if and only if $f \circ b^{-1}$ is slowly varying.
\end{lemma}

\begin{proof}
Since $b(u)\to \infty$ and $a(u)/b(u) \to 1$ for $u \to \infty$, we have $a \circ b^{-1}(t)/t \to 1$ for $t \to \infty$. This implies $a \circ b^{-1} \in \mathrm{RV}_1$ and $a \circ b^{-1}(t)\to \infty$ for $t \to \infty$. So, if $f \circ a^{-1} \in \mathrm{RV}_\alpha$, so is $f \circ b^{-1} = (f \circ a^{-1}) \circ (a \circ b^{-1})$  by Lemma~\ref{lemma:RV}. The converse follows analogously.
\end{proof}

Coming back to the Gaussian joint survival functions, it is well-known that for $\eta_\rho=(1+\rho)/2$
\begin{align}\label{eq:GaussSurviveEta}
S_\rho(u,u) = \ell_\rho( \overline{\Phi}(u)^{-1}) \overline{\Phi}(u)^{1/\eta_\rho}
\end{align}
for a slowly varying function $\ell_\rho$ as $u \to \infty$ \citep{LedTawn96}. In other words, the distribution $\Phi_\rho$ exhibits \emph{asymptotic independence}; see Section~\ref{sec:AI}.

Due to the asymptotic equivalences from \eqref{eq:LambdaPhi} and Lemma~\ref{lem:LambdaPhiRho} and, in addition, Lemma~\ref{lemma:SV}, this relation carries over to the asymptotically equivalent counterparts, that is,
\begin{align}\label{eq:rhoeta}
\Lambda^{(\rho)}((u,\infty)^2) = \widetilde \ell_\rho( \Lambda(u)^{-1}) \Lambda(u)^{1/\eta_\rho}
\end{align}
for a slowly varying function $\widetilde \ell_\rho$.
In particular,
\begin{align}\label{eq:rhochi}
\chi_\rho=\lim_{u \to \infty}\frac{\Lambda^{(\rho)}((u,\infty)^2)}{\Lambda(u)}=\lim_{u \to \infty}\frac{S_{\rho}(u,u)}{\overline{\Phi}(u)}
=0. 
\end{align}

Moreover, moving from a multivariate Gaussian distribution $F$ on $\RR^d=(-\infty,\infty)^d$ to a log-Gaussian distribution $F \circ \log$ on $(0,\infty)^d$ only requires a (componentwise) log-transformation, and the same applies to the corresponding exponent measures; that is, $\Lambda_1 \circ \log$ is the exponent measure of $\Phi \circ \log$ and $\Lambda^{(\rho)} \circ \log$ is the exponent measure of $\Phi_\rho \circ \log$. Hence, the relations 
\eqref{eq:LambdaPhi}, \eqref{eq:rhoeta} and \eqref{eq:rhochi}
all carry over to the log-Gaussian variant on $(0,\infty)$ and $(0,\infty)^2$ by realizing that one can replace $u$ by $\log(u)$ in all situations.
For notational convenience, we shall remain on the Gaussian scale in what follows, even though the construction could also be carried out on the log-Gaussian scale. 
In essence, this corresponds to replacing $0$ by $-\infty$ as the absorbing state, whilst the framework of conditional independence carries over with only minor modifications; see~also Remark~\ref{rk:innerTransformations}.

Now let $a,b \in (0,1)$ and consider a trivariate exponent measure $\Lambda$ on $\RR^3$ with density
\[
\lambda(x)=\frac{\lambda^{(a)}(x_1,x_2)\lambda^{(b)}(x_2,x_3)}{\lambda_1(x_2)}, \qquad x \in \RR^3,
\]
similarly to the example of Section~\ref{subsec:CI3constr} with $p_{12}=p_{23}=1$. 
This ensures that the marginal measures of $\Lambda$ are given by 
\[\Lambda_{12}=\Lambda^{(a)} \qquad \text{and} \qquad \Lambda_{23}=\Lambda^{(b)},\] 
whilst the transformation of $\Lambda$ to log-Gaussian scale satisfies the conditional independence relation 
\[\{1\} \indep \{3\} \mid \{2\}\; [\Lambda \circ \log].\]

From the above considerations, we know that 
\begin{align}\label{eq:easyChiEta}
\chi_{12}=\chi_a=0, \quad \eta_{12}=\frac{1+a}{2}, \quad  \chi_{23}=\chi_b=0 \quad \text{and} \quad \eta_{23}=\eta_b=\frac{1+b}{2}.
\end{align}
In order to derive the corresponding values for $\chi_{13}$ and $\eta_{13}$, we need to investigate the asymptotic behavior of
\begin{align}\label{eq:Lambda13}
\Lambda_{13}((u,\infty)^2) = \int_{u}^\infty \int_{-\infty}^\infty \int_{u}^\infty \lambda(x) \mathrm{d}x_1 \mathrm{d}x_2 \mathrm{d}x_3
\end{align}
in relation to $\Lambda(u)$ as $u \uparrow \infty$. 

Note that $\Lambda_{13}$ is different from the bivariate exponent measure $\Lambda^{(ab)}$, which would have arisen as marginal exponent measure of $(X_1,X_3)$ of a trivariate Gaussian distribution $(X_1,X_2,X_3)$ with $(X_1,X_2) \sim \Phi_a$, $(X_2,X_3) \sim \Phi_b$ and $X_1 \indepp X_3 \mid X_2$ (it is easily seen that the inverse $Q$ of the correlation matrix $\Sigma$ with entries $\rho_{12}=a$, $\rho_{23}=b$, $\rho_{13}=ab$ satisfies $Q_{13}=0$; an easy way to verify that $\Lambda_{13}\neq \Lambda^{(ab)}$ is by considering the  trivariate zero mean Gaussian distribution with covariance matrix $\Sigma$ and by observing that the value at $0 \in \RR^3$ of a hypothetical density of its exponent measure, if it existed,  deviates from $\lambda(0)$).

However, we bear in mind that
\[
\varphi_{ab}(x_1,x_3) = \int_{-\infty}^\infty
\widetilde \varphi(x) \mathrm{d}x_2 \qquad \text{with} \qquad \widetilde \varphi(x)=\frac{\varphi_{a}(x_1,x_2) \varphi_{b}(x_2,x_3) }{\varphi(x_2)},
\]
and that
\[
\lambda^{(ab)}(x_1,x_3)
=\frac{\varphi_{ab}(x_1,x_3)}{\Phi_{ab}(x_1,x_3)} \kappa^{(ab)}(x_1,x_3)
=  \int_{-\infty}^\infty 
 \widetilde \varphi(x) \, 
  \mathrm{d}x_2 \, \frac{\kappa^{(ab)}(x_1,x_3)}{ \Phi_{ab}(x_1,x_3) }.
\]
The density of $\Lambda_{13}$ is instead given by
\begin{align*}
\lambda_{13}(x_1,x_3) 
&= \int_{-\infty}^\infty  \lambda(x) \mathrm{d}x_2 \\
&= \int_{-\infty}^\infty \frac{\lambda^{(a)}(x_1,x_2)\lambda^{(b)}(x_2,x_3)}{\lambda_1(x_2)} \mathrm{d}x_2 \\
&= \int_{-\infty}^\infty 
 \widetilde \varphi(x) \, 
 \frac{\Phi(x_2)}{\Phi_{a}(x_1,x_2) \Phi_{b}(x_2,x_3) }
 \, {\kappa^{(a)}(x_1,x_2)\kappa^{(b)}(x_2,x_3)} \mathrm{d}x_2.
\end{align*}
Although, $\Lambda_{13}\neq \Lambda^{(ab)}$, these measures do not  deviate too significantly from each other in their joint upper tails. In order to see this, all we need to do is to control the behavior of the additional multiplicative terms in the integrand above that appear in addition to the term $\widetilde \varphi(x)$.
\begin{lemma} \label{lemma:PhiabBound}
There exist $C>0$ such that for $x_1,x_3 \geq 0$
\[\frac{\Phi(x_2)}{\Phi_{a}(x_1,x_2) \Phi_{b}(x_2,x_3) } \leq C \frac{1}{\Phi(x_2)}.\]
\end{lemma}
\begin{proof}
For $x_1,x_3 \geq 0$, we have
\[\frac{\Phi(x_2)}{\Phi_{a}(x_1,x_2) \Phi_{b}(x_2,x_3) } \leq \frac{\Phi(x_2)}{\Phi_{a}(0,x_2) \Phi_{b}(x_2,0) }. \]
If $x_2 \geq 0$ the latter is bounded by
$({\Phi_{a}(0,0) \Phi_{b}(0,0) })^{-1}$.
If $x_2 < 0$ we can apply Lemma~\ref{lemma:PhiRhoBound} to see that we may instead use the bound
\[\frac{\Phi(x_2)}{\Phi(-ax_2/\sqrt{1-a^2}) \Phi(x_2)\Phi(-bx_2/\sqrt{1-b^2}) \Phi(x_2)} 
\leq \frac{1}{\Phi(0)^2 } \frac{1}{\Phi(x_2)}.
\]
Hence, the assertion is valid for $C=1/\min(\Phi(0)^2,\Phi_a(0,0)\Phi_b(0,0))$.
\end{proof}

\begin{lemma} \label{lemma:auxbetadelta}
For $\beta \in (0,1)$ and $\delta>0$
\begin{align*}
\lim_{u \to \infty}
\frac{1}{\overline \Phi(u)^{2 - \delta}}
\int_{u}^\infty \int_{-\infty}^{u^\beta} \int_{u}^{\infty} \frac{\widetilde \varphi(x)}{\Phi(x_2)}\,
\mathrm{d}x_1 \mathrm{d}x_2 \mathrm{d}x_3  = 0.
\end{align*}
\end{lemma}

\begin{proof}
We can rewrite the integrand as
\[
\frac{\widetilde \varphi(x)}{\Phi(x_2)}
= C_{ab} \frac{\varphi(x_2)}{\Phi(x_2)} 
\varphi\bigg(\frac{x_1 - a x_2}{\sqrt{1-a^2}}\bigg)
\varphi\bigg(\frac{x_3 - b x_2}{\sqrt{1-b^2}}\bigg)
\]
for a constant $C_{ab}>0$.
For sufficiently large $u$, and $x_1,x_3 \geq u$, $x_2\leq u^\beta$ it is clear that $x_1-ax_2 \geq u-au^\beta > 0$ and $x_3-bx_2>u-bu^\beta>0$. Therefore, and since $\sqrt{1-\rho^2}<1$ for $\rho \in \{a,b\}$, we obtain together with the left-hand side of \eqref{eq:millsratioPhi}
\[
\frac{\widetilde \varphi(x)}{\Phi(x_2)}
\leq C_{ab} \max(-2x_2,c) 
\varphi({x_1 - a x_2})
\varphi({x_3 - b x_2})
\]
for a constant $c>0$. Moreover
\begin{align*}
\frac{\varphi(x_1-a x_2)}{\varphi(x_1 - a u^\beta)} &= \sqrt{2\pi} \exp(-ax_1(u^\beta-x_2)) \exp\bigg(\frac{1}{2} a^2u^{2\beta}\bigg) \varphi(ax_2)\\
&\leq \sqrt{2\pi}  \exp\bigg(\frac{1}{2} a^2u^{2\beta}\bigg) \varphi(ax_2)
\end{align*}
and analogously for the terms involving $x_3$. Thus,
\[
\frac{\widetilde \varphi(x)}{\Phi(x_2)}
\leq C' 
 \exp\bigg(\frac{1}{2} (a^2+b^2)u^{2\beta}\bigg)
\max(-2x_2,c) 
\varphi\big(\sqrt{a^2+b^2}x_2\big)
\varphi({x_1 - a u^\beta})
\varphi({x_3 - b u^\beta})
\]
for a constant $C'>0$.
Now, the function $\max(-2x_2,c) 
\varphi\big(\sqrt{a^2+b^2}x_2\big)$ is integrable over $\RR=(-\infty,\infty)$ (and not only over $(-\infty,u^\beta)$). So, we can find a constant $C''>0$, such that
\[
\int_{u}^\infty \int_{-\infty}^{u^\beta} \int_{u}^{\infty} \frac{\widetilde \varphi(x)}{\Phi(x_2)}\,
\mathrm{d}x_1 \mathrm{d}x_2 \mathrm{d}x_3 
\leq C'' \exp\bigg(\frac{1}{2} (a^2+b^2)u^{2\beta}\bigg) \overline{\Phi}(u-au^\beta)\overline{\Phi}(u-bu^\beta).
\]
Finally, using \eqref{eq:millsratioPhi} again, we find that the right-hand side of the last expression divided by $\overline \Phi(u)^{2-\delta}$ is (up to a multiplicative constant) asymptotically equivalent to 
\begin{align*}
&p(u) \exp\bigg\{\frac{1}{2} \bigg[ (a^2+b^2)u^{2\beta}-[u-au^\beta]^2-[u-bu^\beta]^2+(2-\delta) u^2\bigg] \bigg\}\\
&= p(u)\exp\bigg(-\frac{1}{2} \big[\delta u^2 - 2(a+b)u^{1+\beta} \big] \bigg) \to 0
\end{align*}
as $u \to \infty$,
where $p(u)$ is absorbing the ratio of polynomial terms arising from the approximation via \eqref{eq:millsratioPhi}.
\end{proof}

We can now give the proof for the main result of this section.

\begin{theorem} \label{thm:Lambda_13}
The measure $\Lambda_{13}$ from \eqref{eq:Lambda13} satisfies
\begin{align*}
\Lambda_{13}((u,\infty)^2) = \ell( \overline\Phi(u)^{-1}) \overline \Phi(u)^{1/\eta_{ab}}.
\end{align*}
for a slowly varying function $\ell$ and $\eta_{ab}=(1+ab)/2$.
\end{theorem}

\begin{proof} %%%[Proof of Theorem~\ref{thm:Lambda_13}]
Let $\beta \in (0,1)$. We consider the following decompositions of integrals 
$\Lambda_{13}((u,\infty)^2)=A_\lambda(u)+B_\lambda(u)$ and $\Lambda^{(ab)}((u,\infty)^2)=A_{\widetilde \varphi}(u)+B_{\widetilde \varphi}(u)$, where
\[
A_f(u)= \int_{u}^\infty \int_{-\infty}^{u^\beta} \int_{u}^{\infty} f(x)~\mathrm{d}x_1\mathrm{d}x_2\mathrm{d}x_3 \quad \text{and} \quad B_f(u)= \int_{u}^\infty \int_{u^\beta}^\infty \int_{u}^{\infty} f(x)~\mathrm{d}x_1\mathrm{d}x_2\mathrm{d}x_3.
\]
Due to Lemma~\ref{lemma:kappaBound} we have
\[
(1-C_a u^{-\beta})(1-C_b u^{-\beta}) B_{\widetilde \varphi}(u)
\leq B_{\lambda}(u) \leq \frac{1}{\Phi_a(u,u^\beta)\Phi_b(u^\beta,u)} B_{\widetilde \varphi}(u)
\]
and therefore $B_\lambda(u)=(1+e(u))B_{\widetilde \varphi}(u)$ with $e(u) \to 0$ for $u \to \infty$. We note that therefore the mapping $u \mapsto 1+e(u)$ is slowly varying, as it converges to a positive constant.\\
Denoting $\eta=\eta_{ab}=(1+ab)/2$, we have
\[
\frac{\Lambda_{13}((u,\infty)^2)}{\overline \Phi(u)^{1/\eta}}
=\frac{A_\lambda(u)}{\overline \Phi(u)^{1/\eta}}
+ (1+e(u))\bigg[ \frac{\Lambda^{(ab)}((u,\infty)^2)}{\overline \Phi(u)^{1/\eta}} - \frac{A_{\widetilde \varphi}(u)}{\overline \Phi(u)^{1/\eta}}  \bigg]
\]
and we need to show that  
\[
\frac{\Lambda_{13}((h^{-1}(v),\infty)^2)}{\overline \Phi(h^{-1}(v))^{1/\eta}} \qquad \text{with} \qquad h(u)=\frac{1}{\overline \Phi(u)}
\]
is slowly varying as a function of $v$. To start with, we recall from the knowledge about bivariate Gaussian survival functions from above that 
\[
\frac{\Lambda^{(ab)}((h^{-1}(v),\infty)^2)}{\overline \Phi(h^{-1}(v))^{1/\eta}} 
\]
is slowly varying, which is a reformulation of~\eqref{eq:GaussSurviveEta} and \eqref{eq:rhoeta} for $\rho=ab$. 
Moreover, the function $h^{-1}$ is slowly varying, because $\Phi$ is in the Gumbel domain of attraction and $h^{-1}$ is monotone \citep[see][Prop.~0.10 and Exercise 0.4.3.1.]{res2008}, and in addition $h^{-1}(v)\to \infty$ for $v \to \infty$. Hence $(1+e(h^{-1}(v)))$ is slowly varying.\\ 
What remains to be seen to complete the proof, is that adding (or subtracting) the term $A_f(h^{-1}(v))/\overline \Phi(h^{-1}(v))^{1/\eta}$ to (or from) a slowly varying function for $f \in \{\lambda,\widetilde \varphi\}$, preserves slow variation. 
Indeed, with $C>0$ as in Lemma~\ref{lemma:PhiabBound} we have 
\[\max(A_\lambda(u),A_{\widetilde \varphi}(u)) \leq \max(1,C) \int_{u}^\infty \int_{-\infty}^{u^\beta} \int_{u}^{\infty} \frac{\widetilde \varphi(x)}{\Phi(x_2)}\,
\mathrm{d}x_1 \mathrm{d}x_2 \mathrm{d}x_3. \] 
According to Lemma~\ref{lemma:auxbetadelta} and since $\eta>1/2$, there exists (in both cases, $f=\lambda$ or $f=\widetilde \varphi$) some $\alpha>0$ such that
\[
\lim_{v \to \infty}
\frac{A_f(h^{-1}(v))}{\overline \Phi(h^{-1}(v))^{1/\eta}} \cdot v^\alpha  = \lim_{u \to \infty} \frac{A_f(u)}{\overline \Phi(u)^{1/\eta+\alpha}} = 0.
\]
It is easily checked that adding such a function $g$ (with $g(v)\cdot v^{\alpha} \to 0$) to a slowly varying function will not alter the slow variation property. \\
Taken together, we have seen that
\[\Lambda_{13}((u,\infty)^2)= \ell ( \overline \Phi(u)^{-1} )
\overline \Phi(u)^{1/\eta}\]
for a slowly varying function $\ell$.
\end{proof}

As above, we may also work in Theorem~\ref{thm:Lambda_13} with $\Lambda(u)=-\log \Phi(u)$ instead of $\overline \Phi(u)$ (see~\eqref{eq:LambdaPhi} and Lemma~\ref{lemma:SV}), that is,
\[
\Lambda_{13}((u,\infty)^2) = \widetilde \ell( \Lambda(u)^{-1}) \Lambda(u)^{1/\eta_{ab}}
\]
for a slowly varying function $\widetilde \ell$. Stated in this form, the  asymptotic behavior of the joint survival function of $\Lambda_{13}$ resembles \eqref{eq:rhoeta} with correlation $\rho=ab$ except that the slowly varying function $\widetilde \ell$ does not coincide with $\widetilde \ell_{ab}$. Nevertheless, we may conclude
\[
\chi_{13}=\chi_{ab}=0 \quad \text{and} \quad \eta_{13}=\eta_{ab}=\frac{1+ab}{2},
\]
complementing our knowledge from~\eqref{eq:easyChiEta}.
Hence, although the marginal measure $\Lambda_{13}$ does not coincide with the bivariate exponent measure $\Lambda^{(ab)}$, they share the same asymptotic behavior on their joint survival sets $(u,\infty)^2$ as $u \uparrow \infty$. Collectively, this establishes Proposition~\ref{prop:DG}.

\section{Beyond semi-graphoids --- when we have a graphoid}\label{app:graphoids}
Under the explosiveness condition \eqref{eq:infinite}, the conditional independence with respect to $\Lambda$ satisfies the four semi-graphoid properties, (L1)--(L4), cf.~Section~\ref{sec:semi-graphoid} and Theorem~\ref{thm:L4}.
It would even be a \emph{graphoid} if in addition the following property (L5) was satisfied for all disjoint subsets  $A,B,C,D$ of $V$:
\begin{itemize}
  \itemsep2mm
\item[\textbf{(L5)}]  If $A \indep B \,|\, C \cup D \, [\Lambda]$ and  $A \indep C \,|\, B \cup D \, [\Lambda]$, then  $A \indep B \cup C \,|\, D \, [\Lambda]$ \hfill (Intersection)
\end{itemize}

Importantly, the presence of a graphoid-type conditional independence guarantees the equivalence of global, local and pairwise Markov properties, be it directed or undirected; see~\citet[Theorem~3.7 and page~52]{Lauritzen} or \citet{pearlpaz}.
Hence the question if sufficiently rich conditions can be found, under which the $\Lambda$-based conditional independence defined here is in fact a graphoid, and not only a semi-graphoid; see~also the discussion contribution of Steffen Lauritzen in~\citet{eng2018}. It can be easily seen that (L5) is violated in general.

\begin{example}[Violation of (L5) for a measure on a ray] Let $\eta$ be a measure on the positive real line $(0,\infty)$ that explodes at the $0$ and is finite on sets bounded away from 0, e.g., $\eta((t,\infty))=t^{-1}$. Then the measure
\[
\Lambda(E)=
\eta \big( \{ t > 0 \,:\, t (1,1,1) \in E\} \big).
\]
on $\RR^3$, which is concentrated on the ray through $(1,1,1) \in \RR^3$, satisfies our basic requirements \eqref{eq:Lambda} and \eqref{eq:infinite}, and thus (L1)--(L4) hold. Let $Y\sim \PP_R$ for an admissible $R \in \cR(\Lambda)$. Then the law of $(Y_1,Y_2)|Y_3$ is a point mass at $(Y_3,Y_3)$, which is a product measure. Hence, $\{1\} \indep \{2\} \,|\, \{3\} [\Lambda]$, and by symmetry of the argument, $\{1\} \indep \{3\} \,|\, \{2\} [\Lambda]$. If (L5) were true, we would be able to conclude $\{1\} \indep \{2,3\} [\Lambda]$ and thereby $\Lambda(y_1 \neq 0, (y_{2},y_3)\neq (0,0))=0$; see~Prop.~\ref{prop:indep}. This contradicts 
$\Lambda(y_1 \neq 0, (y_{2},y_3)\neq (0,0))=\eta((0,\infty))=\infty$.
\end{example}

In the classical probabilistic case the existence of a positive continuous density with respect to a product measure is a well-known condition to guarantee (L5) \citep[][discussion of (3.10)]{Lauritzen}. This translates as follows to our situation.

\begin{corollary} \label{cor:graphoid-L5-density}
Assume~\eqref{eq:infinite} and let $\mu=\bigotimes_{v \in V} \mu_v$ be a product measure. Assume that each marginal measure $\Lambda_K$, $\emptyset \neq K \subset V$ (including $\Lambda=\Lambda_V$) has a positive continuous density on $\cE^K=\RR^K\setminus \{0_K\}$ with respect to $\mu_K=\bigotimes_{v \in K} \mu_v$, then conditional independence as defined in Definition~\ref{def:CILambda} satisfies all five graphoid properties (L1)--(L5).
\end{corollary}

\begin{proof}
Properties (L1)--(L4) are already established in Theorem~\ref{thm:L4}.
To prove (L5), it is sufficient to assume that the union of $A,B,C,D$ is~$V$; see Lemma~\ref{lemma:marginalcompatibility} and our assumption that each $\Lambda_K$ has a positive continuous density. Then (L5) follows directly from Definition~\ref{def:CILambda} and the respective counterpart of (L5) for probability laws.
\end{proof}

In Corollary~\ref{cor:graphoid-L5-density}
we require each marginal $\Lambda_K$ to have a positive continuous density on every $\cE^K=\RR^{K}\setminus \{0_K\}$. Instead we may consider a one-sided version (in some or all directions) in this requirement and, for instance, ask for $\Lambda_K$ to have a positive continuous density on $\cE^K_+=[0,\infty)^{K}\setminus \{0_K\}$ for every $\emptyset \neq K \subset V$, and the same argument shows that (L5) is then valid. If $\Lambda$ is $-\alpha$-homogenous and  $\Lambda$ has a positive continuous Lebesgue-density on $\cE_+=[0,\infty)^d\setminus \{0\}$, then the marginal
 measures $\Lambda_K$ will also have a positive continuous Lebesgue-density in their respective dimensions, so the assumptions in Corollary~\ref{cor:graphoid-L5-density} are met and  (L5) follows in addition to (L1)--(L4). This is the case for most of the measures considered in \citet{eng2018}, in particular the H\"usler--Reiss exponent measures, which have a positive continuous Lebesgue-density. Hence, they define a graphoid-type conditional independence.

However, we need to be cautious, when drawing conclusions from Corollary~\ref{cor:graphoid-L5-density} in our more general context. 
Take, for example, a measure $\Lambda$ that has a positive continuous $\mu$-density on each of the elementary sub-faces $\cE_K$ in~\eqref{eq:elementary-face} for $\emptyset \neq K \subset V$, for $\mu$ as in \eqref{product_measure}, where each of the factors $\mu_v$ has an atom at 0. In this case, the conditional independence notion induced by $\Lambda$ defines indeed a graphoid, but likely not for the intended reason. In fact, it is sufficient to assume that the measure $\Lambda$, which satisfies our basic explosiveness assumption \eqref{eq:infinite}, has mass on each of the  $\cE_K$ for $\emptyset \neq K \subset V$, in order to conclude that \emph{no} conditional independence statement $A \indep B \,|\, C ~[\Lambda]$ with non-empty $A$ and $B$ can hold; see~Prop.~\ref{prop:indep} and Lemma~\ref{lem:Lambda0}. So there is nothing to check in order to conclude that (L5) is valid.
On the other extreme end, if $\Lambda$ is a zero measure, then \emph{all} conditional independence relations are true, and so the zero measure induces  trivially a graphoid relation, too.

More interestingly, let us turn our attention again to situations, when some elementary sub-faces are charged and others not. We have already considered one such case above, where we have only mass in the upper most elementary sub-face $\cE_V$ and, consequently, it does not matter if we add an atom at 0 in each reference measure $\mu_v$ or not. Indeed, it is reassuring to see that a positive continuous Lebesgue-density for $\Lambda$ and its marginal measures then ensures that $\Lambda$ induces a graphoid-type conditional independence, and that this leads to the H\"usler--Reiss exponent measures from \cite{eng2018} to define a graphoid.
Still, restricting the mass of $\Lambda$ to only $\cE_V$ comes along with the drawback that associated graphical models cannot be disconnected as no independence statement can hold true; see~Section~\ref{sec:CIextremes}. More generally, we would therefore like to allow for mass on several  $\cE_K$ for non-empty $K \subset V$. The following instructive example demonstrates that caution should be exercised in such situations. Even when requiring the existence of a positive continuous density on each of the charged sub-faces $\cE_K$, the property (L5) may be violated.

\begin{example}[Violation of (L5) for a measure with positive continuous density on each of the charged sub-faces] \label{counterexample:graphoid-densities}
Let $A=\{1\}$, $B=\{2\}$, $C=\{3\}$, $D=\{4\}$, $V=\{1,2,3,4\}$ and consider and explosive measure $\Lambda$ on $\RR^4\setminus \{0\}$ with positive continuous $\mu$-density $h_{14}(y_1,y_4)=\lambda(y_1,0_2,0_3,y_4)$ on $\cE_{14}$ and $h_{234}(y_2,y_3,y_4)=\lambda(0_1,y_2,y_3,y_4)$ on $\cE_{234}$ 
 for $\mu$ as in \eqref{product_measure} (with $d=4$)
and zero mass on any other $\cE_K$, $K\subset\{1,2,3,4\}$. Then $\Lambda(y_4=0_4)=0$, and from Table~\ref{tab:graphoid-counter}, which documents all densities that are relevant for the relations that appear in (L5), we can read off that $\{1\} \indep \{2\} \,\vert \, \{3,4\}~[\Lambda]$ and $\{1\} \indep \{3\} \,\vert \, \{2,4\}~[\Lambda]$, but not $\{1\} \indep \{2,3\} \,\vert \, \{4\} ~[\Lambda]$. So (L5) cannot hold.

What is interesting to note in this example is that we have still the proportionalities
\[
\frac{\lambda(y)}{\lambda_{234}(y_{234})}
= \frac{\lambda_{134}(y_{134})}{\lambda_{34}(y_{34})}
= \frac{\lambda_{124}(y_{124})}{\lambda_{24}(y_{24})}
= \begin{cases}
{h_{14}(y_{14})}/{\int h_{14}(y_{14}) dy_1}, &\qquad y \in \cE_{14},\\
0  &\qquad y \in \cE_{1234}.
\end{cases}
\]
and 
\[
\frac{\lambda(y)}{\lambda_{234}(y_{234})}
= \frac{\lambda_{134}(y_{134})}{\lambda_{34}(y_{34})}
= \frac{\lambda_{124}(y_{124})}{\lambda_{24}(y_{24})}
= \begin{cases}
0, &\qquad y \in \cE_{4},\\
1, &\qquad y \in \cE_{234}.
\end{cases}
\]
When fixing $y_{14}$, the ratios still differ according to the $y_{23}$-component -- not their precise value, but it matters if these components are (jointly) zero or not. If one wanted to conclude  $\{1\} \indep \{2,3\} \,\vert \, \{4\} ~[\Lambda]$, one would need an argument that connects these ratios, e.g., in the form of a path through the $y_{23}$-components that can handle the phase transitions between different charged $\cE_K$. 
\end{example}

\begin{sidewaystable}
\centering
\caption{\small Densities from (Counter-)Example~\ref{counterexample:graphoid-densities} on all elementary sub-faces $\cE_K$, cf.~\eqref{eq:elementary-face}, with $y_4 \neq 0$ (i.e., $4 \in K$). Since  $\lambda(y)\lambda_{34}(y_{34})= \lambda_{134}(y_{134})\lambda_{234}(y_{234})$ and $\lambda(y)\lambda_{24}(y_{24})= \lambda_{124}(y_{124})\lambda_{234}(y_{234})$
for $y_4 \neq 0$, and $\Lambda(y_4= 0)=0$, we have $\{1\}\indep \{2\} \,|\, \{3,4\}~[\Lambda]$ and $\{1\}\indep \{3\} \,|\, \{2,4\} ~[\Lambda]$. However, we may not conclude $\{1\}\indep \{2,3\} \,|\, \{4\}~[\Lambda]$, since we do not(!) have $\lambda(y)\lambda_{4}(y_{4}) = \lambda_{14}(y_{134})\lambda_{234}(y_{234})$ for $\mu$-almost all $y$ with $y_4\neq 0$; see~Theorem~\ref{thm:density_factorization}.
}\label{tab:graphoid-counter}
{\small
\renewcommand{\arraystretch}{1.2}
\centering
\begin{tabular}{|l||c|c|c|c|c|c|c|c|}
\hline
\rowcolor{gray!30}
& $\lambda(y)$ & $\lambda_{124}(y_{124})$ & $\lambda_{134}(y_{134})$ & $\lambda_{234}(y_{234})$ & $\lambda_{14}(y_{14})$ & $\lambda_{24}(y_{24})$ & $\lambda_{34}(y_{34})$ & $\lambda_4(y_4)$ \\ 
\hline
$\cE_{4}$ & $0$ & $0$  & $0$ &  $\int h_{14}(y_{14})dy_1$ & $\int h_{234}(y_{234})dy_{23}$ & $\int h_{14}(y_{14})dy_1$ & $\int h_{14}(y_{14})dy_1$ & $*$\\
\hline
$\cE_{14}$ & $h_{14}(y_{14})$ & $h_{14}(y_{14})$ & $h_{14}(y_{14})$ & $\int h_{14}(y_{14})dy_1$ & $h_{14}(y_{14})$ & $\int h_{14}(y_{14})dy_1$ & $\int h_{14}(y_{14})dy_1$ 
& $*$\\
$\cE_{24}$ & $0$  & $\int h_{234}(y_{234})dy_3$ & $0$ & $0$ & $\int h_{234}(y_{234})dy_{23}$ & $\int h_{234}(y_{234})dy_3$ &$\int h_{14}(y_{14})dy_1$ & $*$\\
$\cE_{34}$ & $0$ & $0$  & $\int h_{234}(y_{234})dy_2$ & $0$ & $\int h_{234}(y_{234})dy_{23}$ & $\int h_{14}(y_{14})dy_1$ & $\int h_{234}(y_{234})dy_2$ & $*$\\
\hline
$\cE_{124}$ & $0$ & $0$ & $h_{14}(y_{14})$ & $0$ & $h_{14}(y_{14})$ & $\int h_{234}(y_{234})dy_3$  & $\int h_{14}(y_{14})dy_1$ & $*$\\
$\cE_{134}$ & $0$ & $h_{14}(y_{14})$
& $0$ & $0$ & $h_{14}(y_{14})$ & $\int h_{14}(y_{14})dy_1$ & $\int h_{234}(y_{234})dy_2$ & $*$\\
$\cE_{234}$ & $h_{234}(y_{234})$ & $\int h_{234}(y_{234})dy_3$ &  $\int h_{234}(y_{234})dy_2$
& $h_{234}(y_{234})$  & $\int h_{234}(y_{234})dy_{23}$ & $\int h_{234}(y_{234})dy_3$  & $\int h_{234}(y_{234})dy_2$ & $*$\\

\hline
$\cE_{1234}$ & $0$ & $0$ & $0$ & $h_{234}(y_{234})$ & $h_{14}(y_{14})$ & $\int h_{234}(y_{234})dy_3$ & $\int h_{234}(y_{234})dy_2$ & $*$\\

\hline
\end{tabular}\\[1mm]
$\phantom{a}$ \hfill $*=\int h_{14}(y_{14})dy_1 + \int h_{234}(y_{234})dy_{23}$
}
\end{sidewaystable}

We would like to conclude the discussion with a final insight how to easily construct a measure $\Lambda$ that has mass on several elementary sub-faces $\cE_K$ for non-empty $K\subset V$, so that the induced conditional independence notion defines indeed a graphoid. In fact, one possibility to do so is to use an \emph{independent concatenation} of graphoids.
Any of the situations above, where we know that we have a graphoid, may serve as an independent component in such a construction. That is, each independent component could be, for instance, a measure as in Corollary~\ref{cor:graphoid-L5-density}, a zero measure, a measure with mass on every elementary sub-face, or an independent concatenation thereof.

\begin{proposition}[Independent concatenation of graphoids] 
\label{prop:indepconcat-graphoids}
For $i=1,2$  let $\Lambda_i$ be a measure on $\RR^{V_i}$ satisfying \eqref{eq:Lambda} and \eqref{eq:infinite}, so that the conditional independence for  $\Lambda_i$ defines a graphoid. Consider the measure $\Lambda$ on $\RR^{V_1 \cup V_2}=\RR^{V_1} \times \RR^{V_2}$ such that $\Lambda^0_{V_1}=\Lambda_{V_1}=\Lambda_1$ and $\Lambda^0_{V_2}=\Lambda_{V_2}=\Lambda_2$. Then $\Lambda$ satisfies \eqref{eq:Lambda} and  \eqref{eq:infinite} and
 the $\Lambda$-induced conditional independence defines a graphoid.  
\end{proposition}
 
The proposition follows from the following auxiliary result, which may be of independent interest. Its proof depends only on semi-graphoid properties. We added the proof here, as we did not find the following lemma explicitly in the literature.
 
 \begin{lemma} 
 Assume~\eqref{eq:infinite}, and 
 suppose $(A \cup B \cup C) \indep (A' \cup B' \cup C') ~[\Lambda]$. Then
 \begin{align*}
 (A \cup A') \indep (B \cup B') \,|\, (C \cup C') ~[\Lambda] \quad\iff\quad
 A \indep B \,|\, C ~[\Lambda] \,\,\text{ and }\,\,  A' \indep B' \,|\, C' ~[\Lambda].
 \end{align*}
 \end{lemma}
 \begin{proof}
 Assumption~\eqref{eq:infinite} ensures the validity of the semi-graphoid properties (L1)--(L4), cf.~Theorem~\ref{thm:L4}. Let us abbreviate the involved statements as follows:\\  
 $(A \cup B \cup C) \indep (A' \cup B' \cup C') ~[\Lambda]$ by (Indep), 
 $(A \cup A') \indep (B \cup B') \,|\, (C \cup C') ~[\Lambda]$ by (CI-joint), 
 $A \indep B \,|\, C ~[\Lambda]$ by (CI-a) and   $A' \indep B' \,|\, C' ~[\Lambda]$ by (CI-b). 
 
 We show first that (Indep) and (CI-joint) implies (CI-a) (and by symmetry (CI-b)): Indeed, (CI-joint) implies $A \indep B \,|\, (C \cup C') ~[\Lambda]$ by (L2), whereas (Indep) gives $A \cup C\indep C'  ~[\Lambda]$ by (L2) and, subsequently, $A \indep C' \,|\, C ~[\Lambda]$ by (L3). Taken together this gives 
 $A \indep B \cup C'\,|\, C  ~[\Lambda]$ by (L4), and hence $A \indep B\,|\, C  ~[\Lambda]$ by (L2).
 
 Conversely, let us assume (CI-a) and (CI-b) in the presence of (Indep). Now, (Indep) and (L3) imply 
  $A \indep (A' \cup B' \cup C') \,|\, (B \cup C) ~[\Lambda]$. Together with (CI-a) and (L4) this gives
 $A \indep B \cup (A' \cup B' \cup C') \,|\,  C ~[\Lambda]$. Applying (L3) again gives 
 \begin{align}\label{eq:indep-concat-I}
 A \indep B \cup (A' \cup B') \,|\,  C \cup C' ~[\Lambda].
 \end{align}
 By symmetry of the prerequisites, the analogous argument gives
  \begin{align}\label{eq:indep-concat-II}
 A' \indep (B \cup B') \cup A \,|\,  C \cup C' ~[\Lambda].
 \end{align}
 With \eqref{eq:indep-concat-I} and (L3) and (L1) we get
  \begin{align}\label{eq:indep-concat-III}
 (B \cup B')  \indep A \,|\,  A' \cup (C \cup C') ~[\Lambda],
 \end{align}
 whereas \eqref{eq:indep-concat-II} and (L2) and (L1) imply
   \begin{align}\label{eq:indep-concat-IV}
 (B \cup B') \indep A'  \,|\,  C \cup C' ~[\Lambda].
 \end{align}
 Finally, \eqref{eq:indep-concat-III} and \eqref{eq:indep-concat-IV} and (L4) give the desired (CI-joint).
 \end{proof}

\end{appendix}

\section*{Acknowledgments}
SE thankfully acknowledges funding from an Eccellenza grant of the Swiss National Science Foundation (Grant 186858). JI was supported by a Sapere Aude Starting Grant of the Independent Research Fund Denmark  (Grant 8049-00021B). The authors would also like to thank two anonymous referees for their constructive feedback on this manuscript. In particular, this has led us to investigate the graphoid question in more depth.

\end{document}